\documentclass[a4paper]{amsart}

\usepackage{esint}
\usepackage{xcolor}
\usepackage{amssymb}
\usepackage{mathrsfs}
\usepackage{soul}
\usepackage[pdfencoding=auto,psdextra,hidelinks,pdfusetitle]{hyperref}

\newtheorem{theorem}{Theorem}[section]
\newtheorem{proposition}[theorem]{Proposition}
\newtheorem{lemma}[theorem]{Lemma}
\newtheorem{corollary}[theorem]{Corollary}

\theoremstyle{definition}
\newtheorem{definition}[theorem]{Definition}
\newtheorem{example}[theorem]{Example}

\theoremstyle{remark}
\newtheorem{remark}[theorem]{Remark}

\numberwithin{equation}{section}

\newcommand{\ip}[2]{\left\langle{#1},{#2}\right\rangle}
\newcommand{\laplace}{\boldsymbol{\Delta}}

\DeclareMathOperator{\dist}{d}
\DeclareMathOperator{\supp}{supp}

\DeclareMathOperator{\Lip}{Lip}

\newcommand{\ls}{\leqslant }
\newcommand{\gs}{\geqslant }
\newcommand{\du}{{\rm d}\mu}

\begin{document}

\title[One-phase Free Boundary Problems on RCD Spaces]{One-phase Free Boundary Problems on RCD Metric Measure Spaces}

\author{Chung-Kwong Chan}
\address{Department of Mathematics, Sun Yat-sen University, Guangzhou 510275}
\email{chsongg@mail2.sysu.edu.cn}
\author{Hui-Chun Zhang}
\address{Department of Mathematics, Sun Yat-sen University, Guangzhou 510275}
\email{zhanghc3@mail.sysu.edu.cn}
%\thanks{The second author was partially supported by NSFC  12025109}
\author{Xi-Ping Zhu}
\address{Department of Mathematics, Sun Yat-sen University, Guangzhou 510275}
\email{stszxp@mail.sysu.edu.cn}
%\thanks{The third author was partially supported by NSFC 12271530.}
%\thanks{Support information for the third author.}

%    General info
%\subjclass[2020]{Primary 35R35, 49Q20; Secondary 35J20, 53C23}

%\date{January 1, 2001 and, in revised form, June 22, 2001.}


%\keywords{Alexandrov space, cooperative system, free boundary problem, Lipschitz regularity, nondegeneracy}

\begin{abstract}
  In this paper, we consider a vector-valued one-phase Bernoulli-type free boundary problem on a metric measure space $(X,d,\mu)$ with Riemannian curvature-dimension condition $RCD(K,N)$. We first prove the existence and the local Lipschitz regularity of the solutions, provided that the space $X$ is non-collapsed, i.e., $\mu$ is the $N$-dimensional Hausdorff measure of $X$. And then we show that the free boundary of the solutions is an $(N-1)$-dimensional topological manifold away from a relatively closed subset of Hausdorff dimension  $\ls N-3$. An example is given to show that such a dimension estimate is sharp.\\ 
  
  
    \end{abstract}

\maketitle

\tableofcontents
\setcounter{tocdepth}{1}

\section{Introduction}

Since the  pioneer  work  of  Alt-Caffarelli \cite{AC81},  Dirichlet problems with free boundary  on
Euclidean spaces have been extensively studied.  
Consider  the critical points of the  one-phase Bernoulli energy  functional:
\begin{equation}\label{equ1.1}
      J(u):= \int_{\Omega}(\left|\nabla u\right|^2+\chi_{\left\{u>0\right\}}) \mathrm{d}x\text{,} \\
\end{equation}
where $\Omega\subset \mathbb R^n$ is a bounded open set. The domain $\Omega_u:=\{x\in \Omega\mid u(x)>0\}$ is a priori unspecified and $\partial  \{u>0\}\cap\Omega$  is the {\emph{free boundary}}.   From  \cite{AC81},  the fundamental results about existence and regularity of minimizers of $J$  and  regularity
of free boundaries were established  (see also   \cite{Caf87,Caf89,Caf88} for a view of viscosity solutions). 
 Recently, the studies on free boundary problems have been extended to the   
fully nonlinear uniformly elliptic operators \cite{DSFS15} and
 uniformly elliptic operators with variable coefficients \cite{Tre20a,Tre20b}. 
In the meantime, the vector-valued  Bernoulli-type free boundary problem has been systematically studied by Caffarelli-Shahgholian-Yeressian \cite{CSY18},  Mazzoleni-Terracini-Velichkov \cite{MTV17,MTV20} and Kriventsov-Lin \cite{KL18}. 
We refer the readers to surveys \cite{FS15,F18,DSFS19,CS20} and their references for recent developments of the free boundary problems in the Euclidean settings.

In this paper, we will extend the study on one-phase Bernoulli-type free boundary problems from the  Euclidean setting to the setting of \emph{non-smooth} spaces satisfying a synthetic notion of lower bounds of Ricci curvature. More precisely, letting  $(X,d,\mu)$ be a metric measure space (a complete metric space $(X,d)$ equipped with a Radon measure $\mu$ with ${\rm supp}(\mu)=X$), we assume that it satisfies the Riemannian curvature-dimension condition $RCD(K,N)$ for some $K\in\mathbb R$ and $N\in[1,+\infty)$. The main examples in the class of $RCD(K,N)$ spaces include the Ricci limit spaces in the Cheeger-Colding theory \cite{CC96,CC97,CC00} and finite-dimensional Alexandrov spaces with curvature bounded from below (see \cite{Pet11} and \cite[Appendix A]{ZZ10}).
The parameters $K\in \mathbb R$ and $N\in[ 1,+\infty]$   play the role of  ``Ricci curvature $\gs  K$ and dimension $\ls  N$" in Riemannian geometry.
The theory of  $RCD(K,N)$ metric measure spaces and their geometric analysis have fast and remarkable developments, see \cite{Amb18} for a recent survey on this topic.

Let $(X, d,\mu) $ be  an $RCD(K,N)$ metric measure space for some $K\in\mathbb R$ and  $N\in [1,+\infty)$,  and let   $\Omega\subset X$ be a bounded domain. The Bernoulli-type energy functional is given by 
\begin{equation}\label{equ1.2}
  J_Q({\bf u}) =\int_{\Omega}(\left|\nabla {\bf u}\right|^2+Q\chi_{\left\{\left|{\bf u}\right|>0\right\}}) \mathrm{d}\mu,\ \ \
\end{equation}
where $Q\in L^\infty(\Omega) \ (= L^\infty(\Omega,\mu))$ with
\begin{equation}\label{equ1.3}
0<Q_{\min}\ls  Q(x)\ls  Q_{\max}<+\infty\quad \mu-{\rm a.e.}\ x\in\Omega
\end{equation}
for two positive real numbers $Q_{\min}$ and $Q_{\max}$.  According to \cite{Che99}, it is now known that the Sobolev space $W^{1,2}(\Omega)$ is well-defined. 
Given a boundary data $\mathbf{g}\in W^{1,2}(\Omega,\left[0,+\infty) ^m\right) $, we consider the minimization  problem: 
\begin{equation}\label{equ1.4}
 \min_{{\bf u}\in \mathscr{A}_{\mathbf{g}}}J_Q({\bf u}),\quad \mathscr{A}_{\mathbf{g}}:=\left\{{\bf u}\in W^{1,2}(\Omega,\left[0,+\infty\right) ^{m}) \mid {\bf u}-\mathbf{g}\in W_{0}^{1,2}(\Omega,\mathbb{R}^{m}) \right\}\text{.}
\end{equation}
It is a cooperative vector-valued one-phase Bernoulli-type free boundary problem.
\begin{definition}\label{def1.1}
A map ${\bf u}\in\mathscr{A}_{\mathbf{g}}$ is called a \emph{local minimizer} of  $J_Q$ in \eqref{equ1.2} if there exists some $\varepsilon_{\bf u}>0$ such that $J_Q({\bf u}) \ls  J_Q({\bf v}) $ for every ${\bf v}\in \mathscr{A}_{\mathbf{g}}$ with $\dist({\bf u},{\bf v}) <\varepsilon_{\bf u}$, where
\begin{equation}\label{equ1.5}
  \dist({\bf u},{\bf v}) :=\left\|{\bf u}-{\bf v}\right\|_{W^{1,2}(\Omega,\mathbb{R}^{m}) }+\left\|\chi_{\left\{\left|{\bf u}\right|>0\right\}}-\chi_{\left\{\left|{\bf v}\right|>0\right\}}\right\|_{L^{1}(\Omega) }\text{.}
\end{equation}
If $\varepsilon_{\bf u}=+\infty$, we call that it is a \emph{minimizer} of $J_Q$ in \eqref{equ1.2}.
\end{definition}

The fundamental problems include the existence and regularity of minimizers  (or local minimizers) of $J_Q$ and the regularity of the free boundary $\partial\{|{\bf u}|>0\}\cap \Omega$.

\subsection{The  Bernoulli-type  free boundary problems on  Euclidean spaces}
We first recall some classical results on this problem in the Euclidean setting, i.e., $(X,d,\mu)=(\mathbb R^n,d_{\rm Eucl}, \mathcal L^n).$  
   In the seminal work of  Alt-Caffarell \cite{AC81},
for the scalar case where $m=1$, they established the following: \\
\indent$  \bullet $ the existence of the minimizer of (\ref{equ1.4}), \\
\indent$  \bullet $ Lipschitz continuity of any local  minimizer  $u$, and \\
\indent$  \bullet $ when $Q\in C^\alpha$,    the free boundary $\partial\left\{  u >0\right\}\cap\Omega$ is a $C^{1,\alpha}$-manifold away from a relatively closed subset $\mathcal S_u$ with $\mathscr H^{n-1}(\mathcal S_u)=0$.

Nowadays, it is well-known that the singular set $\mathcal S_u$ has  $\dim_{\mathcal{H}}(\mathcal{S}_{  u})\ls  n-k^{\ast}$ for some $k^{\ast}\in\{5,6,7\}$ (by Weiss \cite{Wei99}, Caffarelli-Jerison-Kenig \cite{CJK04},    De Silva-Jerison \cite{DSJ09}  and Jerison-Savin \cite{JS15}).
Edelen-Engelstein \cite{EE19} explored the rectifiable structure of the singular set  $\mathcal S_u$.

  Recently,  the vector-valued  case where $m\gs  2$, of Bernoulli-type free boundary problem for local  minimizers of $J({\bf u})$  have been systemically studied   by Caffarelli-Shahgholian-Yeressian \cite{CSY18},  Mazzoleni-Terracini-Velichkov \cite{MTV17,MTV20} and  Kriventsov-Lin \cite{KL18}.   See also \cite{Tre20a,Tre20b} for 
 the uniformly elliptic operators with variable coefficients.   
\begin{theorem}[Caffarelli-Shahgholian-Yeressian \cite{CSY18}]\label{thm1.2}
Let $m\gs  2$ and let $\Omega\subset \mathbb R^n$ be a bounded domain. Suppose that $Q\in L^\infty(\Omega)$ satisfies \eqref{equ1.3}.
  For each $g\in W^{1,2}(\Omega, [0,+\infty)^m)$, there exists a minimizer ${\bf u}\in \mathscr A_{\bf g}$ of $J_Q$ in \eqref{equ1.2}. Moreover,
 for any  local  minimizer ${\bf u}=(u_{1},\ldots,u_{m}) $ of $J_Q$,     the following properties  hold:   
  \begin{enumerate}
    \item {\rm (Lipschitz regularity)}\ ${\bf u}$ is locally Lipschitz continuous on $\Omega$.
    \item {\rm (Local finiteness of perimeter and Euler-Lagrange equation)} If $Q\in C(\Omega)$, then the free boundary has locally finite perimeter, and hence the reduced boundary $\partial_{\rm red}\{|{\bf u}|>0\}$ is well-defined. Furthermore, for $\mathscr H^{n-1}$-a.e. point $x\in\Omega\cap \partial_{\rm red}\{|{\bf u}|>0\}$, $i=1,2,\cdots, m$, and $\eta>0$, the non-tangential limit 
    \begin{equation}\label{equ:nontan-limit}
    w_i(x):=\lim_{y\in\{|{\bf u}|>0\}\cap \{-(y-x)\cdot \nu_{\{|{\bf u}|>0\}}(x)\gs\eta\},\ y\to x}\frac{u_i(y)}{|{\bf u}(y)|}
    \end{equation}
   (here $\nu_{\{|{\bf u}|>0\}}(x)$   is the outer normal to $\{|{\bf u}|>0\}$ at $x$) exists, and we have the equations
   \begin{equation}\label{equ:EL}
   \Delta u_i=w_i\sqrt{Q}\mathscr H^{n-1}\llcorner (\Omega\cap \partial_{\rm red}\{|{\bf u}|>0\})\quad {\rm for}\ \ i=1, 2, \cdots, m.
   \end{equation}
    \item {\rm (Regularity of free boundary)} If $Q\in C^{\alpha}(\Omega) $ and ${\bf u}$ is a minimizer of $J_Q$, then the singular part of the free boundary
    $$\mathcal S_{\bf u}:=\big(\partial\{|{\bf u}|>0\}\backslash\partial_{\rm red}\{|{\bf u}|>0\}\big)\cap \Omega$$
    is a closed set in the relative topology of $\Omega$ with $\dim_{\mathcal{H}}(\mathcal{S}_{\bf u})\ls  n-k^{\ast}$ for some $k^{\ast}\in\{5,6,7\}$, and the regular part of the  free boundary    
        $\partial_{\rm red}\left\{|{\bf u}|>0\right\}\cap\Omega $ is locally $C^{1,\beta}$ smooth for some $\beta\in(0,\alpha]$
          ($C^{k+1,\beta}$ smooth or analytic if $Q$ is $C^{k,\alpha}$ smooth or analytic, respectively).
            \end{enumerate}
\end{theorem}
 
There are many other important developments, for example, 
  two-phase free boundary problems \cite{ACF84,DSV2021a,DSV2021b}, free boundary problems for almost minimizer \cite{DT15,DET19},  for the   
fully nonlinear uniformly elliptic operators \cite{DSFS15} and  for the fractional $\alpha$-Laplace operator \cite{CRS10}.

In general, the theory of free boundary problems can be divided into two main steps.
The first step is to establish the existence and the Lipschitz regularity of the solutions.
The second step is to explore the structure of the free boundary of these solutions, 
including the smoothness of its regular part and the size and structure of its singular part.
A blowup argument and an improved flatness property are applied to analyze the structure of the free boundary of solutions.
One may notice that some basic ideas in the theory of free boundary problems are similar to the ones in the theory of minimal surfaces \cite{F69,Giu84} and harmonic maps \cite{Sim96,SU82}.

\subsection{Free boundary problems in $RCD$-spaces and the main results}

In this subsection, we state the main results of this paper. 
Let $(X,d,\mu)$ be  an $RCD(K,N)$-space  with $K\in \mathbb R$ and $N\in(1,+\infty),$ and let $Q$ be a $\mu$-measurable function on $\Omega$ with (\ref{equ1.3}).
Given a   map $\mathbf{g}=(g_1,g_2,\cdots, g_m)\in W^{1,2}(\Omega,\left[0,+\infty\right)^m) $, we consider the minimization  problem   (\ref{equ1.4}). 
 
 The first result is the existence of a minimizer as follows. 
 \begin{proposition}[Existence of a minimizer]\label{prop1.3}
  If ${\rm diam}(\Omega)\ls {\rm diam}(X)/3$, then for each  ${\bf g}\in W^{1,2}(\Omega,[0,+\infty)^m)$, there exists a ${\bf u}\in\mathscr{A}_{\mathbf{g}}$ such that
  \begin{equation}\label{equ1.8}
    J_Q({\bf u}) =\inf_{{\bf v}\in\mathscr{A}_{\mathbf{g}}}J_Q({\bf v}) \text{.}
  \end{equation}
\end{proposition} 
This proposition is somewhat known to experts. For the completeness, we include a proof in Section \ref{sec-existence}.
  
We then consider the Lipschitz regularity of a local minimizer ${\bf u}$ in (\ref{equ1.4}).  Up to our knowledge, the existing proofs of the Lipschitz regularity of  ${\bf u}$  in the Euclidean setting do not work directly in the setting of $RCD(K,N)$-spaces. In fact, some proofs \cite{AC81,CSY18,Caf87,Caf88,Caf89} make heavy use of the Poisson formula, which is not clear on $RCD(K,N)$-spaces.
Other proofs \cite{DT15,DSS20} rely on the fact that gradients of a harmonic function  are again harmonic, 
which fails even on smooth Riemannian manifolds.
In this paper,  we will overcome this difficulty by using the  Cheng-Yau gradient estimates for harmonic functions and a mean value property (see Lemma \ref{lem5.3}), to obtain the following the Lipschitz continuity, provided that the space is non-collapsed. 
\begin{theorem}[Lipschitz regularity]\label{thm1.4} 
Let $(X,d,\mu)$ be an $RCD(K,N)$-space with $K\in \mathbb R$ and $N\in(1,\infty)$. Assume  that $\mu=\mathscr H^N$, the $N$-dimensional Hausdorff measure on $X$. (I.e., $X$ is non-collapsed.)  Let $\Omega\subset X$ be a bounded domain.  Suppose that ${\bf u}=(u_{1},\ldots,u_{m}) $ is a local minimizer of $J_Q$ in \eqref{equ1.2} and  that $Q$ satisfies \eqref{equ1.3}, then ${\bf u}$ is locally Lipschitz continuous on $\Omega$. Precisely, for any ball $B_R(x)\subset \Omega$, there exists a constant $L$ depending only on $N,K, \Omega, R, Q_{\max}, \varepsilon_{\bf{u}} $ and $\int_{B_R(x)}|{\bf u}|\du$, such that 
  \begin{equation}\label{equ1.9}
 |u_i(y)-u_i(z)|\ls  L\cdot d(y,z),\qquad \forall\ y,z\in B_{R/4}(x),\quad i=1,2,\cdots, m.
  \end{equation}
\end{theorem}

\begin{remark}  
The results for the Lipschitz regularity of energy minimizing harmonic maps from/into/between singular spaces were established in \cite{GS92,KS93,ZZ18,GJZ19}, and two recently important works \cite{MS22+,Gig22+}.
 \end{remark}

Our next result is about the finiteness of the perimeter of the free boundary of a local minimizer. We will also derive the associated Euler-Lagrange equation.  

\begin{theorem}[Local finiteness of perimeter and Euler-Lagrange equation]\label{thm1.6}
  Let $(X,d,\mu)$, $\Omega$ and ${\bf u}$ be as in the above Theorem \ref{thm1.4}. Suppose  $Q\in C(\Omega)$.  Then $\Omega_{\bf u}:=\Omega\cap\{|{\bf u}|>0 \}$ is a set of locally finite perimeter. Moreover, it holds:
    \begin{enumerate}
    \item For all $\Omega'\Subset \Omega$, $\mathscr H^{N-1}\big(\partial\{|{\bf u}|>0\}\cap \Omega'\big)<+\infty$;
    \item There exist nonnegative Borel functions $q_i,\ i=1, 2, \cdots, m,$ such that 
    \begin{equation*} 
    {\bf \Delta } u_i=q_i\cdot\mathscr H^{N-1}\llcorner\big(\partial\{|{\bf u}|>0\}\cap\Omega\big)
    \end{equation*}
 in the sense of distributions (i.e., $-\int_{\Omega}\ip{\nabla u_i}{\nabla \phi}{\rm d}\mu=\int_{\partial\{|{\bf u}|>0\}\cap\Omega} \phi q_i{\rm d}\mathscr H^{N-1}$ for any Lipschitz continuous  $\phi$ with compact support in $\Omega$), and 
 \begin{equation}\label{equ1.10}
 \sum_{i=1}^m q_i^2(x)=Q(x), \quad {\rm for}\  \mathscr H^{N-1}{\rm-a.e.} \ x\in \partial\{|{\bf u}|>0\}\cap \Omega.
         \end{equation}  
         \end{enumerate}
\end{theorem}

\begin{remark}
Recalling in the Euclidean setting, by using the non-tangential limits $w_i, i=1,\cdots, m,$ in   \eqref{equ:nontan-limit},    the densities in the Euler-Lagrange equation \eqref{equ:EL}, $q_i=w_i\sqrt Q$, fulfil (\ref{equ1.10}).
  The proof of the existence of non-tangential limits $w_i$  relies heavily on a 
domain variation formula via a $C^1$ vector field (see \cite[Lemma 11]{CSY18}). In the $RCD$ setting, the notion of ``non-tangential limit" is not well-defined at present. In this paper, we will prove (\ref{equ1.10}) by applying a blow-up argument and the theory of sets of finite perimeter in the setting of \cite{Amb01,Mir03,ABS19} (see Corollary \ref{cor7.2} for the details).
\end{remark}

To consider the regularity of free boundary of the local minimizers of $J_Q$, let us 
recall that on the Euclidean space $\mathbb R^n$,  the dimension of singular part  $\dim_{\mathcal{H}}(\mathcal{S}_{\bf u})\ls  n-k^{\ast}$ for some $k^{\ast}\in\{5,6,7\}$ (see Theorem \ref{thm1.2}(3)).  However, in the non-smooth setting, the singularities of the free boundary may arise from the singularities of the space itself, see the following example.

\begin{example}\label{exam1.8}
Let $Y$ be the doubling of an equilateral triangle in $\mathbb R^2$ (gluing two same equilateral triangles along their boundaries). This is a two-dimensional  Alexandrov space with nonnegative curvature, and thus, $(Y,d_Y,\mathscr H^{2})$ is an $ncRCD(0,2)$ metric measure space. Let  $X:=\mathbb R\times Y$. It is clear that $X$ is an  $ncRCD(0,3)$-space. Assume that $\Omega=(-1,1)\times Y$ and 
$$f(t,y):=
\begin{cases}
t& {\rm if} \quad t\gs0\\
0& {\rm if} \quad t<0.
\end{cases}
$$
It is easy to check that $f$ is a minimizer of $J_{Q=1}$ on $\Omega$.
The free boundary $\partial \{f>0\}=\{0\}\times Y$. It is clear that, assuming that  $y\in Y$ is one of the vertices of the equilateral triangle,   the point $x:=(0,y)$ is a singular point of the free boundary. 
\end{example}
 
This example shows that the best expectation of the bound of the singular part of the free boundary in the $RCD(K,N)$-space without boundary is co-dimension $\gs 3$. We shall prove this bound for non-collapsed  $RCD(K,N)$-spaces without boundary. 
 
  Two different notions of boundary of $RCD$-spaces  have been introduced in \cite{DPG18} and \cite{KM19}, respectively. Here we will use the one introduced in \cite{DPG18}. 
Let $(X,d,\mu:=\mathscr H^N)$ be a  non-collapsed $RCD(K,N)$-space.  Recall that \cite{AnBS19,DPG18} the singular part of $X$ has a stratification:
\begin{equation}\label{equ1.11}
\mathcal S^0\subset \mathcal S^1\subset \cdots\subset \mathcal S^{N-1}=\mathcal S:=X\backslash\mathcal R,
\end{equation}
where $\mathcal R$ is the regular part of $X$ given by
$$\mathcal R:=\Big\{x\in X\big|\ {\rm each\ tangent\ cone\ at }\ x \ {\rm is}\ (\mathbb R^N,d_{\rm Eucl})\Big\},$$
and, for any $0\ls k\ls N-1$, 
$$\mathcal S^k:=\Big\{ x\in X\big| \ {\rm no\ tangent\ cone\ at}\ x\ {\rm splits\ off}\ \mathbb R^{k+1}\Big\}.$$
It holds 
\begin{equation}\label{equ1.12}
\dim_{\mathscr H}(\mathcal S^k)\ls k,\quad \forall\ k=1,2,\cdots, N-1.
\end{equation}
(This was first given in \cite{CC97} for non-collapsed Ricci limit spaces.) According to \cite{DPG18}, the boundary of $X$ is defined by 
$$\partial X:=\overline{\mathcal S^{N-1}\backslash \mathcal S^{N-2}}.$$

Our last result is to show that the free boundary has a manifold structure away from a subset having co-dimension 3, which is similar to the one in   \cite{MS21} for the boundary of a set minimizing the perimeter in $RCD$-spaces. 
\begin{theorem}[Regularity of free boundary]  \label{thm1.9}
Let $(X,d,\mu)$ and  $\Omega$  be as in the above Theorem \ref{thm1.4}.   Suppose that ${\bf u}=(u_{1},\ldots,u_{m}) $ is a minimizer of $J_Q$ in \eqref{equ1.2} and  that $Q$ satisfies \eqref{equ1.3}.   Assume that $  \Omega  \cap \partial X=\emptyset$ and $Q\in C(\Omega)$. Then for any $\varepsilon>0$, there exists a relatively open set $O_\varepsilon\subset \partial\{|{\bf u}|>0\}\cap\Omega$ satisfying  the following properties:
   \begin{enumerate}
    \item {\rm ($\varepsilon$-Reifenberg flatness)}
  For any $x\in O_\varepsilon$ there exists a radius $r_x>0$ such that  for any ball   $B_r(y)$  with $y\in B_{r_x}(x)\cap\partial\{|{\bf u}|>0\}$ and $r\in(0,r_x)$,  it holds  that $B_r(y)$  is $\varepsilon r$-closed to $B_r(0^N)$ in the pointed measured  Gromov-Hausdorff topology and that  $B_r(y)\cap\partial\{|{\bf u}|>0\}$  is $\varepsilon r$-closed to $B_r(0^{N-1})$ in the Gromov-Hausdorff topology,  where $B_r(0^N)$ is the ball in $\mathbb R^N$ with centered at $0\in \mathbb R^N$ and radius $r$;
 \item {\rm (The smallness of the remainder  of $O_\varepsilon$)}
    \begin{equation}\label{equ1.13}
\dim_{\mathscr H}\big((\partial\{|{\bf u}|>0\}\cap \Omega)\backslash O_\varepsilon\big)\ls N-3.
\end{equation}   
Moreover,  if $N=3 $, then $(\partial\{|{\bf u}|>0\} \cap \Omega)\backslash O_\varepsilon$ is a discrete set of points.
\end{enumerate}

In particular, the  relatively open set $O_\varepsilon $  
is  $C^\alpha$-biH\"older homeomorphic to an $(N-1)$-dimensional manifold, where $\alpha=\alpha(\varepsilon)\in (0,1)$ with $\lim_{\varepsilon\to0}\alpha(\varepsilon)=1$.

  \end{theorem}

\begin{remark}\label{rem1-10}
According to the above Example \ref{exam1.8}, the bound   (\ref{equ1.13}) is sharp. 
\end{remark}
 
As a direct consequence, we have the following result.
\begin{corollary}\label{cor1-11}
Let $(X,d,\mu), \Omega$ and $ {\bf u}=(u_{1},\ldots,u_{m}) $  be as in the above Theorem \ref{thm1.9}.  
Assume that $  \Omega  \cap \partial X=\emptyset$ and $Q\in C(\Omega)$. Let  
$$ \partial\{|{\bf u}|>0\}\cap \Omega:= \mathcal R^{\Omega_{\bf u}}\cup  \mathcal S^{\Omega_{\bf u}},$$
where  $\mathcal R^{\Omega_{\bf u}}$ and $ \mathcal S^{\Omega_{\bf u}} $ are  the regular  and singular  parts of  
 the free boundary $\partial\{|{\bf u}|>0\}\cap \Omega$. (That is,  
 $x\in \mathcal R^{\Omega_{\bf u}}$ means that each  tangent  cone  at  $ x$  is $  (\mathbb R^N,d_{\rm Eucl})$ and each blow-up limit of $\partial\{|{\bf u}|>0\}\cap \Omega$ at $x$ is an $(N-1)$-dimensional affine hyperplane  in $\mathbb R^{N}$. The singular part
$ \mathcal S^{\Omega_{\bf u}}:=(\partial\{|{\bf u}|>0\}\cap \Omega) \setminus \mathcal R^{\Omega_{\bf u}}.$)
 
Then we have $$\dim_{\mathscr H}\big(\mathcal S^{\Omega_{\bf u}}\big)\ls N-3. $$
\end{corollary}

 \begin{remark}\label{rem1-12}

In general, the regular part $\mathcal R^{\Omega_{\bf u}}$ might not form a manifold. 
In fact, it might not be relatively open in the free boundary $\partial\{|{\bf u}|>0\}\cap \Omega$. (In the Euclidean case and if $Q\in C^\alpha$, the regular part is relatively open in the free boundary,  see Theorem \ref{thm1.2}(3)).   This will be seen by the following simple example. 
 Recall that Y. Ostu and T. Shioya in \cite{OS94}  constructed a two-dimensional Alexandrov space without boundary, denoted by $Y_{OS}$, such that the singular set $\mathcal S$ of   $Y_{OS}$  is dense in $Y_{OS}$. Recalling the above Example \ref{exam1.8}, we  replace the space $Y$ in   Example \ref{exam1.8} by $Y_{OS}$. By using the same construction of $f$, we know that for any singular point $y\in Y_{OS}$,  the point $x:=(0,y)$ is a singular point of the free boundary. Thus, the singular part  $\mathcal S^{\Omega_f}$ is dense in  the
 free boundary $\partial \{f>0\}=\{0\}\times Y_{OS}$.
 
   \end{remark}

\begin{remark} (1) The theory of one-phase free boundary problems was used by Caffarelli-Lin \cite{CL08} in the study of the nodal sets of harmonic maps into a singular space with non-positive curvature in the sense of Alexandrov.
 
(2) The one-phase free boundary problem is strongly related to spectral geometry (for example, quantitative stability of Faber-Krahn inequality \cite{BDV15}).  
The Levy-Gromov inequality and the Faber-Krahn inequality have been established on RCD spaces in \cite{CM17-1,CM17-2}. In the future, we will consider their quantitative properties on RCD spaces.

\end{remark}

{\bf Acknowledgments}. The authors thank Dr. Kai-Hsiang Wang and Alejandro Bellati for their very careful reading of this paper.  The second author was partially supported by NSFC  12025109, and the third author was partially supported by NSFC 12271530.

\section{Preliminaries\label{sec:preliminaries}}

Let $(X,d)$ be a complete metric space and $\mu$ be a Radon measure on $X$ with ${\rm supp}(\mu)=X.$ The triple $(X,d,\mu)$ is called a metric measure space.  Given any $p\in X$ and $R>0$, we
denote by $B_R(p)$ the open ball centered at $p$ with radius $R$.

\subsection{$RCD(K,N)$ metric measure spaces and their calculus}  Let $K\in\mathbb R$ and $N\in[1,+\infty]$.  The curvature-dimension condition $CD(K,N)$ for a metric measure space $(X,d,\mu)$  was introduced by Sturm \cite{Stu06a,Stu06b} and Lott-Villani \cite{LV09,LV07}. The $RCD(K,\infty)$-condition  was introduced by Ambrosio-Gigli-Savar\'e in \cite{AGS14b}. The finitely dimensional case, $RCD(K,N)$, was given by  Gigli in  \cite{Gig13,Gig15}.   Erbar-Kuwada-Sturm \cite{EKS15} and Ambrosio-Mondino-Savar\'e \cite{AMS16}   proved that a weak formulation of the Bochner inequality is equivalent to the  (reduced) Riemannian curvature-dimension condition  $RCD(K,N)$.  In \cite[Theorem 1.1]{CM16}, Cavalletti-Milman showed that the condition $RCD(K,N)$ is equivalent to the condition $RCD(K,N)$, if the total measure $\mu(X)<+\infty$.  

We refer the readers to the survey  \cite{Amb18} and its references for the basic facts of the theory of  $RCD(K,N)$ metric measure spaces. Here we only recall some basic properties  \cite{LV09,AGS14b,AGMR15,EKS15} as follows:\\
$\bullet$ $(X,d)$ is a locally compact length space. In particular,  for any $p,q\in X$, there is a shortest curve connecting them;\\
$\bullet$ 
If $N>1$, then the generalized Bishop-Gromov inequality holds. In particular, it implies  a local measure doubling property: for all $0<r_1<r_2<R$, we have
\begin{equation}\label{equ2.1}
 \frac{\mu\big( B_{r_2}(p)\big)}{\mu\big( B_{r_1}(p)\big)}\ls C_{N,K,R}\Big(\frac{r_2}{r_1}\Big)^N 
 \end{equation}
for some constant $C_{N,K,R}>0$ depending only on $N,K$ and $R$;\\
$\bullet$ If $N>1$ and $\Omega\subset X$ is a bounded set, then there exists a constant $C_{N,K,\Omega}>0$ such that
\begin{equation}\label{equ2.2}
\frac{d^+}{dr}\mu(B_r(x)):=\limsup_{\delta\to0^+}\frac{\mu\big(B_{r+\delta}(x)\backslash B_r(x)\big)}{\delta}\ls C_{N,K,\Omega}  
\end{equation}
for all    $x\in\Omega$ and $r\ls1$. In fact, the generalized  Bishop-Gromov inequality implies 
\begin{equation}\label{equ2.3}
\frac{\mu\big(B_{r+\delta}(x)\backslash B_r(x)\big)}{\mu\big(  B_r(x)\big)}\ls \frac{\bar\mu\big(B_{r+\delta} \backslash B_r \big)}{\bar\mu\big(  B_r \big)},
\end{equation}
where $\bar \mu$ is the $N$-dimensional Hausdorff measure on $\mathbb M^N_{K/(N-1)}$, the simply connected space form with constant sectional curvature $K/(N-1)$, and $B_r$ is a geodesic ball of radius $r$ in $\mathbb M^N_{K/(N-1)}$.
It follows  $$  \limsup_{\delta\to0^+} \frac{\bar\mu\left(B_{r+\delta} \backslash B_r \right)}{\delta\cdot \bar\mu(B_r)}\ls \frac{C_{N,K}}{r},\quad \forall r\ls1.$$
 This gives $ \frac{d^+}{dr}\mu(B_r(x)) \ls C_{N,K}\cdot\frac{\mu(B_r(x))}{r}$ for all $x\in X$ and $r\ls 1$.  Thus, by 
 $$\mu(B_r(x))\ls C_{N,K,\Omega}\cdot r,\qquad   \forall  x\in \Omega,\ \ \forall  r\ls 1,$$ (see \cite[Eq.(4.3)]{CC00} or \cite[Corollary 5.5]{KL16}), we conclude \eqref{equ2.2}.

 Several different notions of Sobolev spaces for metric measure spaces have been given in  \cite{Che99,Sha00,AGS13,AGS14a,HK00}.  They are equivalent to each other in the setting of $RCD$-metric measure spaces (see, for example, \cite{AGS14a,AGS13}).  
 Given a  continuous function $f$ on $X$,  the \emph{pointwise Lipschitz constant} (\cite{Che99}) of $f$ at $x$  is defined by
\begin{equation}\label{equ2.4}
\begin{split}
{\rm Lip} f(x):&=\limsup_{y\to x}\frac{|f(y)-f(x)|}{d(x,y)}\\
&=\limsup_{r\to0}\sup_{d(x,y)\ls r}\frac{|f(y)-f(x)|}{r},
\end{split}
\end{equation}
and $ {\rm Lip}f(x)=0$ if $x$ is isolated. It is clear that ${\rm Lip}f$ is    $\mu$-measurable.  Let $\Omega\subset X$ be an open domain and let   $1\ls  p\ls  +\infty$. The $W^{1,p}$-norm of a locally Lipschitz function
$f\in  Lip_\mathrm{loc}(\Omega)$ on $\Omega$ is defined by
\begin{equation*} 
  \left\|f\right\|_{W^{1,p}(\Omega)}=\left\|f\right\|_{L^{p}(\Omega) }+\left\|\Lip f\right\|_{L^{p}(\Omega) }\text{.}
\end{equation*}
The Sobolev space  $W^{1,p}(\Omega)$ is defined by the completion of the set of locally Lipschitz functions $f$ with  $\left\|f\right\|_{W^{1,p}(\Omega)}<+\infty$.
The space $W_0^{1,p}(\Omega) $ is defined by the closure of $Lip_0(\Omega)$ under the $W^{1,p}$-norm, where  $Lip_0(\Omega)$ is
the set of Lipschitz continuous functions on $\Omega$ with compact support in $\Omega$.
We denote $f\in W^{1,p}_\mathrm{loc}(\Omega)$ if $f\in W^{1,p}(\Omega')$ for every open subset $\Omega'\Subset \Omega$, where ``$\Omega'\Subset \Omega$''
means $\Omega'$ is compactly contained in $\Omega$. We itemize some basic properties of Sobolev functions as follows.

\begin{proposition}\label{prop2.1}
Let $1<p<\infty$.
 \begin{enumerate} 
 \item For each $f\in W^{1,p}(\Omega)$, there is a function in $L^p(\Omega)$, denoted by $|\nabla f|$, $($so-called weak upper gradient for $f$, see \cite[Sect.2]{Che99}$,)$ such that $\|f\|_{W^{1,p}(\Omega)}=\|f\|_{L^{p}(\Omega)}+\||\nabla f|\|_{L^{p}(\Omega)}$. Moreover, if $f\in Lip_{\rm loc}(\Omega)$ then $|\nabla f|=\Lip f$ holds $\mu$-a.e.  in  $ \Omega$ $($\cite[Theorem 5.1]{Che99}$).$
  
  \item {\rm (Lower semicontinuity of energy.)}  If $f_j\in W^{1,p}(\Omega)$ and $f_j\to f$ in $L^p(\Omega)$, then $\liminf_{j\to\infty}\||\nabla f_j|\|_{L^p(\Omega)}\gs  \||\nabla f|\|_{L^p(\Omega)}$. 
  
\item If $f,g\in W^{1,p}(\Omega)$ and $f|_A=g|_A$ for some Borel set $A\subset \Omega$, then $|\nabla f|(x)=|\nabla g|(x)$ at $\mu$-a.e. $x\in A$.

\item The $W^{1,2}(\Omega)$ is a Hilbert space, and the inner product  $\ip{\nabla f}{\nabla g}\in L^1(\Omega)$ for $f,g\in W^{1,2}(\Omega)$ can be given by the polarization  $($see \cite{Gig15}$)$:
\begin{equation}\label{equ2.5}
\ip{\nabla f}{\nabla g}=\frac{1}{4}(|\nabla(f+g)|^2-|\nabla(f-g)|^2) .
\end{equation} 
\item {\rm(Poincar\'e inequality, see \cite[Eq. (2.6)]{BB11} or \cite{Raj12}.)} If $\Omega$ is bounded, then there exists a constant $C_{P}>0$ depending only on $p,K,N$ and ${\rm diam}(\Omega)$, such that  for every ball $B_R(x)\subset \Omega$ with $R\ls  {\rm diam}(X)/3$, it holds
\begin{equation*} 
\int_{B_{R}(x)}f^p\ls  C_P\cdot R^p\int_{B_R(x)}|\nabla f|^p,\qquad \forall\ f\in W^{1,p}_0(B_R(x)).
\end{equation*}  
\end{enumerate} 
 
\end{proposition}

The following fact is well-known, and a proof is given here, since we are not able to find a reference.
\begin{proposition}\label{prop2.2}
Let $D,\Omega$ be two open sets with $\overline{D}\subset \Omega$ and $\mu(\partial D)=0$. Let $1<p< \infty$ and $u\in W^{1,p}(\Omega)$. We denote $v\in W^{1,p}_u(\Omega)$ whenever $v-u\in W^{1,p}_0(\Omega)$. 

Assume  $g\in W_u^{1,p}(D)$ and $h\in W^{1,p}_u(\Omega\backslash\overline{D})$. Then the function 
   \begin{equation*}
    f(x):=\begin{cases}
g(x), & x\in D\\
h(x), & x\in \Omega\backslash \overline{D},
\end{cases}
\end{equation*}
 has a representative in $W^{1,p}_u(\Omega)$. 

In particular, 
 if $g\in W^{1,p}_0(D)$, then its zero extension $\bar g$ on $\Omega$ (namely, $\bar g=g $ on $D$ and $\bar g=0$ on $\Omega\backslash D$) is in $W_0^{1,p}(\Omega)$.
\end{proposition}
\begin{proof}
Since $g-u\in W^{1,p}_0(D)$, there are $\hat g_j\in Lip_0(D)$ such that $\hat g_j\to g-u$ in $W^{1,p}(D)$ as $j\to\infty$. Similarly, there are $\hat h_j\in Lip_0(\Omega\backslash \overline{D})$ such that $\hat h_j\to h-u$ in $W^{1,p}(\Omega\backslash\overline{D})$. Consider the  functions
   \begin{equation*}
    \hat f_j(x):=\begin{cases}
\hat g_j(x), & x\in D\\
0, & x\in\partial D\\
\hat h_j(x), & x\in \Omega\backslash \overline{D},
\end{cases}
\end{equation*}
for each $j\in\mathbb N$. Then   $\hat f_j\in Lip_0(\Omega)$ for each $j=1,2,\cdots.$  Since $\mu(\partial D)=0$, we have   ${\rm Lip}\hat f_j={\rm Lip}\hat g_j$ or ${\rm Lip}\hat h_j$ $\mu$-a.e. in $\Omega$, and then
 $$\|\hat f_j-\hat f_k\|_{W^{1,p}(\Omega)}=\|\hat g_j-\hat g_k\|_{W^{1,p}(D)}+\|\hat h_j-\hat h_k\|_{W^{1,p}(\Omega\backslash \overline{D})}$$
 for all $ j,k=1,2,\cdots.$ It follows that 
  $\{\hat f_j\}_{j=1}^\infty$ is a Cauchy sequence under $W^{1,p}(\Omega)$-norm. Let $\hat f\in W^{1,p}_0(\Omega)$ be the $W^{1,p}(\Omega)$-limit of $\{\hat f_j\}_{j=1}^\infty$. Meanwhile, since $\mu(\partial D)=0$, it is clear that $\hat f_j\to (f-u)$ in  $L^p(\Omega)$, as $j\to\infty$. Therefore, $\hat f=f-u$ in $L^p(\Omega)$. It follows that $\hat f$ is a $W_0^{1,p}(\Omega)$-representative of  $f-u$.
  
  The second assertion follows from the first one, by taking $u=0$ and $h=0$.
\end{proof}

\begin{definition}[Distributional Laplacian \cite{Gig15}]\label{def2.3}
  For each function $f\in W_\mathrm{loc}^{1,2}(\Omega) $, its \emph{Laplacian} $\laplace f$ on $\Omega$ is a linear functional acting on  $ Lip_0(\Omega) $ given by
  \begin{equation}\label{equ2.6}
    \laplace f(\phi) =-\int_\Omega\ip{\nabla f}{\nabla \phi}\mathrm{d}\mu
  \end{equation}
  for all $\phi\in  Lip_0(\Omega) $.
  If there is a signed Radon measure $\nu$ such that $\laplace f(\phi)  =\int_\Omega \phi\mathrm{d}\nu$ for all $\phi\in  Lip_0(\Omega) $, then we say that $\laplace f=\nu$ in the sense of distribution.
\end{definition}

In general, the measure-valued Laplacian ${\bf\Delta} f$ for a function $f\in W^{1,2}_{\rm loc}(\Omega)$ may not be absolutely continuous with respect to $\mu$. We consider its Radon-Nikodym decomposition 
$${\bf \Delta}f=({\bf \Delta }f)^{\rm ac}\cdot\mu+({\bf \Delta}f)^{\rm sing}.$$
This Laplacian on $\Omega$ is linear,  and it satisfies the chain rule and the Leibniz rule \cite{Gig15}.
\begin{remark}\label{rem2.4}
(1) \ When $f\in W^{1,2}(\Omega)$, the test function in (\ref{equ2.6}) can be taken any $\phi\in W^{1,2}_0(\Omega)$.

(2) When $\Omega=X$, the inner product (\ref{equ2.5}) provides a Dirichlet form $\mathcal E(f,g):=\int_X\ip{\nabla f}{\nabla g}$ on $X$. This Dirichlet form $(\mathcal E,W^{1,2}(X))$ has an infinitesimal generator $\Delta_{\mathcal E}$ with domain $D(\Delta_{\mathcal E})\subset W^{1,2}(X)$, i.e., for any $f\in D(\Delta_{\mathcal E})$ and any $g\in W^{1,2}(X), $ it holds $(\Delta_{\mathcal E}f,g)_{L^2(X)}=-\mathcal E(f,g)$. In case $f\in D(\Delta_{\mathcal E})$, the measure valued Laplacian ${\bf \Delta}f$ is absolutely continuous with respect to $\mu$ and ${\bf\Delta}f=\Delta_{\mathcal E}f\cdot\mu.$
\end{remark}

We recall the following Laplacian comparison theorem for distance functions.

\begin{theorem}[Laplacian comparison theorem,  \protect{\cite[Corollary 5.15]{Gig15}}]\label{thm2.5} Let $(X,d,\mu)$ be an $RCD(K,N)$ space with $K\in \mathbb R$ and $N>1$, and let $p\in X$. Put $\rho(x):=d(x,p)$, then
  \begin{equation*} 
    \laplace\rho\ls  (N-1) \cot_\kappa(\sqrt{-\kappa} \rho)\cdot \mu
  \end{equation*}
  in the sense of distribution on $X\backslash\left\{p\right\}$,
   where 
    \begin{equation*}
 \kappa=K/(N-1)\quad{\rm and}\quad 
 \cot_\kappa(s)=
  \begin{cases}
  \sqrt\kappa \cot(\sqrt\kappa s)& {\rm if}\quad \kappa>0,\\
  1/s& {\rm if}\quad \kappa=0,\\
  \sqrt{-\kappa}\coth(\sqrt{-\kappa}s)& {\rm if}\quad \kappa<0.
  \end{cases}
  \end{equation*} 
  In particular, if $K<0,\  N>1$, then we have
  \begin{equation}\label{equ2.7}
  {\bf \Delta}\phi_{N,K}(\rho)\ls 0
  \end{equation}
  in the sense of distributions, where 
   \begin{equation*}
    \phi_{N,K}(s) =-\int^1_s\left(\frac{\sinh(\sqrt{-\kappa}t) }{\sqrt{-\kappa}}\right)^{1-N}\mathrm{d}t,\qquad  \kappa=K/(N-1);
  \end{equation*}  
  and  if $\rho\ls  1$ additionally, then we have
    \begin{equation}\label{equ2.8}
    \laplace(\rho^2)=2|\nabla \rho|^2\cdot\mu+2\rho{\bf\Delta}\rho\ls  2(N+C \rho^2) \cdot\mu
  \end{equation}
  in the sense of distributions on $X$,  where $C$ depends only on $N$ and $K$.
\end{theorem}

A function $f\in W^{1,2}_\mathrm{loc}(\Omega) $ is called \emph{subharmonic}  on $\Omega$ if $\laplace f\gs  0$ in the sense of distributions, that is $\int_\Omega\ip{\nabla f}{\nabla \phi}\mathrm{d}\mu\ls  0$ for all $0\ls  \phi\in  Lip_0(\Omega) .$ A function $f\in W^{1,2}_\mathrm{loc}(\Omega) $ is called \emph{harmonic} on $\Omega$ if both $f$ and $-f$ are subharmonic on $\Omega$. From \cite[Theorem 7.17]{Che99}, the maximum principle holds for subharmonic functions. To be precise, if $f\in W^{1,2}(\Omega)$ is a subharmonic function such that $f-g\in W^{1,2}_0(\Omega)$ for some $g\in L^{\infty}(\Omega)\cap W^{1,2}(\Omega)$, then ${\rm esssup}_\Omega f\ls  {\rm esssup}_\Omega g.$

According to \cite[Theorem 7.12]{Che99},   the (relaxed) Dirichlet problem is solvable: for any ball $B_R(z)$ with $B_{3R}(z)\subset \Omega$ and any $f\in W^{1,2}(B_R(z))$, there exists a (unique) solution $f_H$ to ${\bf \Delta}f_H=0$ in the sense of distributions with boundary data $f_H-f\in W^{1,2}_0(B_R(z))$.
The classical Cheng-Yau's estimate \cite{CY75} for harmonic functions has been extended to $RCD$ metric measure spaces \cite{JKY14,ZZ16}.

\begin{theorem}[see \protect{\cite[Theorem 1.6]{ZZ16}}]\label{thm2.6}
  Let $(X,d,\mu)$ be an $RCD(K,N)$-space with  $K\ls 0$ and $N\in (1,+\infty)$. Then every harmonic function $f$ on a geodesic ball $B_{R}(x_0)\subset X$  admits a locally Lipschitz continuous representative. Moreover,  there exists a constant $C_N$, depending only on $N$, such that  every positive harmonic function $f$ on $B_R(x_0)$  satisfies
  \begin{equation}\label{equ2.9}
    \sup_{B_{R/2}(x_0)}\frac{|\nabla f|}{f}\ls  C_N\frac{1+\sqrt{-K}R}{R}.
  \end{equation}
\end{theorem}
Remark that Cheng-Yau's estimate implies the Harnack estimate: if $f$ is a positive harmonic function on $B_{R}(x_0)$ with $R\ls  1$, then
\begin{equation}\label{equ2.10}
\frac{f(x)}{f(y)}\ls  C_{K,N},\quad \ \forall \ x,y\in B_{R/2}(x_0)
\end{equation} 
for some constant $C_{K,N}$ depending only on $K$ and $N$.
  Indeed,   for any two points $x,y\in B_{R/2}(x_0)$, one can connect them by a curve $\gamma(t)\subset B_{R/2}(x_0)$ with length $L(\gamma)\ls  R$. By (\ref{equ2.9}) and $R\ls  1$, it holds 
\begin{equation*}
\begin{split}
 | \ln f(x)-\ln f(y)|&\ls  \int_0^{L(\gamma)}|(\ln f\circ\gamma)'|\ls \int_{0}^{L(\gamma)}(|\nabla \ln f|\circ\gamma)\\
 &\ls  \frac{C_N(1+\sqrt{-K}R)}{R}  L(\gamma)\ls  C_{N,K}.
 \end{split}
 \end{equation*}
 
We recall the notion of the pointed measured Gromov-Hausdorff convergence of a sequence of metric measure spaces (see \cite{Gro07,GMS15}). We focus on $RCD(K,N)$ metric measure spaces.
\begin{definition}\label{pmGH}
(1) Let   $(Z,d_Z)$ be a complete metric space. Given $\epsilon>0$ and subsets $A,B\subset Z$, we say that the Hausdorff distance $d^Z_H(A,B)<\epsilon$ if
 $$A\subset B_\epsilon\quad {\rm and}\quad B\subset A_\epsilon,$$
 where $A_\epsilon$ is the $\epsilon$-neighborhood of $A$ by $A_\epsilon:=\{z\in Z|\ d_Z(z,A)<\epsilon\}$. We denote by $A_j\overset{H}{\longrightarrow}A_\infty$ in $Z$ if $d^Z_H(A_j,A_\infty)\to0$  as $j\to\infty$. 
 
(2) Let $(X_j,d_j)$  be a sequence of compact metric spaces. We say that   $(X_j,d_j)$ converge to a metric space $(X_\infty, d_\infty)$ in the Gromov-Hausdorff topology ($GH$ for short), denoted by $X_j\overset{GH}{\longrightarrow}X_\infty$ for short,   if there exist  a complete metric space $(Z,d
_Z)$ and a sequence of isometric embedding $\Phi_j:X_j\to Z$  such that $\Phi_j(X_j)\overset{H}{\longrightarrow}\Phi_\infty(X_\infty)$ in $Z$.

(3) Let $K\in\mathbb R$, $N\in[1,\infty)$  and let $\{(X_j,d_j,\mu_j)\}_{j\in\mathbb N}$ be a sequence of  $RCD(K,N)$ metric measure spaces with based points $p_j\in X_j$ for all $j\in\mathbb N$.  We say that $(X_j,d_j,\mu_j,p_j)$ converges  to a pointed metric measure space  $(X_\infty,d_\infty,\mu_\infty,p_\infty)$ in the pointed measured Gromov-Hausdorff topology ($pmGH$ for short), denoted by $$(X_j,d_j,\mu_j,p_j) \overset{pmGH}{\longrightarrow} (X_\infty,d_\infty,\mu_\infty,p_\infty),$$ if there exist a  metric space $(Z,d_Z)$ and a sequence of isometric embeddings $\Phi_j:X_j\to Z$, for all $j\in\mathbb N$,  and $\Phi_\infty:X_\infty \to Z$ such that the following hold:

 $\bullet\quad\Phi_j(p_j)\to \Phi_\infty(p_\infty)$ in $Z$, 
 
$\bullet\quad $for every $R>0$, 
$ \Phi_j(B_R(p_j))\overset{H}{\longrightarrow}\Phi_\infty(B_R(p_\infty))$ in   $Z$, and

$\bullet\quad(\Phi_j)_{\#}\mu_j\rightharpoonup(\Phi_\infty)_{\#}\mu_\infty $ as $j\to\infty$ (as the duality on $C_{bs}(Z)$, the space of  continuous functions  on $Z$ with bounded support). 
\end{definition}
This is the so-called extrinsic approach \cite{GMS15,Gro07}, and we will fix the choice of $(Z,d_Z)$ and the embeddings $\Phi_j,\Phi_\infty$ in the rest of this paper. 
It is well-known that the limit metric measure space $(X_\infty,d_\infty,\mu_\infty)$ is also of $RCD(K,N)$, and that \cite{GMS15} the ambient space $(Z,d_Z)$ can be chosen to be proper. Hence, the weak convergence of measures   $(\Phi_j)_{\#}\mu_j\rightharpoonup(\Phi_\infty)_{\#}\mu_\infty $ can be also understood in the duality on $C_{0}(Z)$ (the space of  continuous functions  on $Z$ with compact support). 

  Let $(X_j,d_j,\mu_j,p_j) \overset{pmGH}{\longrightarrow} (X_\infty,d_\infty,\mu_\infty,p_\infty)$, and let $A_j\subset X_j$, $A_\infty\subset X_\infty$ be  Borel subsets. We will denote by $A_j\overset{GH}{\longrightarrow} A$ if $\Phi_j(A_j)\overset{H}{\longrightarrow} \Phi_\infty(A_\infty)$ in   $Z$, where  the ambient space $Z$ and  the embeddings $\Phi_j,\Phi_\infty$ are given in Definition \ref{pmGH}(3).

For any $x\in X$ and $r>0$, we consider  the rescaled pointed metric measure space 
\begin{equation}\label{equ2.11}
(X, r^{-1}d,\mu^x_r,x),\quad {\rm where} \  \ \mu^x_r:=c^x_r\cdot\mu:=\frac{\mu}{\int_{B_r(x)}(1-r^{-1}d(\cdot,x) ) \du}.
\end{equation}

\begin{definition}\label{def-tangent-cone}
A pointed metric measure space $(Y,\rho,\nu,y)$ is called a {\emph{tangent cone}} of $(X,d,\mu)$ at $x$ if there exists $r_j\to0$ such that
$(X,r^{-1}_jd,\mu^x_{r_j},x)\overset{pmGH}{\longrightarrow}(Y,\rho,\nu,y)$.

A point $x$ is called a {\emph{$k$-regular}} if the tangent cone at $x$ is unique and is isomorphic to
\begin{equation}\label{equ-cn}
(\mathbb R^k,d_{E},c_k\mathscr L^k,0^k), \quad {\rm where}\ \ 
 c_k:=(\int_{B_1(0^k)}(1-|x|){\rm d}x)^{-1}.
 \end{equation}
\end{definition}
Remark that, from \cite{MN19,BS19}, now it is known  that there exists a unique integer $k\in[1,N]$ such that $\mu(X\backslash \mathcal R_k)=0$, where $\mathcal R_k$ is the set of all $k$-regular points of $(X,d,\mu)$.

 We also consider the convergence of functions defined on varying spaces.
 Let
 $$(X_j,d_j,\mu_j,p_j)\overset{pmGH}{\longrightarrow} (X_\infty,d_\infty,\mu_\infty,p_\infty).$$
\begin{definition}\label{def2.9}
Let $R>0$.  Suppose that $\{f_j\}_{j\in\mathbb N\cup\{\infty\}}$ is a sequence of Borel functions on $B_R(p_j)$.
It is said that:\\
$(i)$ $f_j\rightarrow f_\infty$ over $B_R(p_j)$, if for  any $x_j\overset{GH}{\longrightarrow}x_\infty\in B_R(p_\infty)$ then
$f_j(x_j)\to f_\infty(x_\infty)$ as $j\to\infty$;\\
$(ii)$
 $f_j\rightarrow f_\infty$ \emph{uniformly} over $B_R(p_j)$, if for any $\varepsilon>0$ there exist $N(\varepsilon)\in\mathbb N$ and $\delta:=\delta(\varepsilon)>0$ such that
$$\sup_{x\in B_R(p_j),\ y\in B_R(p_\infty),\ d_Z(\Phi_j(x),\Phi_\infty(y))<\delta}|f_j(x)-f_\infty(y)|<\varepsilon,\quad \forall\ j\gs N(\varepsilon),$$
where $\Phi_j,\Phi_\infty$ and $Z$ are given in Definition \ref{pmGH}.
\end{definition}

We remark that the Arzela-Ascoli theorem can be generalized to the case where the functions live on varying spaces (see, for example, \cite{LV09}). We recall the following Cheeger's lifting lemma:
\begin{lemma}[{\cite[Lemma 10.7]{Che99}}] \label{lem2.10}
Let $R>0$ and let $(X_j,d_j,\mu_j)$ be a sequence of  $RCD(K,N)$ metric measure spaces and $$(X_j,d_j,\mu_j,p_j)\overset{pmGH}{\longrightarrow} (X_\infty,d_\infty,\mu_\infty,p_\infty).$$
Then,  given any  Lipschitz function $f_\infty\in Lip(\overline{B_R(p_\infty)})$  with a Lipschitz constant $L>0$, there exist a sequence  of  Lipschitz functions $f_j\in Lip(\overline{B_R(p_j)})$ such that
 $f_j\rightarrow f_\infty$ uniformly over $B_R(p_\infty)$, 
$\||\nabla f_j|\|_{L^\infty(B_R(p_j))}\ls L+1$ for all $j\in\mathbb N$, and that
\begin{equation*}
  \lim_{j\to\infty}\int_{B_R(p_j)}|\nabla f_j|^2\du_j= \int_{B_R(p_\infty)}|\nabla f_\infty|^2\du_\infty.
\end{equation*}
\end{lemma}
We also need a variant of it as follows.
\begin{lemma}\label{lem2.11}
Let $R>0$ and let $(X_j,d_j,\mu_j)$ be as above in Lemma \ref{lem2.10}. Let $F_j\in Lip(\overline{B_R(p_j)})$ be a sequence of Lipschitz functions with a uniform Lipschitz constant and satisfying that $F_j\to F_\infty$ uniformly over $B_{R}(p_j)$.  Then,  given any    function $f_\infty\in Lip_{\rm loc}(B_R(p_\infty))$  with $f_\infty-F_\infty \in W^{1,2}_0(B_R(p_\infty))$, for each $\delta>0$,  there exist a sequence  of   functions $f_j $ on $B_R(p_j)$ such that $f_j-F_j\in W^{1,2}_0(B_R(p_j))$, 
 $f_j\rightarrow f_\infty$ uniformly over $B_{R-\delta}(p_j)$,  and that
\begin{equation*}
  \lim_{j\to\infty}\int_{B_R(p_j)}|\nabla f_j|^2\du_j= \int_{B_R(p_\infty)}|\nabla f_\infty|^2\du_\infty.
\end{equation*}
\end{lemma}
\begin{proof}
Fix any $\delta>0$. Since $f_\infty\in Lip(\overline{B_{R-\delta}(p_\infty)})$, we can use the above Cheeger's lemma, Lemma \ref{lem2.10}, to obtain a sequence of  Lipschitz functions $g_j\in Lip(\overline{B_{R-\delta}(p_j)})$ such that 
$g_j\to g_\infty:=f_\infty|_{\overline{B_{R-\delta}(p_\infty)}}$ uniformly over $B_{R-\delta}(p_j)$, $\|\nabla g_j\|_{L^2(B_{R-\delta}(p_j))}\to \|\nabla f_\infty\|_{L^2(B_{R-\delta}(p_\infty))}$ as $j\to\infty$, and $|\nabla g_j|_{L^\infty(B_{R-\delta}(p_j))}\ls L_\delta$, where $L_\delta>0$ is independent of $j$ (may depend on $\delta$).

Denoted by $A_\delta(p_j):=B_R(p_j)\backslash \overline{B_{R-\delta}(p_j)}$ for each $j\in\mathbb N\cup\{+\infty\}.$ Let $G_j  \in Lip(\partial{A_\delta(p_j)})$ be defined by $G_j= F_j$ on $\partial B_R(p_j)$ and $G_j=g_j$ on $\partial B_{R-\delta}(p_j)$. Then, for each $j\in \mathbb N$,  we can extend $G_j$ to a Lipschitz function  $\bar G_j\in Lip(\overline{ A_\delta(p_j)})$.    Since $g_j$ and $F_j$ have a uniform Lipschitz constant, we can assume that $\bar G_j$ have a uniform Lipschitz constant $L'_\delta.$ Since $G_j\to G_\infty$ uniformly on $\partial A_\delta(p_j)$, we can also assume that $\bar G_j\to \bar G_\infty$ uniformly (up to a subsequence, by Arzela-Ascoli theorem). Now we have $f_\infty-\bar G_\infty\in W^{1,2}_0(A_\delta(p_\infty))$. By \cite[Proposition 3.2(ii)]{ZZ19}, there exists a sequence $h_j$ on $A_\delta(p_j)$ such that $h_j-\bar G_j\in W^{1,2}_0(A_\delta(p_j))$ and $|\nabla h_j|_{L^2(A_\delta(p_j))}\to |\nabla f_\infty|_{L^2(A_\delta(p_\infty))}$ as $j\to\infty$.

At last, we define the function $f_j$ on $B_R(p_j)$ by $f_j:=g_j$ on $\overline{B_{R-\delta}(p_j)}$ and $f_j:=h_j$ on $A_\delta(p_j)$. From Proposition \ref{prop2.2}, we conclude that $f_j\in W^{1,2}(B_R(p_j))$ and $f_j-F_j\in W^{1,2}_0(B_R(p_j))$. The functions $f_j$ satisfy the desired assertions. The proof is finished.
\end{proof}

\subsection{Non-collapsed $RCD(K,N)$ metric measure spaces}
We will recall the notion of {\emph{non-collapsed}} $RCD$-space as introduced in \cite{DPG18}. 
\begin{definition}\label{def2.12}
Let  $(X,d,\mu)$ be an  $RCD(K,N)$-space with $K\in\mathbb R$ and $N\in[1,+\infty)$. It is   called  a non-collapsed $RCD(K,N)$-space, denoted by $ncRCD(K,N)$-space  for short, if   $\mu=\mathscr H^N$, the $N$-dimensional Hausdorff measure on $X$.
\end{definition}
The main examples of $ncRCD(K,N)$ metric measure spaces are non-collapsed Ricci limit spaces \cite{CC97,CC00,CJN21} and $N$-dimensional Alexandrov space with curvature $\gs K/(N-1)$.
It was shown \cite{DPG18}   that if $\left\{(X_{\text{i}},d_{\text{i}},\mu_{\text{i}})\right\}$ is a sequence of   $ncRCD(K,N)$ metric measure spaces  and $(X_j,d_j,\mu_j,p_j) \overset{pmGH}{\longrightarrow} (X_\infty,d_\infty,\mu_\infty,p_\infty)$, then $(X_\infty,d_\infty,\mu_\infty)$ is of $ncRCD(K,N)$ too.

  If $(X,d,\mu)$ is an $ncRCD(K,N)$-space, then  $N$ must be an integer, and there holds  (from Corollary 2.14 in \cite{DPG18})
    \begin{equation}\label{equ2.12}
\mu\big(B_r(x)\big)\ls \bar\mu(B_r)\ls C_{N,K}\cdot r^N,\quad \forall \ x\in X\ {\rm and}\ r\ls 1,
\end{equation}
for a constant $C_{N,K}>0$, where $\bar \mu$ is the $N$-dimensional Hausdorff measure on $\mathbb M^N_{K/(N-1)}$, the simply connected space form with constant sectional curvature $K/(N-1)$, and $B_r$ is a geodesic ball of radius $r$ in $\mathbb M^N_{K/(N-1)}$. Furthermore, if $N>1$, 
\begin{equation}\label{equ2.13}
\frac{d^+}{dr}\mu(B_r(x)):=\limsup_{\delta\to0^+}\frac{\mu\big(B_{r+\delta}(x)\backslash B_r(x)\big)}{\delta}\ls C_{N,K}\cdot r^{N-1} 
\end{equation}
for all  $\ r\ls 1$ and $x\in X$. In fact by (\ref{equ2.3}) and
  $\mu(B_r(x))\ls \bar\mu(B_r)$, we get $ \mu\left(B_{r+\delta}(x)\backslash B_r(x)\right)\ls \bar\mu\left(B_{r+\delta} \backslash B_r \right).$ It follows (\ref{equ2.13}). Remark that it holds only (\ref{equ2.2}) in general $RCD(K,N)$-spaces (without the assumption of non-collapsing).

Let $(X,d,\mu:=\mathscr H^N)$ be an $ncRCD(K,N)$-space. For each point $x\in X$, any tangent cone is a metric measure cone (a Euclidean cone with a natural measure). Indeed,  the existence of the limit  $\lim_{r\to0}\frac{\mu(B_r(x))}{\omega_Nr^N}$ implies that any tangent cone at $x$ is a volume cone,  and hence, by \cite{DPG16}, it  is a metric cone.  It was shown that \cite[Corollary 1.7]{DPG18} a point $x\in X$ is regular (i.e., any one tangent cone is isometric to $\mathbb R^N$) if and only if 
$$\lim_{r\to0}\frac{\mu(B_r(x))}{\omega_Nr^N}=1.$$

 \subsection{Sets of finite perimeter and the reduced boundary} The theory of Euclidean sets of finite perimeter of De Giorgi has been extended to $RCD(K,N)$-spaces  \cite{Amb01, Mir03, ABS19}, and recently \cite{BPS19,BPS21}.  

 \begin{definition}\label{def2.13} 
A function $f\in L^1(X,\mu)$ is called a function of bounded variation, denoted by $f\in BV(X)$ for short, if there exists a sequence $f_j\in Lip_{\rm loc}(X)$ converging to $f$ in $L^1(X)$ such that
$$\limsup_{j\to\infty}\int_X|\nabla f_j|\du<+\infty.$$
Its total variation is a finite Borel measure and is denoted by $|D f|$. Moreover, for any open subset $A\subset X$,
$$|D f|(A):=\inf\left\{\liminf_{j\to\infty}\int_\Omega|\nabla f_j|{\rm d}\mu\ \big|\ f_j\in Lip_{\rm loc}(A),\ \ f_j\overset{L^1(A)}{\longrightarrow}f\right\}.$$
 \end{definition}
  A function $f\in BV_{\rm loc}(X)$ if $\phi f\in BV(X)$ for any $\phi\in Lip_{0}(X)$.
 
  \begin{definition}\label{def2.14}
  Let $E\subset X$ be a Borel subset and let $A$ be an open set. The  perimeter $\mathcal P(E,A)$ is given by
  $$\mathcal P(E,A):=\inf\left\{\liminf_{j\to\infty}\int_A|\nabla f_j|{\rm d}\mu\ \big|\ f_j\in Lip_{\rm loc}(A),\ \ f_j\overset{L^1_{\rm loc}(A)}{\longrightarrow}\chi_E\right\}.$$
A Borel set $E$ is called of finite perimeter in $X$ if $\mathcal P(E,X)<\infty.$ In that case, it is proved \cite{Amb01,Mir03} that the set function $A\mapsto \mathcal P(E,A)$ is the restriction to open sets of a finite Borel measure $\mathcal P(E,\cdot)$ defined by 
  $$\mathcal P(E,B):=\inf\{\mathcal P(E,A)|\ B\subset A,\ A\subset X\ \ {\rm open}\}.$$
  \end{definition}

A subset  $E\subset X$ with $\mu(E)<\infty$ is a set of finite perimeter if and only if the characteristic function $\chi_E\in BV(X)$, and its perimeter measure is $\mathcal P(E,\cdot):=|D\chi_E|(\cdot).$  A subset $E\subset X$ is called a set of locally finite perimeter if $\chi_E\in BV_{\rm loc}(X)$.

We recall the following basic properties, collected in \cite{Amb01,Mir03,ABS19}.
\begin{proposition}\label{prop2.15}
Let $(X,d,\mu)$ be an $RCD(K,N)$ metric measure space with $K\in \mathbb R$ and $N\in[1,+\infty)$. Then the followings hold:
\begin{enumerate} 
  \item {\rm (Lower semicontinuity)}  $E\mapsto \mathcal P(E,X)$ is lower semicontinuous with respect to the $L^1_{\rm loc}(X)$ topology; 
    
\item {\rm (Coarea formula)}   Let  $v\in BV(X)$. Then $\{v>r\}:=\{x|\ v(x)>r\}$ has finite perimeter for $\mathscr L^1$-a.e. $r\in\mathbb R$. Moreover, if $v\in BV(X)$  is  continuous and nonnegative, then, for any Borel function $f:X\to [0,+\infty]$, it holds
 \begin{equation}\label{equ2.14}
\int_{s\ls v<t}f{\rm d}|D v|=\int_s^t\int_Xf{\rm d}\Big(\mathcal P(\{v>r\},\cdot)\Big) {\rm d}r
\end{equation} 
for any $0\ls s<t<+\infty$ (see \cite[Remark 4.3]{Mir03} or \cite[Corollary 1.9]{ABS19}).
 \end{enumerate} 
\end{proposition}

By applying to distance functions, we get the following weak convergence of the measures on spheres.
\begin{lemma}\label{lem2.16}
Let $(X_j,d_j,\mu_j)$ be a sequence of $RCD(K,N)$ metric measure spaces with $K\in \mathbb R$ and $N\in[1,+\infty)$. Suppose that 
 $(X_j,d_j,\mu_j,p_j) \overset{pmGH}{\longrightarrow} (X_\infty,d_\infty,\mu_\infty,p_\infty)$. Then we have,  for $\mathscr L^1$-a.e. $r\in\mathbb R^+$,  the functions 
 $\chi_{B_r(p_j)}$ is in $BV(X_j)$, and it holds
 \begin{equation}\label{equ2.15}
 |D \chi_{B_r(p_j)}|\rightharpoonup |D \chi_{B_r(p_\infty)}|,\quad {\rm as}\ \ j\to\infty
 \end{equation}
 in duality with $C_0(Z)$, where $Z$ is given in Definition \ref{pmGH}(3).
 \end{lemma}
 \begin{proof}
By the coarea formula, we know that for $\mathscr L^1$-a.e. $r\in \mathbb R^+$, the functions $\chi_{B_r(p_j)}\in BV(X_j)$. For such $r>0$, from  the observation that $\chi_{B_r(p_j)}\to \chi_{B_r(p_\infty)}$ in $L^1$-strong and  
Proposition 3.6 \cite{ABS19}, we conclude that  
$$\liminf_{j\to\infty} \int_{X_j}g{\rm d}|D\chi_{B_r(p_j)}|(X_j)\gs \int_{X_\infty}g{\rm d} |D\chi_{B_r(p_\infty)}|(X_\infty),\quad \forall 0\ls g\in Lip_0(Z),$$
where $Z$ is given in Definition \ref{pmGH} and it  is proper.
 On the other hand, for any $R>0$, 
\begin{equation*}
\begin{split}
\mu(B_R(p_j))&=\int_0^R|D\chi_{B_r(p_j)}|(X_j){\rm d}r\\
\to \mu_\infty(B_R(p_\infty))&=\int_0^R |D\chi_{B_r(p_\infty)}|(X_\infty){\rm d}r.
\end{split}
\end{equation*}
Therefore, we have $\lim_{j\to\infty} |D\chi_{B_r(p_j)}|(X_j)=  |D\chi_{B_r(p_\infty)}|(X_\infty)$ for $\mathscr L^1$-a.e. $r\in(0,R)$. At last, the desired assertion (\ref{equ2.15}) follows from  Corollary 3.7 in \cite{ABS19}, and it completes the proof.
\end{proof}

\begin{definition}\label{def2.17}
Let  $(X_j,d_j,\mu_j,p_j) \overset{pmGH}{\longrightarrow} (X_\infty,d_\infty,\mu_\infty,p_\infty)$, and let $Z,\Phi_j,\Phi_\infty$ be as in Definition \ref{pmGH}. A sequence of Borel sets $E_j\subset X_j$ with $\mu_j(E_j)<\infty$ for all $j\in\mathbb N$ is called to converge  in $L^1$-strong to a Borel set $F\subset X_\infty$ with $\mu_\infty(F)<\infty$ if $\chi_{E_j}\cdot\mu_j \rightharpoonup \chi_{F}\cdot\mu_\infty$ and   $\mu_j(E_j)\to \mu_\infty(F)$ as $j\to\infty$.

A sequence of Borel set $E_j\subset X_j$ is called to converge  in $L_{\rm loc}^1$  to a Borel set $F\subset X_\infty$ if $E_j\cap B_R(p_j)  \to F\cap B_R(p_\infty) $ in $L^1$-strong for every $R>0.$  
\end{definition}
Now recall the notion of reduced boundary of a set of locally finite perimeter in \cite{ABS19}. 
\begin{definition}[Reduced bondary]\label{def2.18}
Let $E$ be a set of locally finite perimeter in an  $ncRCD(K,N)$ metric measure space $(X,d,\mu)$. A point $x\in X$ is called a reduced boundary point of $E$, denoted by $x\in\mathcal{F}E$, if it satisfies the following:\\
(1)  it is in $  {\rm supp}(\mathcal P(E,\cdot))$  and it is a regular point of $X$. That is, for each $\{r_j\}$ with $r_j\to0$ the sequence $(X, r_j^{-1}d,\mu^x_{r_j},x)$ pointed Gromov-Hausdorff converges to $\mathbb R^N$ with the Euclidean metric; and\\
(2)  for each $\{r_j\}$ with $r_j\to0$, the sequence $ E\subset (X, r_j^{-1}d,\mu^x_{r_j},x)$ converges to the upper half space $\{x_N>0\}\subset \mathbb R^N$ in $L^1_{\rm loc}$.
 \end{definition}

 We also need the following properties for sets of finite perimeter in non-collapsed spaces, see \cite{Amb01,ABS19}.
\begin{proposition}\label{prop2.19}
Let $(X,d,\mu)$ be an  $ncRCD(K,N)$ metric measure space with $K\in \mathbb R$ and $N\in[1,+\infty)$. Then the followings hold:
\begin{enumerate} 
\item If  $E$ is a set of finite perimeter in $X$, then  $\mathscr H^{N-1}(\partial^*E)<+\infty$,
where  $\partial^*E$  is the {\emph{essential boundary}} of $E$, defined by  
$$\partial^*E:=\left\{x\in M\big|\ \limsup_{r\to0}\frac{\mu(B_r(x)\cap E)}{\mu(B_r(x))}>0\ \ \&\ \ \limsup_{r\to0}\frac{\mu(B_r(x)\backslash E)}{\mu(B_r(x))}>0\right\}.$$

 \item  If  $E$ is a set of locally finite perimeter, then $\mathscr H^{N-1}(\partial^*E\backslash \mathcal FE)=0$. Moreover, up to an $\mathscr H^{N-1}$-negligible set, it holds
$$\mathcal FE=\Big\{x\in E\big|\ \lim_{r\to0}\frac{\mathscr H^N(B_r(x)\cap E)}{\omega_Nr^N}=\frac 1 2\Big\},$$
   and that  $\mathcal P(E,\cdot)=   \mathscr{H}^{N-1}\llcorner \mathcal FE, $ (the De Giorgi's Theorem, see  \cite[Corollary 4.7]{ABS19}.)
   \end{enumerate} 
\end{proposition}

\section{Existence of a  minimizer\label{sec:existence}\label{sec-existence}}

In this section, we will derive the existence of a minimizer of \eqref{equ1.4}, where we always assume that $\Omega$ is a bounded domain in an $RCD(K,N)$-space $(X,d,\mu) $ with $K\in \mathbb R$ and $N\in(1,+\infty)$.   Let ${\bf g}=(g_1,g_2,\cdots, g_m)\in W^{1,2}(\Omega,[0,+\infty)^m)$ and let $\mathscr{A}_{\mathbf{g}}$ be given in (\ref{equ1.4}).   

Now we are ready to prove the existence of a minimizer, which is asserted in Proposition \ref{prop1.3}.
 
 \begin{proof}[Proof of Proposition \ref{prop1.3}]
  Since $\mathbf{g}\in\mathscr{A}_{\mathbf{g}}$ guarantees that $\mathscr{A}_{\mathbf{g}}\neq\emptyset$, there exists a minimizing sequence $\left\{{\bf u}^{k}\right\}_{k=1}^{\infty}\subseteq\mathscr{A}_{\mathbf{g}}$ such that
  \begin{equation}\label{equ3.1}
    \lim_{k\to\infty}J_Q({\bf u}^{k}) =\inf_{{\bf v}\in\mathscr{A}_{\mathbf{g}}}J_Q({\bf v})  \ \ \ (\ls  J_Q({\bf g}))\text{.}
  \end{equation}
From the Poincar\'e inequality (see Proposition \ref{prop2.1}(5)),   we get
  \begin{equation*} 
    \int_{\Omega}\left|{\bf u}^k-{\bf g}\right|^{2}\ls  C_1\int_{\Omega}\left|\nabla({\bf u}^k-{\bf g})\right|^{2},
  \end{equation*}
where the constant $C_1>0$ depends only on $N,K,\Omega$.   Thus,
  \begin{equation*}
    \begin{aligned}
      \left\|{\bf u}^{k}\right\|_{W^{1,2}(\Omega,\mathbb{R}^{m}) }^{2}\ls  &    2  \left\|{\bf u}^{k}-{\bf g}\right\|_{W^{1,2}(\Omega,\mathbb{R}^{m}) }^{2}+2      \left\|\mathbf{g}\right\|_{W^{1,2}(\Omega,\mathbb{R}^{m}) }^{2}      \\      
      \ls                                                                       & 2(1+C_1)\int_{\Omega}\left|\nabla({\bf u}^{k}-\mathbf{g})\right|^{2}+2\left\|\mathbf{g}\right\|_{W^{1,2}(\Omega,\mathbb{R}^{m}) }^{2}     \\
      \ls                                                                       & C_2\Big(\int_{\Omega}\left|\nabla{\bf u}^{k}\right|^{2}+\int_{\Omega}\left|\nabla\mathbf{g}\right|^{2}\Big)+2\left\|\mathbf{g}\right\|_{W^{1,2}(\Omega,\mathbb{R}^{m}) }^{2}  \\
      \ls                                                                       & C_2J_Q({\bf u}^{k}) +(C_2+2)\left\|\mathbf{g}\right\|_{W^{1,2}(\Omega,\mathbb{R}^{m}) }^{2}                                                           \\
      \ls                                                                       & C_2(J_Q(\mathbf{g}) +1) +(C_2+2)\left\|\mathbf{g}\right\|_{W^{1,2}(\Omega,\mathbb{R}^{m}) }^{2}
    \end{aligned}
  \end{equation*}
  for all sufficiently large $k$, where $C_2=4(1+C_1)$.
  Then, the fact that $W^{1,2}(\Omega,\mathbb{R}^{m}) $ is a Hilbert space implies that there exists a subsequence $\left\{{\bf u}^{k_{\ell}}\right\}_{\ell=1}^{\infty}$ of $\left\{{\bf u}^{k}\right\}_{k=1}^{\infty}$ such that $\left\{{\bf u}^{k_{\ell}}\right\}_{\ell=1}^{\infty}$ weakly converges to some ${\bf u}$ in $ W^{1,2}(\Omega,\mathbb{R}^{m}) $  and converges to ${\bf u}$ almost everywhere on $\Omega$. Noted that $\mathscr{A}_{\mathbf{g}}$ is a closed convex subset of $W^{1,2}(\Omega,\mathbb{R}^{m}) $, we conclude that  ${\bf u}\in\mathscr{A}_{\mathbf{g}}$. 
 
By noticing  that $\left\{\chi_{\left\{x\in\Omega\mid\left|{\bf u}^{k_{\ell}}(x)\right|>0\right\}}\right\}_{\ell=1}^{\infty}$ converges to $1$ almost everywhere on $\left\{x\in\Omega\mid\left|{\bf u}(x)\right|>0\right\}$ and that $Q\gs  0$, we have
  \begin{equation*}
Q\chi_{\{|{\bf u}|>0\}}\ls  \liminf_{\ell\to\infty}Q\chi_{\{|{\bf u}^{k_\ell}|>0\}}\quad \mu-{\rm a.e.\ in}\ \ \Omega.
  \end{equation*}
It follows from the Fatou lemma that
  \begin{equation*}
\int_{\Omega}Q\chi_{\{|{\bf u}|>0\}}\ls  \liminf_{\ell\to\infty}\int_{\Omega}Q\chi_{\{|{\bf u}^{k_\ell}|>0\}}.
  \end{equation*}
By combining  this and the lower semicontinuity of energy in Proposition \ref{prop2.1}(2), we obtain
  \begin{equation*} 
    \begin{aligned}
      J_Q({\bf u})  & =\int_{\Omega}(\left|\nabla{\bf u}\right|^{2}+Q\chi_{\left\{\left|{\bf u}\right|>0\right\}})                                        \\
                               & \ls  \liminf_{\ell\to\infty}\int_{\Omega}\left|\nabla{\bf u}^{k_{\ell}}\right|^{2}+\liminf_{\ell\to\infty}\int_{\Omega}Q\chi_{\left\{\left|{\bf u}^{k_{\ell}}\right|>0\right\}} \\
                               & \ls   \liminf_{\ell\to\infty}J_Q({\bf u}^{k_{\ell}})                                                                                                                              \\
                               & =\inf_{{\bf v}\in\mathscr{A}_{\mathbf{g}}}J_Q({\bf v}) \text{,}
    \end{aligned}
  \end{equation*}
 where we have used (\ref{equ3.1}).  Therefore, ${\bf u}$ is a minimizer of $J_Q$. The proof is finished.
\end{proof}

\begin{remark}\label{rem3.1}
Here, we only need to assume that $Q\in L^\infty(\Omega)$ and $Q\gs  0$ almost everywhere on $\Omega.$  
\end{remark}

\section{H\"older continuity of local minimizers\label{sec:holder}}

In this section, we will derive the locally H\"older regularity for the local minimizers of $J_Q$ in (\ref{equ1.2}). Recall the notations.
Let $(X,d,\mu) $ be  $RCD(K,N)$ metric measure space $(X,d,\mu)$ with $K\in \mathbb R$ and $N\in(1,+\infty)$. Let $\Omega\subset X$ be a bounded domain and let  $Q\in L^\infty(\Omega)$. Suppose that
$${\bf u}:=(u_1,u_2,\cdots, u_m)\in W^{1,2}(\Omega,[0,+\infty)^m)$$
is  a local minimizer of $J_Q$ in (\ref{equ1.2}). Namely, there are a boundary value ${\bf g}=(g_1,g_2,\cdots g_m)\in W^{1,2}(\Omega,[0,+\infty)^m)$ and a number 
$\varepsilon_{\bf u}>0$ such that 
$$ J_Q({\bf u})\ls  J_Q({\bf v}),\qquad \forall\ {\bf v}\in \mathscr A_{\bf g}\ \ {\rm with}\ \ \dist({\bf u},{\bf v})<\varepsilon_{\bf u},$$
where $\mathscr A_{\bf g}$ and $\dist({\bf u},{\bf v})$ are given in (\ref{equ1.4}) and (\ref{equ1.5}) respectively.

To begin, we argue that the components of local minimizers are subharmonic, so that powerful analytic tools can be applied later on.

\begin{lemma}[Subharmonicity]\label{lem4.1}
  Let ${\bf u}=(u_{1},\ldots,u_{m}) $ be a local minimizer of $J_Q$ on $\Omega$ with $Q\in L^\infty(\Omega)$. Then $\laplace u_{i}\gs  0$ on $\Omega$ in the sense of distributions, for all $i=1,\ldots,m$.
\end{lemma}
\begin{proof}
Fixed any ball $B_R(x)\subset \Omega$ such that $\mu(B_R(x))\ls  \frac{\varepsilon_{\bf u}}{2},$
it suffices to show that $u_i$ is subharmonic on  $B_R(x)$, where $i=1,\cdots,m.$ 
 
  For each $0\ls \phi\in \Lip_{0}(B_R(x))$, $\phi\not=0$, let
  \begin{equation}\label{equ4.1}
    {\bf v}_{i,\delta}:=(u_{1},\ldots,u_{i-1},(u_{i}-\delta\phi) ^{+},u_{i+1},\ldots,u_{m}) \text{,}
  \end{equation}
  where $i\in\left\{1,\ldots,m\right\}$ and $\delta>0$, then ${\bf v}_{i,\delta}\in \mathscr{A}_{\mathbf{g}}$ and, by (\ref{equ1.5}), 
\begin{equation*}
\begin{split}
\dist({\bf v}_{i,\delta},{\bf u})  &\ls  \|{\bf u}-{\bf v}_{i,\delta}\|_{W^{1,2}(\Omega,\mathbb R^m)}+\|\chi_{\{|{\bf u}-{\bf v}_{i,\delta}|>0\}}\|_{L^1(\Omega)}\\
&\ls \delta\|\phi\|_{W^{1,2}(B_R(x))}+ \mu(B_R(x)).
\end{split}
\end{equation*}
This implies $\dist({\bf v}_{i,\delta},{\bf u}) <\varepsilon_{\bf u}$ provided the $\delta<\delta_0:=\frac{\varepsilon_{\bf u}}{2\|\phi\|_{W^{1,2}(B_R(x))}}$.   Noted that $\left\{\left|{\bf v}_{i,\delta}\right|>0\right\}\subseteq\left\{\left|{\bf u}\right|>0\right\}$, the local minimality of ${\bf u}$ gives
  \begin{equation}\label{equ4.2}
    \begin{aligned}
      \int_{\Omega}(\left|\nabla{\bf u}\right|^{2}+Q\chi_{\left\{\left|{\bf u}\right|>0\right\}})  & \ls \int_{\Omega}(\left|\nabla{\bf v}_{i,\delta}\right|^{2}+Q\chi_{\left\{\left|{\bf v}_{i,\delta}\right|>0\right\}})  \\
                                                                                                                   & \ls \int_{\Omega}(\left|\nabla{\bf v}_{i,\delta}\right|^{2}+Q\chi_{\left\{\left|{\bf u}\right|>0\right\}}) 
    \end{aligned}
  \end{equation}
   for all  $\delta\in (0,\delta_0)$. Hence, we get $  \int_{\Omega} \left|\nabla{\bf u}\right|^{2} \ls \int_{\Omega} \left|\nabla{\bf v}_{i,\delta}\right|^{2}$. 
  This implies 
   \begin{equation}\label{equ4.3}
      \int_{\Omega} \left|\nabla u_{i}\right|^{2}\ls    \int_{\Omega} \left|\nabla(u_{i}-\delta\phi) \right|^{2}                                                      
      =         \int_{\Omega}(\left|\nabla u_{i}\right|^{2}-2\delta\ip{\nabla u_{i}}{\nabla\phi}+\delta^{2}\left|\nabla\phi\right|^{2}   ) 
  \end{equation}
  for all sufficiently small $\delta$.   Therefore, by letting $\delta\to 0$, we obtain  
    \begin{equation}\label{equ4.4}
    \int_{\Omega}\ip{\nabla u_{i}}{\nabla\phi}\ls  0\text{.}
  \end{equation}
The arbitrariness of $\phi$ implies $\laplace u_{i}\gs  0$ on $B_R(x)$ in the sense of distributions. The proof is finished.
\end{proof}
An immediate consequence is the local boundedness of ${\bf u}$.
\begin{remark}\label{rem4.2}
Let ${\bf u}=(u_{1},\ldots,u_{m}) $ be a local minimizer of  $J_Q$ on $\Omega$ with $Q\in L^\infty(\Omega)$, then for all $R>0$, there exists a constant $C=C_{N,K,R}>0$, depending only on $N,K,R$,  such that  
$$\sup_{B_{R/2}(x)}|{\bf u}|\ls C m \cdot \fint_{B_R(x)}|{\bf u}|$$
provided the ball $B_R(x)\subset \Omega$. Indeed,  from Lemma \ref{lem4.1} and \cite[Theorem 4.2]{KS01}, we conclude $\sup_{B_{R/2}(x)}|{ u_i}|\ls C  \cdot \fint_{B_R(x)}|{u_i}|$ for each $i=1,2,  
\cdots, m.$ \end{remark}

We will prove the H\"older continuity of  ${\bf u}$ by using Campanato theory, so we need to obtain a decay   estimate on $\int_{B_{r}(x_{0})}\left|\nabla{\bf u}\right|^{2}{\rm d}\mu$ (see, for example, \cite{Gor09}).

\begin{lemma}[H\"older continuity]\label{lem4.3}
  Let ${\bf u}=(u_{1},\ldots,u_{m}) $ be a local minimizer of $J_Q$ in \eqref{equ1.2}   with $Q\in L^\infty(\Omega)$. Then ${\bf u}\in C_{\mathrm{loc}}^{\alpha}(\Omega) $ for some $\alpha\in(0,1) $ (means that it has a $C_{\rm loc}^\alpha$ representative).
\end{lemma}
 \begin{proof} Fixed any ball $B_{\bar R}(\bar x)\subset\subset \Omega$ such that $\bar R<{\rm diam}(\Omega)/3$, it suffices to show ${\bf u}\in C^{\alpha}(B_{\bar R/2}(\bar x))$. Since ${\bf u}\in L^\infty_{\rm loc}(\Omega)$. We denote $M_1:=\sup_{B_{\bar R}(\bar x)}|{\bf u}|.$

  For each $x_{0}\in B_{\bar{R}/2}(\bar{x})$ and $R\in (0,\bar{R}/4)$,
  there exists some ${\bf v}\in W^{1,2}(B_{R}(x_{0}),\mathbb R^m)$ that solves the following (relaxed) Dirichlet problem \cite[Theorem 7.12]{Che99}:
  \begin{equation}\label{equ4.5}
    \begin{cases}
      \laplace {\bf v}=0 \qquad \text{ on }B_{R}(x_{0}) \text{,} \\
      {\bf u}-{\bf v}\in W_{0}^{1,2}(B_{R}(x_{0}),\mathbb R^m ) \text{.}
    \end{cases}
  \end{equation}
  After extending ${\bf v}$ by ${\bf u}$ on $\Omega\backslash B_{R}(x_{0}) $, (that is, ${\bf v}:={\bf u}$ in $\Omega\backslash B_{R}(x_{0})$, see Proposition \ref{prop2.2}), we have
  ${\bf v}\in\mathscr{A}_{\mathbf{g}}$ because all components of ${\bf v}$ are nonnegative
  on $B_{R}(x_{0}) $ by the maximum principle (see \cite[Theorem 7.17]{Che99}).

(i) We first check that $\dist({\bf u},{\bf v})<\varepsilon_{\bf u}$ whenever $R<R_0$ for some $R_0>0$ depending only on $K,N,\bar R$ and ${\bf u}$.

 By the Poincar\'e inequality 
$$\int_{B_R(x_0)}|{\bf u}-{\bf v}|^2\ls  C_1\int_{B_R(x_0)}|\nabla({\bf u}-{\bf v})|^2$$
for some constant $C_1$ depending only on $K,N,\bar R$ (see, for example, Proposition \ref{prop2.1}(5)), we get
$$ \|{\bf v}-{\bf u}\|_{W^{1,2}(\Omega,\mathbb R^m)}\ls  (\sqrt{C_1}+1) (\int_{B_R(x_0)}|\nabla({\bf u}-{\bf v})|^2) ^{1/2},$$
and  hence we have
  \begin{equation}\label{equ4.6}
    \begin{aligned}
 \dist({\bf u},{\bf v}) &\ls   C_2 (\int_{B_R(x_0)}|\nabla({\bf u}-{\bf v})|^2) ^{1/2}+ \int_{B_{R}(x_{0})} \left|\chi_{\left\{\left|{\bf u}\right|>0\right\}}-\chi_{\left\{\left|{\bf v}\right|>0\right\}}\right|  \\
 & \ls   C_2 (\int_{B_R(x_0)}|\nabla({\bf u}-{\bf v})|^2) ^{1/2}+ \mu(B_R(x_0)),
    \end{aligned}
  \end{equation}
where $C_2:=\sqrt{C_1}+1.$ 
Note  that
  \begin{equation}\label{equ4.7}
    \begin{aligned}
      \int_{B_{R}(x_{0}) }\left|\nabla({\bf u}-{\bf v}) \right|^{2} & \ls  2\int_{B_{R}(x_{0}) }\left|\nabla{\bf u}\right|^{2}+2\int_{B_{R}(x_{0}) }\left|\nabla{\bf v}\right|^{2} \\
                                                                                              & \ls  4\int_{B_{R}(x_{0}) }\left|\nabla{\bf u}\right|^{2},     
    \end{aligned}
  \end{equation}
  by the Dirichlet energy minimizing property of ${\bf v}$. By the combination of (\ref{equ4.6}), (\ref{equ4.7}) and the facts that
 $|\nabla {\bf u}|^2\in L^1(\Omega)$,  we conclude that there is $R_0\in (0,1)$ (depending only on $K,N,\bar R$ and ${\bf u}$) such that $\dist({\bf u},{\bf v})<\varepsilon_{\bf u}$ for all $R\in(0,R_0)$.
   
(ii)    Now by  the local minimality of ${\bf u}$, we have
  \begin{equation}\label{equ4.8}
    \begin{aligned}                      
                                                                                                \int_{B_{R}(x_{0}) }(\left|\nabla{\bf u}\right|^{2}-\left|\nabla{\bf v}\right|^{2})                                        
                                                                                              & \ls  \int_{B_{R}(x_{0}) }(\chi_{\left\{\left|{\bf v}\right|>0\right\}}-\chi_{\left\{\left|{\bf u}\right|>0\right\}}) Q \\
                                                                                              & \ls  \|Q\|_{L^\infty}\cdot \mu(B_{R}(x_{0}) )  
                                                                                                  \end{aligned}
  \end{equation}
   for all $R\in (0,R_0)$.  On the other hand, from ${\bf u}-{\bf v}\in W^{1,2}(\Omega,\mathbb R^m)$ and Remark \ref{rem2.4}(1), it can be taken as test functions for ${\bf \Delta}{\bf u}$ and ${\bf \Delta}{\bf v}$. Hence, we have (recalling ${\bf \Delta}{\bf v}=0$) that  
  \begin{equation}\label{equ4.9}
    \begin{aligned}
      \int_{B_{R}(x_{0}) }\left|\nabla({\bf u}-{\bf v}) \right|^{2} & =-\int_{B_{R}(x_{0}) }({\bf u}-{\bf v}){\rm d}\laplace({\bf u}-{\bf v})                                           \\
                                                                                              & =-\int_{B_{R}(x_{0}) }({\bf u}-{\bf v}){\rm d}\laplace({\bf u}+{\bf v})                                           \\ 
                                                                                              & =\int_{B_{R}(x_{0}) }(\left|\nabla{\bf u}\right|^{2}-\left|\nabla{\bf v}\right|^{2}) .                               
    \end{aligned}
  \end{equation}
Recall that $M_1:=\sup_{B_{\bar R}(\bar x)}|{\bf u}|.$  By the maximum principle and the Cheng-Yau's estimate (Theorem \ref{thm2.6}),   we have
\begin{equation}\label{equ4.10}
\begin{split}
  \sup_{B_{R/2}(x_0)}|\nabla{\bf v}| & \ls  C_3\sup_{B_{R}(x_0)}|{\bf v}| \ls  C_3\sup_{B_{3\bar R/4}(\bar x)}|{\bf u}|\ls  C_3M_1,
  \end{split}
  \end{equation}
 where we have used $B_R(x_0)\subset B_{3\bar R/4}(\bar x)$.  By combining the equations (\ref{equ4.8})-(\ref{equ4.10}), we conclude that for all  $R\in(0,R_0)$ and $r<R/2$, \begin{equation}\label{equ4.11}
    \begin{aligned}
      \int_{B_{r}(x_{0}) }\left|\nabla{\bf u}\right|^{2} & \ls  2\int_{B_{r}(x_{0}) }\left|\nabla({\bf u}-{\bf v}) \right|^{2}+2\int_{B_{r}(x_{0}) }\left|\nabla{\bf v}\right|^{2}                                           \\
                                                                      & \ls  2\int_{B_{R}(x_{0}) }\left|\nabla({\bf u}-{\bf v}) \right|^{2}+2C_3M_{1}\cdot \mu(B_{r}(x_{0})) \\
                                                                      & \ls  2\|Q\|_{L^\infty}\cdot\mu(B_{R}(x_{0}) ) +2C_3M_{1}\cdot \mu(B_{r}(x_{0})) .
                                                              \end{aligned}
  \end{equation}
 Let $R_1:=\min\{R_0,2^{-N/2}\}$. Then for any $R<R_1$, by 
  (\ref{equ4.11}) and taking $r=R^{1+2/N} \ ( \ls  \frac 1 2 R)$, we have for all $R\in(0,R_1)$ that
  \begin{equation}\label{equ4.12}
    \begin{aligned}
      r^{2}\fint_{B_{r}(x_{0}) }\left|\nabla{\bf u}\right|^{2} & \ls  2\|Q\|_{L^\infty} r^{2}\frac{\mu(B_{R}(x_{0}) ) }{\mu(B_{r}(x_{0}) ) }+2C_3M_{1}r^2 \\
                                                                            & \ls  2\|Q\|_{L^\infty} r^{2}C_4 (\frac{R}{r}) ^{N}+2C_3M_{1}r^2\\
                                                                            &\ls  (2\|Q\|_{L^\infty}  C_4  +2C_3M_{1})r^{\frac{4}{N+2}},
                                                                                                                                                           \end{aligned}
  \end{equation}
 where we used   $ \frac{\mu(B_{R}(x_{0}) ) }{\mu(B_{r}(x_{0}) ) }\ls  C_{4}(R/r)^N$ for some $C_4$ depending only on $K,N$, (by $r <R<1$, see (\ref{equ2.1})).
  It follows from the local Poincar\'{e} inequality \cite{Raj12} that
  \begin{equation}\label{equ4.13}
    \begin{aligned}
      \fint_{B_{r}(x_{0}) }\left|{\bf u}-\fint_{B_{r}(x_{0}) }{\bf u}\right| & \ls  C_{N,K,\bar R}\Big(r^{2}\fint_{B_{2r}(x_{0}) }\left|\nabla{\bf u}\right|^{2}\Big) ^{\frac{1}{2}}             \\
                                                                                                       & \ls  C_{K,N, \bar R,\|Q\|_{L^\infty}, M_{1}}\cdot r^{\frac{2}{N+2}} 
    \end{aligned}
  \end{equation}
  for all $r$ such that $r^{\frac{N}{N+2}}\in (0,R_1)$, which guarantees that ${\bf u}\in C^{1/(N+2) }(B_{\bar R/2}(\bar{x}) ) $, due to the Campanato theorem on metric measure spaces \cite[Theorem 3.2]{Gor09}.
\end{proof}
Here the H\"older index $1/(N+2)$ is not optimal. We will show,  in the next section, that ${\bf u}$ is locally Lipschitz continuous provided that the metric measure space $(X,d,\mu)$ is non-collapsed  and $Q$ satisfies (\ref{equ1.3}). 

In particular, Lemma \ref{lem4.3} implies that $\left\{x\in\Omega\mid\left|{\bf u}(x)\right|>0\right\}$ is an open set. By combining with the argument in Lemma \ref{lem4.1}, we get the following consequence.   

\begin{lemma}[Harmonicity]\label{lem4.4}
  Let ${\bf u}=(u_{1},\ldots,u_{m}) $ be a local minimizer of $J_Q$ in  \eqref{equ1.2}   with $Q\in L^\infty(\Omega)$. Then each component  $u_{i}$ is harmonic on the open set $\Omega_{\bf u}:=\left\{x\in\Omega\mid\left|{\bf u}(x)\right|>0\right\}$ for $i=1,\ldots,m$.
\end{lemma}

\begin{proof}
We know from Lemma \ref{lem4.3} that the set $ \Omega_{\bf u} $ is open. It suffices  to show that ${\bf\Delta} u_i=0$ in the sense of distributions on each small ball $B_R(x_0)\subset \Omega_{\bf u}$, for each $i=1,2,\cdots m.$

  For  each $0\ls \phi\in Lip_{0}(B_R(x_0)) $ let
  \begin{equation}\label{equ4.14}
    {\bf v}_{i,\delta}=(u_{1},\ldots,u_{i-1},u_{i}+\delta\phi,u_{i+1},\ldots,u_{m}) \text{,}
  \end{equation}
  where $i\in\left\{1,\ldots,m\right\}$ and $\delta>0$, then
  ${\bf v}_{i,\delta}\in \mathscr{A}_{\mathbf{g}}$.
  Noted that $ \dist({\bf v}_{i,\delta},{\bf u}) \ls  \delta\|\phi\|_{W^{1,2}}+Q_{\max}\mu({\rm supp}\phi)$, when both $\delta$ and $\mu(B_R(x_0))$ are  sufficiently small, the local minimality of ${\bf u}$ yields 
  \begin{equation*}
    \begin{aligned}
      \int_{\Omega}(\left|\nabla{\bf u}\right|^{2}+Q\chi_{\left\{\left|{\bf u}\right|>0\right\}})  & \ls \int_{\Omega}(\left|\nabla{\bf v}_{i,\delta}\right|^{2}+Q\chi_{\left\{\left|{\bf v}_{i,\delta}\right|>0\right\}})  \\
                                                                                                                   & \ls \int_{\Omega}(\left|\nabla{\bf v}_{i,\delta}\right|^{2}+Q\chi_{\left\{\left|{\bf u}\right|>0\right\}}) \text{,}
    \end{aligned}
  \end{equation*}
 where we have observed    $\{|{\bf v}_{i,\delta}|>0\}=\{|{\bf u}|>0\}$, as $\delta$ small enough.  Thus, we have
  \begin{equation}\label{4.15}
      \int_{\Omega} \left|\nabla u_{i}\right|^{2} \ls    \int_{\Omega} \left|\nabla(u_{i}+\delta\phi) \right|^{2}  
            =     \int_{\Omega}(\left|\nabla u_{i}\right|^{2}+2\delta\ip{\nabla u_{i}}{\nabla\phi}+\delta^{2}\left|\nabla\phi\right|^{2} ) 
  \end{equation}
  for all sufficiently small $\delta$.   Therefore, the arbitrariness of $\delta$ gives
  \begin{equation}\label{4.16}
    -\int_{\Omega}\ip{\nabla u_{i}}{\nabla\phi}\ls  0\text{.}
  \end{equation}
This yields  $\laplace u_{i}\ls  0$ on $B_R(x_0)$ in the sense of distributions.
  Meanwhile, Lemma \ref{lem4.1} asserts that $\laplace u_{i}\gs  0$ on
  $\Omega$ in the sense of distributions.
Thus, we conclude that $u_{i}$ is harmonic on $B_R(x_0)\subset \Omega_{\bf u}$. The proof is finished.
\end{proof}

\begin{remark}\label{rem4.5}
Recently, N. Gigli and I. V. Violo \cite{GV21}   obtained the locally H\"older continuity of a solution to an obstructed problem on  $RCD(K,N)$-spaces.
\end{remark}

\section{Lipschitz continuity of local minimizers\label{sec:lipschitz}}

In this section, we derive the Lipschitz regularity for local minimizers of $J_Q$ in \eqref{equ1.2} on a non-collapsed $RCD$ metric measure space. For this, we will begin with a mean value inequality on general $RCD(K,N)$-spaces.

\subsection{Mean value inequality} $\ $\\

Let $(X,d,\mu)$ be an $RCD(K,N)$ metric measure space with $K\in\mathbb R$ and $N\in (1,+\infty).$   
\begin{lemma}[Stokes formula on balls]\label{lem5.1}
 Let $B_R(x_0)\subset X$ and $\rho(\cdot)=d(\cdot, x_0)$, $\phi\in C^2([0,R])$ and let $\psi=\phi\circ\rho.$ Suppose that 
 $u\in C(\overline{B_{R}(x_{0}) }) \cap W^{1,2}(B_{R}(x_{0}) ) $.  If $\laplace\psi$ is a signed Radon measure, then
  \begin{equation}\label{equ5.1}
    \int_{B_{r}(x_{0}) }u{\rm d}\laplace\psi =-\int_{B_{r}(x_{0}) }\ip{\nabla u}{\nabla \psi }+\phi'(r) \left.\frac{\mathrm{d}}{\mathrm{d}s}\right|_{s=r}\Big(\int_{B_{s}(x_{0}) }u\Big) 
  \end{equation}
  holds for almost all $r\in(0,R) $.
\end{lemma}
\begin{proof}
  Since $\laplace \psi$ and $\mu$ are signed Radon measures, we have for almost all $r\in(0,R)$ that 
  \begin{equation}\label{equ5.2}
    \lim_{j\to\infty}\left|\laplace \psi\right|(\overline{B_{r+1/j}(x_{0}) }\backslash B_{r}(x_{0}) ) =0
  \end{equation}
  and
  \begin{equation}\label{equ5.3}
    \lim_{j\to\infty}\mu(\overline{B_{r+1/j}(x_{0}) }\backslash B_{r}(x_{0}) ) =0.
  \end{equation}
 Meanwhile, noted that $s\mapsto\int_{B_{s}(x_{0}) }u$ is locally Lipschitz continuous on $(0,R)$, it is differentiable almost everywhere on $(0,R) $ too. We fix an $r$ such that both of them hold. For $j$ sufficiently large, let $u_{j}=\eta_{j}(\rho ) u\in W_{0}^{1,2}(B_{R}(x_{0}) ) $, where
  \begin{equation}\label{equ5.4}
    \eta_{j}(t) =\begin{cases}
      1                   & \text{if }t\in\left[0,r\right]\text{,}             \\
      1-j(t-r)  & \text{if }t\in(r,r+\frac{1}{j}) \text{,} \\
      0                   & \text{if }t\in\left[r+\frac{1}{j},R\right]\text{.}
    \end{cases}
  \end{equation}
  On the one hand,
  \begin{equation}\label{equ5.5}
    \begin{aligned}
      \int_{B_{R}(x_{0}) }u_{j}{\rm d}\laplace\psi & =\int_{B_{R}(x_{0}) \backslash B_{r}(x_{0}) }u_{j}{\rm d}\laplace\psi +\int_{B_{r}(x_{0}) }u_{j}{\rm d}\laplace\psi          \\
                                                      & =\int_{B_{R}(x_{0}) \backslash B_{r}(x_{0}) }\eta_{j}u{\rm d}\laplace\psi +\int_{B_{r}(x_{0}) }u{\rm d}\laplace\psi \text{,}
    \end{aligned}
  \end{equation}
  where
  \begin{equation}\label{equ5.6}
    \left|\int_{B_{R}(x_{0}) \backslash B_{r}(x_{0}) }\eta_{j}u{\rm d}\laplace\psi \right|\ls \left|\laplace \psi \right|(\overline{B_{r+1/j}(x_{0}) }\backslash B_{r}(x_{0}) ) \left\|u\right\|_{C^0(B_{R}(x_{0}) )}.
  \end{equation}
By combining this, (\ref{equ5.2})  and (\ref{equ5.6}), we have
  \begin{equation}\label{equ5.7}
    \lim_{j\to\infty}\int_{B_{R}(x_{0}) }u_{j}{\rm d}\laplace\psi =\int_{B_{r}(x_{0}) }u{\rm d}\laplace\psi \text{.}
  \end{equation}
  On the other hand, from Remark \ref{rem2.4}(1) and $u_j\in W^{1,2}_0(B_R(x_0))$, we have 
  \begin{equation}\label{equ5.8}
    \begin{aligned}
      \int_{B_{R}(x_{0}) }u_{j}{\rm d}\laplace\psi = & -\int_{B_{R}(x_{0}) }\ip{\nabla u_{j}}{\nabla \psi }                                                                                                \\
      =                                                 & -\int_{B_{R}(x_{0}) }\ip{\nabla u}{\nabla \psi }\eta_{j}-\int_{B_{R}(x_{0}) }\ip{\nabla \eta_{j}}{\nabla \psi }u                          \\
      =                                                 & -\int_{B_{R}(x_{0}) \backslash B_{r}(x_{0}) }\ip{\nabla u}{\nabla \psi }\eta_{j}-\int_{B_{r}(x_{0}) }\ip{\nabla u}{\nabla \psi } \\
                                                        & -\int_{B_{R}(x_{0}) }\ip{\nabla \eta_{j}}{\nabla \psi }u\text{.}
    \end{aligned}
  \end{equation}
We want to estimate the first and the third terms of right hand side in (\ref{equ5.8}). From (\ref{equ5.3}) and   
  \begin{equation*}
    \begin{aligned}
                & \left|\int_{B_{R}(x_{0}) \backslash B_{r}(x_{0}) }\ip{\nabla u}{\nabla \psi }\eta_{j}\right|                                                                                                                                                             
      \ls    \int_{B_{r+\frac{1}{j}}(x_{0}) \backslash B_{r}(x_{0}) }\left|\nabla u\right|\left|\nabla \psi \right|                                                                                                                                                  \\
      \ls  & \sup_{\left[0,R\right]}\left|\phi'\right|\Big(\int_{B_{r+\frac{1}{j}}(x_{0}) \backslash B_{r}(x_{0}) }\left|\nabla u\right|^{2}\Big)^{\frac{1}{2}}\mu(B_{r+\frac{1}{j}}(x_{0}) \backslash B_{r}(x_{0}) ) ^{\frac{1}{2}},
    \end{aligned}
  \end{equation*}
 we conclude that the first term of right hand side in (\ref{equ5.8}) converges to $0$ as $j\to\infty.$ Noticed that  
  \begin{equation}\label{equ5.9}
    \begin{aligned}
        & \int_{B_{R}(x_{0}) }\ip{\nabla \eta_{j}}{\nabla \psi }u                                                                                                                                                                                                =  \int_{B_{r+\frac{1}{j}}(x_{0}) \backslash B_{r}(x_{0}) }\ip{-j\nabla \rho }{\nabla \psi }u                                                                                                                                                    \\
      = & -j\int_{B_{r+\frac{1}{j}}(x_{0}) \backslash B_{r}(x_{0}) }\phi'(\rho) u                                                                                                                                                             \\
      = & -j\phi'(r) \int_{B_{r+\frac{1}{j}}(x_{0}) \backslash B_{r}(x_{0}) }u-j\int_{B_{r+\frac{1}{j}}(x_{0}) \backslash B_{r}(x_{0}) }\big(\phi'(\rho) -\phi'(r) \big) u                           \\
      = & -\phi'(r) \frac{\int_{B_{r+\frac{1}{j}}(x_{0}) }u-\int_{B_{r}(x_{0}) }u}{\frac{1}{j}}-j\int_{B_{r+\frac{1}{j}}(x_{0}) \backslash B_{r}(x_{0}) }\big(\phi'(\rho) -\phi'(r) \big) u\text{,}
    \end{aligned}
  \end{equation}
and that
  \begin{equation}\label{equ5.10}
    \begin{aligned}
                & \left|j\int_{B_{r+\frac{1}{j}}(x_{0}) \backslash B_{r}(x_{0}) }\big(\phi'(\rho) -\phi'(r)\big ) u\right|                                                                                                                            \\                                                                                                                           \\
      \ls  &j \int_{B_{r+\frac{1}{j}}(x_{0}) \backslash B_{r}(x_{0}) }\frac{1}{j}\sup_{\left[0,R\right]}\left|\phi''\right|\left|u\right|                                                                                                                              \\
      \ls  & \sup_{\left[0,R\right]}\left|\phi''\right|(\int_{B_{r+\frac{1}{j}}(x_{0}) \backslash B_{r}(x_{0}) }\left|u\right|^{2}) ^{\frac{1}{2}}\mu(B_{r+\frac{1}{j}}(x_{0}) \backslash B_{r}(x_{0}) ) ^{\frac{1}{2}}\text{,}
    \end{aligned}
  \end{equation}
we conclude, by (\ref{equ5.3}) and the fact that $r\mapsto \int_{B_r(x_0)}u$ is differentiable at $r$, that the third term of right hand side in (\ref{equ5.8}) converges to $\phi'(r)\frac{{\rm d}}{{\rm d}r}(\int_{B_r(x_0)}u) ,$  as $j\to\infty$. Therefore, letting $j\to\infty$ in (\ref{equ5.8}), we obtain for almost all $r\in (0,R)$ that 
  \begin{equation}\label{equ5.11}
    \int_{B_{r}(x_{0}) }u{\rm d}\laplace\psi =-\int_{B_{r}(x_{0}) }\ip{\nabla u}{\nabla \psi }+\phi'(r) \frac{\mathrm{d}}{\mathrm{d}r}(\int_{B_{r}(x_{0}) }u) .
  \end{equation}
  The proof is finished.
\end{proof}

A similar argument with a different cut-off function yields the following slight variant.

\begin{remark}[Stokes formula on annuli]\label{rem5.2}
  Let $\overline{B_{R_{2}}(x_{0}) }\backslash B_{R_{1}}(x_{0}) \subset X$,   $\rho(\cdot)=d(\cdot,x_{0})$, $\phi\in C^{2}(\left[R_{1},R_{2}\right]) $, and let $\psi=\phi\circ\rho$.  
  Suppose that 
  \begin{equation}\label{equ5.12}
    u\in C(\overline{B_{R_{2}}(x_{0}) }\backslash B_{R_{1}}(x_{0}) ) \cap W^{1,2}(B_{R_{2}}(x_{0}) \backslash B_{R_{1}}(x_{0}) ) \text{,}
  \end{equation}
and if $\laplace\psi$ is a signed Radon measure, then
  \begin{equation}\label{equ5.13}
    \begin{aligned}
      \int_{B_{r_{2}}(x_{0}) \backslash B_{r_{1}}(x_{0}) }u{\rm d}\laplace\psi =&-\int_{B_{r_{2}}(x_{0}) \backslash B_{r_{1}}(x_{0}) }\ip{\nabla u}{\nabla \psi } \\
      & +\phi'(r_{2}) \left.\frac{\mathrm{d}}{\mathrm{d}s}\right|_{s=r_{2}}(\int_{B_{s}(x_{0}) }u)   -\phi'(r_{1}) \left.\frac{\mathrm{d}}{\mathrm{d}s}\right|_{s=r_{1}}(\int_{B_{s}(x_{0}) }u) 
    \end{aligned}
  \end{equation}
  holds for almost all $R_{1}<r_{1}<r_{2}<R_{2}$.
\end{remark}

\begin{lemma}[Mean value inequality]\label{lem5.3}
  Let $(X,d,\mu)$ be an   $RCD(K,N)$-space with $K\in\mathbb R$ and $N\in(1,+\infty)$, and let $\Omega\subset X$ be a bounded domain. Suppose that     $B_{R_{0}}(x_{0}) \Subset \Omega$ and that 
  \begin{equation}\label{equ5.14}
    u\in C(\overline{B_{R_{0}}(x_{0}) }) \cap W^{1,2}(B_{R_{0}}(x_{0}) ,\left[0,+\infty\right) ) \text{,}
  \end{equation}
  and let $\rho=d(\cdot,x_{0}) $. If 
\begin{equation}\label{equ5.15}
\liminf_{r\to0^+}\frac{1}{r^N}\int_{B_r(x_0)}u(x){\rm d}\mu=0,
\end{equation}  then
  \begin{equation}\label{equ5.16}
    \fint_{B_{R}(x_{0}) }u\ls  C_1 \int_{0}^{R}\frac{e^{- C_2s^2}}{s^{N+1}}\int_{B_{s}(x_{0}) }\langle\nabla u,\nabla \rho ^{2}\rangle\mathrm{d}s 
  \end{equation}
for all $R\in(0,R_0)$,  where the constants $C_1,C_2$ only depend  on $N,K$  and $\Omega$.
\end{lemma}
\begin{remark}\label{rem5.4}
  On the  Euclidean space  $\mathbb R^N$, we have
  \begin{equation*}
    \fint_{B_{R}(x_{0}) }u-u(x_{0}) =\frac{1}{2}\int_{0}^{R}\frac{1}{s}\fint_{B_{s}(x_{0}) }\langle\nabla u,\nabla \rho ^{2}\rangle\mathrm{d}s\text{.}
  \end{equation*}
\end{remark}

\begin{proof}[Proof of Lemma \ref{lem5.3}]
  Noted that $\laplace \rho ^{2}$ is a signed Radon measure \cite{Gig15} and $u\gs  0$, Lemma \ref{lem5.1} asserts that
  \begin{equation*} 
    \int_{B_{r}(x_{0}) }u{\rm d}\laplace(\rho ^{2}) =-\int_{B_{r}(x_{0}) }\ip{\nabla u}{\nabla \rho ^{2}}+2r\frac{\mathrm{d}}{\mathrm{d}r}\big(\int_{B_{r}(x_{0}) }u\big) 
  \end{equation*}
  holds for almost all $r\in(0,R) $.
  By combining this and the Laplacian comparison theorem (see \eqref{equ2.8} in Theorem \ref{thm2.5}), we get
  \begin{equation}\label{equ5.17}
    \begin{aligned}
      \frac{\mathrm{d}}{\mathrm{d}r}(\int_{B_{r}(x_{0}) }u)  & =\frac{1}{2r}\int_{B_{r}(x_{0}) }u{\rm d}\laplace(\rho ^{2}) +\frac{1}{2r}\int_{B_{r}(x_{0}) }\ip{\nabla u}{\nabla \rho ^{2}}                                  \\
                                                                                 & \ls \frac{1}{2r}\int_{B_{r}(x_{0}) }2(N+C_3\rho ^{2}) u+\frac{1}{2r}\int_{B_{r}(x_{0}) }\ip{\nabla u}{\nabla \rho ^{2}}\text{.}
    \end{aligned}
  \end{equation}
 Here and in the following of this proof, $C_1, C_2, C_3,\cdots, $ will denote positive constants depending only on $N,K$ and $\Omega$. 
 This gives
    \begin{equation*}
    \begin{aligned}
      \frac{\mathrm{d}}{\mathrm{d}r}(\frac{1}{r^{N}}\int_{B_{r}(x_{0}) }u) \ls  &      C_3r\cdot  \Big(\frac{1}{r^{N}}\int_{B_{r}(x_{0}) } u\Big)+\frac{1}{2r^{N+1}}\int_{B_{r}(x_{0}) }\ip{\nabla u}{\nabla \rho ^{2}}  
          \end{aligned}
  \end{equation*}
for almost all $r\in(0,R)$.   Multiplying both sides by $\exp(-C_3r^{2}/2) $, we have
  \begin{equation*}
    \begin{aligned}
      \frac{\mathrm{d}}{\mathrm{d}r}\left(e^{-\frac{1}{2}C_3r^{2}}\frac{1}{r^{N}}\int_{B_{r}(x_{0}) }u\right)                        & \ls  \frac{e^{-\frac 1 2 C_3r^2}}{2r^{N+1}}\int_{B_{r}(x_{0}) }\ip{\nabla u}{\nabla \rho ^{2}}\text{.}
    \end{aligned}
  \end{equation*}
  Since $r\mapsto\frac{1}{r^{N}}e^{-\frac{1}{2}C_3r^{2}}\int_{B_{r}(x_{0}) }u$ is locally Lipschitz on $\left(0,R\right]$, by integrating the above inequality over $(r,R)$ for any $r<R$, we get
  \begin{equation}\label{equ5.18}
  \begin{split}
    \frac{e^{-\frac{1}{2}C_3R^{2}}}{R^{N}}   \int_{B_{R}(x_{0}) }u- &\frac{e^{-\frac{1}{2}C_3r^{2}}}{r^{N}} \int_{B_{r}(x_{0}) }u\\ 
 &\ls   \int_{r}^{R}\frac{e^{-\frac 1 2 C_3s^2}}{2s^{N+1}} \int_{B_{s}(x_{0}) }\ip{\nabla u}{\nabla \rho ^{2}}\mathrm{d}s\text{.}
    \end{split}
  \end{equation}
By substituting the assumption \eqref{equ5.15} into (\ref{equ5.18}) and let $r\to0^+$, we get
  \begin{equation*}
    \begin{aligned}
      \frac{e^{-\frac{1}{2}C_3R^{2}}}{R^{N}}\int_{B_{R}(x_{0}) }u  \ls  \int_{0}^{R}\frac{e^{-\frac 1 2 C_3s^2}}{2s^{N+1}} \int_{B_{s}(x_{0}) }\ip{\nabla u}{\nabla \rho ^{2}}\mathrm{d}s\text{.}
    \end{aligned}
  \end{equation*}
  Therefore, by  $\mu(B_R(x_0))/R^N\gs  C_5:=\mu(B_{{\rm diam}(\Omega)}(x_0))/[{\rm diam}(\Omega)]^N $ (this follows from \eqref{equ2.1} and  $R<{\rm diam}(\Omega)$), we conclude 
  \begin{equation*}
      \fint_{B_{R}(x_{0}) }u   \ls   e^{\frac{1}{2}C_3R^{2}}  \frac{1}{C_5}\int_{0}^{R}\frac{e^{-\frac 1 2 C_3s^2}}{2s^{N+1}} \int_{B_{s}(x_{0}) }\ip{\nabla u}{\nabla \rho ^{2}}\mathrm{d}s\text{.}
        \end{equation*}
This implies (\ref{equ5.16}) with $C_1:= \frac{e^{\frac{1}{2}C_3R^{2}} }{2C_5}$ and $C_2:=\frac{C_3}{2}$. The proof is finished. 
\end{proof}

\subsection{Lipschitz continuity of local minimizers of $J_Q$ }\ $\ $\\

From now on, we shall suppose that  $(X,d,\mu)$ is an $ncRCD(K,N)$ metric measure space with $K\ls 0$ and $N\in(1,+\infty)$. Let $\Omega\subset X$ be a bounded  domain and let $Q\in L^\infty(\Omega)$ satisfy (\ref{equ1.3}) for two positive numbers $Q_{\min}$ and $Q_{\max}$. Recall that a map
$${\bf u}=(u_1,u_2,\cdots, u_m)\in W^{1,2}(\Omega,[0,+\infty)^m)$$
is  a local minimizer of $J_Q$ in (\ref{equ1.2}) if there exist a  data ${\bf g}\in W^{1,2}(\Omega,[0,+\infty)^m)$ and $\varepsilon_{\bf u}>0$ such that $J_Q({\bf u})\ls  J_Q({\bf v})$ for all ${\bf v}\in\mathscr A_{\bf g}$ with $\dist({\bf u},{\bf v})<\varepsilon_{\bf u}$, where the $\mathscr A_{\bf g}$ and $\dist({\bf u},{\bf v})$ are given in (\ref{equ1.4}) and (\ref{equ1.5}), respectively. From Lemma \ref{lem4.3}, we can assume that ${\bf u}$ is continuous on $\Omega$. The set $\Omega_{\bf u}:=\{x\in\Omega|\ |{\bf u}|(x)>0\}$ is open.

The following lemma is inspired by the classical Caccioppoli inequality.
\begin{lemma}\label{lem5.5}
 Let ${\bf u}=(u_{1},\ldots,u_{m}) $ be a local minimizer of $J_Q$ in  \eqref{equ1.2}   with $Q$ satisfying  \eqref{equ1.3},
 and let $\Omega'\Subset \Omega$. Then there exists a constant $R_0\ (\ls  1)$ depending only on $N,K,\Omega,d(\Omega',\partial\Omega),Q_{\max}$ and $\varepsilon_{\bf u}$, such that for all balls $B_r(x)$ with $r<R_0$ and $d(x,\Omega')<R_0$, it holds
  \begin{equation}\label{equ5.19}
    -\int_{B_{r}(x) }\langle\nabla u_{i},\nabla\phi\rangle\ls  \big(Q_{\max}\cdot \mu(B_r(x))\big)^{1/2} \cdot\|\phi\|_{W^{1,2}(B_r(x))}
      \end{equation}
  for all  $i=1,\cdots, m$ and $\phi\in W^{1,2}_0(B_r(x))$ and $\phi\gs  0$.
\end{lemma}
\begin{proof}
From (\ref{equ2.12}), there is a number $R_0\in (0,\frac{1}{2}d(\Omega',\partial \Omega))$ with $R_0<1$, (depending only on $N,K,d(\Omega',\partial\Omega),\Omega,Q_{\max} $ and $\varepsilon_{\bf u}$,)
such that 
\begin{equation}\label{equ5.20}
\big(Q_{\max}\mu(B_{R_0}(x))\big)^{1/2}+ Q_{\max}\mu(B_{R_0}(x))<\varepsilon_{\bf u},\qquad \forall\ x\in\Omega.
\end{equation}
 Fix any ball $B_r(x)$ with $r<R_0$ and $d(x,\Omega')<R_0$, where $R_0$ is given in  the above  (\ref{equ5.20}). The inequality (\ref{equ5.19}) obviously holds if $\|\phi\|_{W^{1,2}(B_r(x))}=0$,
  so we are assuming that $\|\phi\|_{W^{1,2}(B_r(x))}>0$ in the following.
  Put
  \begin{equation}\label{equ5.21}
    \delta:=\big(Q_{\max}
    \cdot\mu(B_{r}(x)) \big)^{\frac{1}{2}}\Big/ \|\phi\|_{W^{1,2}(B_r(x))} \text{.}
  \end{equation}
 Let ${\bf v}=(u_1,\cdots, u_{i-1},u_i+\delta\phi,u_{i+1},\cdots,u_m)\in \mathscr A_{\bf g}$. Note  that
  \begin{equation*}
    \begin{aligned}
    \dist({\bf v},{\bf u})&\ls     \delta\cdot \|\phi\|_{W^{1,2}(B_r(x)))}+Q_{\max}\mu({\rm supp}\phi)\\
   & \ls   (Q_{\max}\mu(B_{r}(x)) ) ^{\frac{1}{2}}   +Q_{\max}\mu(B_{r}(x)) <\varepsilon_{\bf u} \\
     \end{aligned}
  \end{equation*}
provided $r<R_0$, by  (\ref{equ5.20}). The local minimality of ${\bf u}$ implies
  \begin{equation*}
    \begin{aligned}
      \int_{\Omega}\left|\nabla {\bf u}\right|^{2}= & J_Q({\bf u}) -\int_{\Omega}Q\chi_{\left\{\left|{\bf u}\right|>0\right\}}  \ls   J_Q({\bf v}) -\int_{\Omega}Q\chi_{\left\{\left|{\bf u}\right|>0\right\}}                                                                                                                                                                \\
      \ls                                         & \int_{\Omega}\left|\nabla{\bf v}\right|^{2}+\int_{\Omega}(Q\chi_{\left\{\left|{\bf v}\right|>0\right\}}-Q\chi_{\left\{\left|{\bf u}\right|>0\right\}})  \\
      \ls                                         & \int_{\Omega}\left|\nabla  {\bf v}  \right|^{2}+Q_{\max}\mu(B_{r}(x ) )                                                                                                             \\
      \ls                                         & \int_{\Omega}\left|\nabla {\bf u}\right|^{2}+2\delta\int_{B_{r}(x) }\langle\nabla u_{i},\nabla\phi\rangle+\delta^{2}\int_{B_{r}(x) }\left|\nabla\phi\right|^{2}     +Q_{\max}\mu(B_{r}(x) ) 
    \end{aligned}
  \end{equation*}
 for all $r<R_0$. Therefore,
  \begin{equation*}
    \begin{aligned}
      -2\int_{B_{r}(x) }\langle\nabla u_{i},\nabla\phi\rangle & \ls \delta\int_{B_{r}(x) }\left|\nabla\phi\right|^{2}+\frac{Q_{\max}}{\delta}\mu(B_{r}(x) )                         \\
                                                                            & \ls \delta\cdot \|\phi\|_{W^{1,2}(B_r(x))}^{2}+\frac{Q_{\max}}{\delta}\mu(B_{r}(x) )            \text{,}
    \end{aligned}
  \end{equation*}
  which is equivalent to \eqref{equ5.19} by (\ref{equ5.21}), and the proof is finished.
\end{proof}

Combining Lemma \ref{lem5.5} and Lemma \ref{lem5.3}, we are able to control the growth of local minimizers near the free boundary $\partial\{|{\bf u}|>0\}\cap \Omega.$

\begin{lemma}[Optimal linear growth]\label{lem5.6}
  Let ${\bf u}=(u_{1},\ldots,u_{m}) $ be a local minimizer of $J_Q$ in \eqref{equ1.2}   with $Q$ satisfying  \eqref{equ1.3}, and let $\Omega'\Subset\Omega$. Let  $R_0$ be the constant given in Lemma \ref{lem5.5}.  Suppose that  $B_r(x_0)$ is a ball with  radius $r<R_0$ and  $d(x_0,\Omega')<R_0$. Then,   if $u_i(x_0)=0$ for some $i=1,\cdots,m$, it holds 
    \begin{equation}\label{equ5.22}
    \sup_{B_{r/2}(x_{0}) }u_{i}(x)\ls  C\sqrt{Q_{\max}}\cdot r,
  \end{equation}
 where $C=C_{N,K,\Omega}>0$  depends only on $K,N $ and $\Omega$.
 \end{lemma}
 
\begin{proof} Since $(X,d,\mu)$ is non-collapsed, from $u_i(x_0)=0$ and (\ref{equ2.12}), we know that the condition  (\ref{equ5.15}) in Lemma \ref{lem5.3} holds.    
For each $s<R_0$, by using Lemma \ref{lem5.5} to $\phi=s^2-\rho^2(x)$, we get
  \begin{equation*}
    \begin{aligned}
      \int_{B_{s}(x_{0}) }\langle\nabla u_{i},\nabla \rho ^{2}\rangle & =-\int_{B_{s}(x_{0}) }\langle\nabla u_{i},\nabla(s^{2}-\rho ^{2}) \rangle                               \\
      &\ls   \big(Q_{\max}\cdot \mu(B_s(x_0))\big)^{1/2}\cdot\|s^2-\rho^2\|_{W^{1,2}(B_s(x_0))} \\
                                                                                 & \ls \big(Q_{\max}\cdot \mu(B_s(x_0))\big)^{1/2}\Big(\int_{B_s(x_0)}(4\rho^2|\nabla \rho|^2+(s^2-\rho^2)^2 ) \Big)^{\frac{1}{2}}       \\
                                                                                 & \ls  3 \sqrt{Q_{\max}}\cdot \mu(B_s(x_0))\cdot s  \quad\qquad ({\rm by}\ \ 
                                                                                 \rho<s,\   \ s<R_0\ls  1).   \\
 & \ls  C_{N,K} \sqrt{Q_{\max}}  \cdot s^{N+1}  \quad\qquad ({\rm by}\ \eqref{equ2.12},\  s\ls1).
                                                                  \end{aligned}
  \end{equation*}
 Since ${\bf u}$ is continuous on $\overline{B_{r}(x_{0}) }$ by Lemma \ref{lem4.3},
  it follows from Lemma \ref{lem5.3} and $u_i(x_0)=0$ that
  \begin{equation}\label{equ5.23}
    \begin{aligned}
      \fint_{B_{r}(x_{0}) }u_{i} & \ls   C_1 \int_{0}^{r}\frac{e^{- C_2s^2}}{s^{N+1}}\int_{B_{s}(x_{0}) }\langle\nabla u_i,\nabla \rho ^{2}\rangle\mathrm{d}s                                             \ls  C_3\sqrt{Q_{\max}}\cdot r\text{,}
    \end{aligned}
  \end{equation}
  for all $r<R_0$. Thus, by using the fact that $u_i$ is subharmonic (Lemma \ref{lem4.1}) and $u_i\gs  0$, we get (see, for example, \cite[Theorem 4.2]{KS01})
  $$\sup_{B_{r/2}(x_0)}u_i\ls  C_4\fint_{B_r(x_0)}u_i\ls  C_4C_3\sqrt{Q_{\max}}\cdot r,\qquad \forall\ r<R_0.$$
The proof is finished.
\end{proof}

As a corollary of the combination of the linear growth and Cheng-Yau's gradient estimate for harmonic functions, one can get the following gradient estimate near the free boundary   $\partial\{|{\bf u}|>0\}\cap \Omega.$

\begin{lemma}\label{lem5.7}
  Let ${\bf u}=(u_{1},\ldots,u_{m}) $ be a local minimizer of $J_Q$ in \eqref{equ1.2}  with $Q$ satisfying  \eqref{equ1.3},   and let $\Omega'\Subset\Omega$.
  There exists a positive constant $C=C_{N,K,\Omega}>0$ (depending only on $K,N$ and $\Omega$), such that: if $x_1\in \Omega'$ and if $d(x_1,\{|{\bf u}|=0\}\cap\Omega)<R_0/8$,  then it holds
  \begin{equation}\label{equ5.24}
    \Lip u_i(x_1)  \ls  C\sqrt{Q_{\max}} ,\qquad i=1,2,\cdots, m,
  \end{equation}
 where $R_0$ is given in Lemma \ref{lem5.5}, and $\Lip u_i$ is the pointwise Lipschitz constant defined in \eqref{equ2.4}.
 \end{lemma}

\begin{proof}

We will finish this proof by considering two cases as follows.\\
(i) In the case where $d(x_1,\{|{\bf u}|=0\}\cap\Omega)=0$. The continuity of ${\bf u}$ implies the  $\{|{\bf u}|=0\}$ is relative closed in $\Omega$. This implies $x_1\in \{|{\bf u}|=0\}$ in this case. By Lemma \ref{lem5.6}, we have
$$\sup_{y\in B_{r/2}(x_1)}\frac{|u_i(y)-u_i(x_1)|}{r}\ls  C\sqrt{Q_{\max}},\quad \ i=1,2,\cdots, m,$$
for all $r<R_0$. By (\ref{equ2.4}) and letting $r\to0$, this yields $\Lip u_i(x_1)\ls  C\sqrt{Q_{\max}}.$

(ii)  In the case  where $d(x_1,\{|{\bf u}|=0\}\cap\Omega)>0$. We put   
$$r_1:=d(x_1,\{|{\bf u}|=0\}\cap\Omega)\in (0,R_0/8).$$
Since $B_{r_1/2}(x_1)\subset \{|{\bf u}|>0\}$, 
 from Lemma \ref{lem4.4}, we have known that all $u_i$, $i=1,\cdots, m,$ are harmonic on $B_{r_1/2}(x_1)$. By using Cheng-Yau estimate, Theorem \ref{thm2.6}, we obtain
\begin{equation}\label{equ5.25}
\Lip u_i(x_1)\ls  \sup_{y\in B_{r_1/4}(x_1)}\Lip u_i(y)\ls    \frac{C}{r_1}\sup_{y\in B_{ r_1/4}(x_1)} u_i(y),
\end{equation}
 where the constant $C$ depends only on $N,K$ and $\Omega$. Take $x_2\in  B_{2r_1}(x_1)\cap \{|{\bf u}|=0\}$. 
By applying Lemma \ref{lem5.6} to $B_{8r_1}(x_2)$, (remark that  $d(x_2,\Omega')<2r$ and the assumption $8r_1<R_0$,) we have
 \begin{equation}\label{equ5.26}
 \sup_{y\in B_{r_1/4}(x_1)} u_i(y)\ls  \sup_{y\in B_{4r_1}(x_2)} u_i(y)\ls  C\sqrt{Q_{\max}}\cdot r_1.
   \end{equation}
 The combination of (\ref{equ5.25}) and (\ref{equ5.26}) implies the desired estimate \eqref{equ5.24}. Now the proof is completed.
\end{proof}

Now we are in the position to show the local Lipschitz continuity of ${\bf u}$.

\begin{proof}[Proof of Theorem \ref{thm1.4}]
  Let $B_R(x)\subset\Omega$. Let  $R_0$ be the constant given in Lemma \ref{lem5.5} with respect to $\Omega':=B_{R/2}(x)$.    

Take any $x_1\in   B_{R/2}(x)$.  If $d(x_1, \Omega\cap \{|{\bf u}|=0\})<R_0/8$,  then Lemma \ref{lem5.7} asserts 
  \begin{equation*}
    \Lip u_i(x_1)  \ls  C\sqrt{Q_{\max}} ,\qquad i=1,2,\cdots, m,
  \end{equation*}
 If $d(x_1, \Omega\cap \{|{\bf u}|=0\})\gs R_0/8$, that is,  $B_{R_0/10}(x_1)\subset \{|{\bf u}|>0\}$,  then Cheng-Yau's estimate, Theorem \ref{thm2.6}, asserts 
  \begin{equation*}
    \Lip u_i(x_1)  \ls  \frac{C}{R_0} \sup_{\Omega} |{\bf u}| ,\qquad i=1,2,\cdots, m.
  \end{equation*}
By summing up both cases and recalling Remark \ref{rem4.2}, we conclude that 
 there exists a constant $L$ depending only on $N,K,\Omega, R, Q_{\max}, \varepsilon_{\bf u}$ and $\int_{B_R(x)} |{\bf u}|{\rm d}\mu$, such that 
  \begin{equation}\label{equ5.27}
  \sup_{B_{R/2}(x)}\Lip u_i(x)\ls  L, \qquad \forall\ i=1,2,\cdots, m.
  \end{equation}
 
Take any $y,z\in B_{R/4}(x).$ Let $\gamma:[0,d(x,y)]\to \Omega$ be a geodesic from $y$ to $z$. The triangle inequality implies that $\gamma\subset B_{R/2}(x)$. Noted that $\Lip u_i $ is one of the upper gradient of $u_i$ (see \cite{Che99}), the   estimate (\ref{equ5.27}) implies that 
$$|u_i(y)-u_i(z)|\ls  \int_0^{d(y,z)}\Lip u_i\circ \gamma(s){\rm d}s\ls  L\cdot d(y,z),$$
for each $i=1, 2, \cdots, m$. The proof is finished.
\end{proof}

\section{Local finiteness of perimeter for the free boundary\label{sec:nondegeneracy}}

We  continue to assume   that $(X,d,\mu)$ is an $ncRCD(K,N)$ metric measure space with $K\ls 0$ and $N\in(1,+\infty)$. Let $\Omega\subset X$ be a bounded  domain and let $Q\in L^\infty(\Omega)$ satisfy (\ref{equ1.3}) for two positive numbers $Q_{\min}$ and $Q_{\max}$. Let
$${\bf u}=(u_1,u_2,\cdots, u_m)$$
be a local minimizer of $J_Q$ in (\ref{equ1.2}) with a boundary data ${\bf g}\in W^{1,2}(\Omega,[0,+\infty)^m)$, i.e., there exists $\varepsilon_{\bf u}>0$ such that $J_Q({\bf u})\ls  J_Q({\bf v})$ for all ${\bf v}\in\mathscr A_{\bf g}$ with $\dist({\bf u},{\bf v})<\varepsilon_{\bf u}$, where the $\mathscr A_{\bf g}$ and $\dist({\bf u},{\bf v})$ are given in (\ref{equ1.4}) and (\ref{equ1.5}), respectively.  
From Theorem \ref{thm1.4}, we know that ${\bf u}$ is locally Lipschitz continuous in the interior of  $\Omega$.  

We begin with the nondegeneracy of the local minimizer ${\bf  u}$ near the free boundary.  

\subsection{Nondegeneracy}

\begin{theorem}[Nondegeneracy] \label{thm6.1}
  Let ${\bf u}=(u_{1},\ldots,u_{m})$ be a local minimizer of $J_Q$ in \eqref{equ1.2}   with $Q$ satisfying  \eqref{equ1.3},  and let $\Omega'\Subset \Omega.$   Then there is a constant $R_1>0$ (depending only on $N,K,\Omega',Q_{\max},\varepsilon_{\bf u}$ and  the Lipschitz constant of ${\bf u}$ on $\Omega'$)  such that for any   ball $B_{r}(x_0)\subset\Omega$ with $x_0\in\partial\{\left|{\bf u}\right|>0 \}\cap \Omega'$ and  $r<R_1$, it holds
  \begin{equation}\label{equ6.1}
    \sup_{B_{r}(x_0)}\left|{\bf u}\right|\gs  c\sqrt{Q_{\min}}\cdot r\text{,}
  \end{equation}
  where the positive constant $c$   depends only on $N,K,\Omega',\varepsilon_{\bf u}$ and  the Lipschitz constant of ${\bf u}$ on $\Omega'$.
\end{theorem}

\begin{proof} In the Euclidean setting, this assertion was established in \cite[Theorem 3]{CSY18} and \cite{MTV17}. It was extended to smooth Riemannian manifolds in \cite{LS20}. Here we will extend their arguments to a nonsmooth setting. Without loss of the generality, we can assume that $K<0$.

 Fix  any $r\in (0,1)  $ and let   $M=\sup_{B_{r}(x_0) }\left|{\bf u}\right|$. Since ${\bf u}$ is Lipschitz continuous on $\Omega'$ (see Theorem \ref{thm1.4}) and ${\bf u}(x_0)=0$, we get $M\ls L$, the Lipschitz constant of ${\bf u}$ on $\Omega'$.  Given any $\theta\in (0,1) $, as in   \cite{CSY18}, we consider the map ${\bf v}=(v_{1},\ldots,v_{m}) $, where
  \begin{equation}\label{equ6.2}
      v_{i}(y) =\begin{cases}
      \min\left\{u_{i}(y) ,M\psi_{\theta}(\frac{\rho(y) }{r}) \right\} & \text{if }y\in B_{r}(x_0) \text{,}     \\
      u_{i}(y)                                                                             & \text{if }y\not\in B_{r}(x_0) \text{,} 
    \end{cases}
  \end{equation}
  for all $i=1,\ldots,m$ and $y\in\Omega$, 
where  $\rho(\cdot)=d(\cdot,x_0)$,
  \begin{equation}\label{equ6.3}  
    \psi_{\theta}(t)=\frac{(\phi_{N,K}(t)-\phi_{N,K}(\theta)) ^{+}}{\phi_{N,K}(1) -\phi_{N,K}(\theta) } 
  \end{equation}
and
  \begin{equation*} 
    \phi_{N,K}(s) =-\int^1_s\left(\frac{\sinh(\sqrt{-K/(N-1)}t) }{\sqrt{-K/(N-1)}}\right)^{1-N}\mathrm{d}t.
  \end{equation*}
 Then it is clear that ${\bf v}\in\mathscr{A}_{\mathbf{g}}$ and $u_i-v_i\in W^{1,2}_0(B_r(x_0))$ for all $i=1,2,\cdots, m$.
  
 (i) \ We first   check that $\dist({\bf v},{\bf u})<\varepsilon_{\bf u}$    provided both $r$ and $\theta$ are sufficiently small. 
The co-area formula gives 
 \begin{equation*} 
    \begin{split}
      \int_{B_r(x_0) \setminus B_{\theta r}(x_0) }\left|\nabla \psi_\theta(\frac{\rho}{r}) \right|^2{\rm d}\mu&\ls\int_{\theta r}^r\left|\frac 1 r\psi'_\theta(\frac{s}{r}) \right|^2\cdot \frac{d^+}{dr}\mu(B_s(x_0)){\rm d}s.
                 \end{split}
  \end{equation*}
%The Bishop-Gromov inequality implies  $ \frac{d^+}{ds}\mu(B_r(x)) \ls C_{N,K}\cdot\frac{\mu(B_s(x))}{s}$ for all $s<1$ (see in the argument of (\ref{equ2.2})). 
%From the condition $\mu(B_s(x))\ls C_{1}s^2$, we have  $ \frac{d^+}{ds}\mu(B_s(x)) \ls C_{2}s$. 
Since $(X,d,\mu)$ is non-collapsed,   substituting \eqref{equ2.13}  into the above inequality, we  obtain 
    \begin{equation}\label{equ6.4}
    \begin{split}
      \int_{B_r(x_0) \setminus B_{\theta r}(x_0) }\left|\nabla \psi_\theta(\frac{\rho}{r}) \right|^2{\rm d}\mu  %&\ls  C_{2}  \int_{\theta r}^r\left|\frac 1 r\psi'_\theta(\frac{s}{r}) \right|^2\cdot s^{N-1}  {\rm d}s\\
         &\ls  C_{1}r^{N-2}\cdot\left(\int_{\theta}^1  s^{1-N}   {\rm d}s\right)^{-1},  
           \end{split}
  \end{equation}
where  we have used $|\psi_{\theta}'(t)|\ls  C_{N,K}\frac{t^{1-N}}{|\phi_{N,K}(1)-\phi_{N,K}(\theta)|}\ls  \frac{ C'_{N,K}\cdot t^{1-N}}{\int_{\theta  }^1  s^{1-N}   {\rm d}s} $ for all $t\in(\theta,1)$.
Here and in the following of this proof, all constants $C_1,C_2,\cdots$ depend only on $N,K,\Omega'$.

Noticed that  $|\psi_\theta| \ls  1$ and $u_i-v_i=(u_i-M\psi_{\theta}(\frac{\rho}{r}) ) ^{+}\in W^{1,2}_0(B_r(x_0))$, we get
  \begin{equation}\label{equ6.5}  
    \int_{\Omega}\left|u_{i}-v_{i}\right|^{2}\ls \int_{B_{r}(x_0) }((u_{i}-M\psi_{\theta}(\frac{\rho}{r}) ) ^{+}) ^{2}\ls  M^{2}\mu(B_{r}(x_0) ) 
  \end{equation}
  and (by the fact that \cite{Che99}, for any $w\in W^{1,2}(\Omega)$, $|\nabla w^+|\ls  |\nabla w|$ holds almost everywhere in $\Omega$,)  
  \begin{equation}\label{equ6.6}
      \begin{aligned}
      \int_{\Omega}\left|\nabla(u_{i}-v_{i}) \right|^{2} & \ls \int_{B_{r}(x_0) }\left|\nabla(u_{i}-M\psi_{\theta}(\frac{\rho}{r}) ) \right|^{2} \\
       & \ls  2\int_{B_{r}(x_0) \setminus B_{\theta r}(x_0) }\left|\nabla(M\psi_{\theta}(\frac{\rho}{r}) ) \right|^{2}+2\int_{B_{ r}(x_0) }\left|\nabla u_i\right|^2\\
       & \ls  2 M^2 \int_{B_r(x_0) \setminus B_{\theta r}(x_0) }\left|\nabla \psi_\theta(\frac{\rho}{r}) \right|^2+ 2L^2\mu(B_r(x_0)),
           \end{aligned}
  \end{equation}
where for the second inequality we have used $\psi_\theta=0$ on $B_{\theta r}(x_0)$, and  for the last inequality  we have used $|\nabla u_i|\ls  L$.  Recall    $M\ls L$. 

From the combination of (\ref{equ6.4})--(\ref{equ6.6}), the fact  $M\ls L$,  ${\bf u}-{\bf v}\in W^{1,2}_0(B_r(x_0),\mathbb R^m)$, and taking $\theta $ such that
\begin{equation}\label{equ6.7}
 C_1\Big(\int_ \theta^1s^{1-N}{\rm d}s\Big)^{-1}=\frac{\varepsilon_{\bf u}}{8L^2},
 \end{equation}
 we conclude   that $\dist({\bf v},{\bf u})<\varepsilon_{\bf u}$   provided   $r<R_1$ for some small  number $R_1>0$ depending only on $N,K, L,Q_{\max}$ and $\varepsilon_{\bf u}$.
 
(ii)\  Fixed any $r\in (0,R_1)  $    and taken $\theta$ in (\ref{equ6.7}),
  the local minimality of ${\bf u}$  gives
  \begin{equation*}
    \begin{aligned}
                & \int_{B_{\theta r}(x_0) }(\left|\nabla{\bf u}\right|^{2}+Q\chi_{\left\{\left|{\bf u}\right|>0\right\}})                                                                                                                                                                                             \\
      =         & \int_{B_{r}(x_0) }(\left|\nabla{\bf u}\right|^{2}+Q\chi_{\left\{\left|{\bf u}\right|>0\right\}}) -\int_{B_{r}(x_0) \setminus B_{\theta r}(x_0) }(\left|\nabla{\bf u}\right|^{2}+Q\chi_{\left\{\left|{\bf u}\right|>0\right\}})                                      \\
      \ls  & \int_{B_{r}(x_0) }(\left|\nabla{\bf v}\right|^{2}+Q\chi_{\left\{\left|{\bf v}\right|>0\right\}}) -\int_{B_{r}(x_0) \setminus B_{\theta r}(x) }(\left|\nabla{\bf u}\right|^{2}+Q\chi_{\left\{\left|{\bf u}\right|>0\right\}})                                      \\
      =         & \int_{B_{r}(x_0) \setminus B_{\theta r}(x_0) }(\left|\nabla{\bf v}\right|^{2}+Q\chi_{\left\{\left|{\bf v}\right|>0\right\}}) -\int_{B_{r}(x_0) \setminus B_{\theta r}(x_0) }(\left|\nabla{\bf u}\right|^{2}+Q\chi_{\left\{\left|{\bf u}\right|>0\right\}}) ,         \end{aligned}
  \end{equation*}
where we used that ${\bf v}\llcorner B_{\theta r}(x_0) \equiv 0$. Noticed that $  (B_{r}(x_0) \setminus B_{\theta r}(x_0) )\cap  \{|{\bf u}|>0\}  = (B_{r}(x_0) \setminus B_{\theta r}(x_0) )\cap  \{|{\bf v}|>0\}$, this yields 
 \begin{equation}\label{equ6.8}
    \begin{aligned}
                 \int_{B_{\theta r}(x_0) }(\left|\nabla{\bf u}\right|^{2}+Q\chi_{\left\{\left|{\bf u}\right|>0\right\}})                                                                                                                                                                                            
         \ls         & \int_{B_{r}(x_0) \setminus B_{\theta r}(x_0) }(\left|\nabla {\bf v}\right|^{2}-\left|\nabla {\bf u}\right|^{2})                                                                                                                                                                             \\
      =         & \sum_{i=1}^{m}\int_{B_{r}(x_0) \setminus B_{\theta r}(x_0) }(\left|\nabla v_{i}\right|^{2}-\left|\nabla u_{i}\right|^{2}) \text{.}
    \end{aligned}
  \end{equation}
  Let $w_i=(u_{i}-M\psi_{\theta}(\frac{\rho}{r}) ) ^{+}$ for each $i=1,2,\cdots,m$. Then $u_i=v_i+w_i$ and 
  \begin{equation}\label{equ6.9}
    \begin{aligned}
                & \int_{B_{r}(x_0) \setminus B_{\theta r}(x_0) }(\left|\nabla v_{i}\right|^{2}-\left|\nabla u_{i}\right|^{2})     \\
      =         & \int_{B_{r}(x_0) \setminus B_{\theta r}(x_0) }\Big(-2\ip{\nabla v_i}{\nabla w_i}-\left|\nabla w_i\right|^{2}\Big)                                                                                  \\
            \ls       & -2\int_{B_{r}(x_0) \setminus B_{\theta r}(x_0) } \ip{\nabla v_i}{\nabla w_i}                                                                                \\
                  =         & -2\int_{B_{r}(x_0) \setminus B_{\theta r}(x_0) }\ip{\nabla(M\psi_{\theta}(\frac{\rho}{r}) ) }{\nabla w_i},
    \end{aligned}
  \end{equation}
  for each $i=1,\ldots, m$, 
where for the last equality, we have used that  $\left|\nabla w_i\right|=0 $ $\mu$-{a.e.} on $\left\{w_i=0\right\} \cap B_r(x_0)$ (see \cite{Che99} or Proposition \ref{prop2.1}(3)) and that $v_i=M\psi_{\theta}(\frac{\rho}{r})$ on $\{w_i\not=0\}\cap B_r(x_0)$.  Remark that $\{w_i\not=0\}$ is open by the continuity of $w_i$.

 (iii) \ Next we want to  estimate $\sum_{i=1,}^m I_{\theta r,r}(w_i),$ where we denote
  \begin{equation}\label{equ6.10}
    I_{r_1,r_2}(w_i):=- \int_{B_{r_2}(x_0) \setminus B_{ r_1}(x_0) }\ip{\nabla(M\psi_{\theta}(\frac{\rho}{r}) ) }{\nabla w_i}
  \end{equation}
  for any $r_1,r_2\in\left[\theta r,r\right]$ with $r_1<r_2$. 
     
By the Laplacian comparison theorem (see \eqref{equ2.7} in Theorem \ref{thm2.5}) and that the space $(M,\frac d r) $ satisfies $RCD(K\cdot r^2,N)$, we conclude that
  \begin{equation*}
    \laplace \psi_{\theta}(\frac{\rho}{r}) \ls  0\quad \mathrm{ on }\
    B_{r}(x_0) \setminus \overline{B_{\theta r}(x_0) } 
  \end{equation*}
 in the sense of distributions.   From Remark \ref{rem5.2}, we have for almost all $r_1,r_2\in (\theta r,r)$ with $r_1<r_2$, that
  \begin{equation}\label{equ6.11}
    \begin{aligned}
      I_{r_1,r_2}(w_i) = & \int_{B_{r_2}(x_0) \setminus B_{r_1}(x_0) }w_i{\rm d}\laplace(M\psi_{\rho}(\frac{\rho}{r}) )                                                                                                                                                             \\
                    & + M\frac{\psi_{\theta}'(\frac{r_1}{r}) }{r} \left.\frac{\mathrm{d}}{\mathrm{d}s}\right|_{s=r_1}\int_{B_{s}(x_{0}) }w_i - M\frac{\psi_{\theta}'(\frac{r_2}{r}) }{r} \left.\frac{\mathrm{d}}{\mathrm{d}s}\right|_{s=r_2}\int_{B_{s}(x_{0}) }w_i \\
      \ls      & C_{N,K,\theta}\frac{M}{r}\left.\frac{\mathrm{d}}{\mathrm{d}s} \right|_{s=r_1}\int_{B_{s}(x_{0}) }w_i,
    \end{aligned}
  \end{equation}
where we have used that $0\ls  \phi'_{N,K}(t) \ls  C_{N,K,\theta}$ for all $t\in(\theta,1)  $ and that $ \frac{\mathrm{d}}{\mathrm{d}s} \int_{B_{s}(x_{0}) }w_i \gs  0$ for almost all $s\in(\theta r,r) .$

For almost every $r_1\in (\theta r, r)$ such that both $s\mapsto\int_{B_{s}(x_0)}w_i$ and $s\mapsto\int_{B_{s}(x_0)}u_i$ are differentiable at $r_1$, we have
\begin{equation}\label{equ6.12}
\begin{split}
  \left.\frac{\mathrm{d}}{\mathrm{d}s}\right|_{s=r_1}\int_{B_{s}(x_0) }w_{i} -\left.\frac{\mathrm{d}}{\mathrm{d}s}\right|_{s=r_1}\int_{B_{s}(x_0) }u_{i}   
 = &\lim_{\delta\to0^+}\frac 1 \delta\int_{A_{r_1,r_1+\delta}}\big(u_i-w_i\big) \\
 \ls   &\lim_{\delta\to0^+}\frac 1 \delta\int_{A_{r_1,r_1+\delta}}\big|u_i-w_i\big|,
  \end{split}
\end{equation}
where $A_{r_1,r_1+\delta}:=B_{r_1+\delta}(x_0)\backslash B_{r_1}(x_0)$. 

 From the definition $w_i=(u_{i}-M\psi_{\theta}(\frac{\rho}{r}) ) ^{+}$, $M\ls L$ and   $|\psi'_\theta(t)|\ls  C_{N,K,\theta}$ in $(\theta,1)$, it follows  that $(u_i -w_i)$ is Lipschitz continuous  on $B_{r}(x_0)$ with a Lipschitz constant $C_{L,N,K,\theta}>0$.
By using $(u_i-w_i)(y)=0$ for any $y\in \partial B_{\theta r}(x_0)$, we conclude that
\begin{equation}
\sup_{A_{r_1,r_1+\delta}}\big|u_i-w_i\big|\ls  C_{L,N,K,\theta}\cdot (r_1+\delta-\theta r).
\end{equation}
 Substituting this into (\ref{equ6.12}), we have
 \begin{equation}\label{equ6.14}
\begin{split}
&  \left.\frac{\mathrm{d}}{\mathrm{d}s}\right|_{s=r_1}\int_{B_{s}(x_0) }w_{i} -\left.\frac{\mathrm{d}}{\mathrm{d}s}\right|_{s=r_1}\int_{B_{s}(x_0) }u_{i}   \\
\ls  &\lim_{\delta\to0^+}C_{L,N,K,\theta}\cdot (r_1+\delta-\theta r)\cdot \frac{\mu(B_{r_1+\delta}(x_0)\backslash B_{r_1}(x_0))}{\delta}\\
 \ls  & C'_{L,N,K,\theta,\Omega'}\cdot (r_1-\theta r),
 \end{split}
\end{equation} 
where for the last inequality we have used (\ref{equ2.2}).

  On the other hand, by using Lemma \ref{lem5.1} and Theorem \ref{thm2.5} (the Laplacian comparison theorem), we have  for almost all $r_1\in(\theta r, 3\theta r/2) $ that
    \begin{equation*}
    \begin{aligned}
      \left.\frac{\mathrm{d}}{\mathrm{d}s}\right|_{s=r_1}\int_{B_{s}(x_0) }u_{i} & =\frac{1}{2r_{1}}\left(\int_{B_{r_1}(x_0) }u_{i}{\rm d}\laplace(\rho^{2}) +\int_{B_{ r_1}(x_0) }\left\langle \nabla u_{i},\nabla \rho^{2}\right\rangle\right)         \\
                                                                                               & \ls  \frac{C_{N,K}}{2r_1}  \int_{B_{r_1}(x_0) }u_{i}+ \int_{B_{r_1}(x_0) }\left|\nabla u_{i}\right|                                 \\
                                                                                               & \ls  \frac{C_{N,K,\theta}}{r} \int_{B_{r_{1}}(x_0) }\left|{\bf u}\right|+ \int_{B_{r_{1}}(x_0) }\left|\nabla {\bf u}\right|,
    \end{aligned}
  \end{equation*}
where we have used $ u_i\gs 0$ and $r_1\gs\theta r$. This implies for     almost all $r_1\in(\theta r,3\theta r/2)$ that  
\begin{equation}\label{equ6.15}
    \begin{aligned}
                & \sum_{i=1}^{m}\left.\frac{\mathrm{d}}{\mathrm{d}s}\right|_{s=r_1}\int_{B_{s}(x_0) }u_{i}                                                                                                                                                                                  \\
      \ls  & C_{N,K,m,\theta}\int_{B_{r_{1}}(x_0) }\Big(\frac{1}{r}\left|{\bf u}\right|+\left|\nabla{\bf u}\right|\Big)                                                                                                                                         \\
      =  &  C_{N,K,m,\theta}\int_{B_{r_{1}}(x_0) }\Big(\frac{1}{r}\left|{\bf u}\right|\chi_{\left\{\left|{\bf u}\right|>0\right\}}+\left|\nabla{\bf u}\right|\chi_{\left\{\left|{\bf u}\right|>0\right\}}\Big)                                           \\
      \ls  
              &   C_{N,K,m,\theta}\int_{B_{r_{1}}(x_0) }\Big(\frac{M}{r}\chi_{\left\{\left|{\bf u}\right|>0\right\}}+\frac{1}{2\sqrt{Q_{\min}}}(Q_{\min}\chi^2_{\left\{\left|{\bf u}\right|>0\right\}}+\left|\nabla{\bf u}\right|^{2}) \Big)          \\
      \ls  &  C_{N,K,m,\theta}\int_{B_{r_{1}}(x_0) }\Big(\frac{M}{Q_{\min}r}Q\chi_{\left\{\left|{\bf u}\right|>0\right\}}+\frac{1}{2\sqrt{Q_{\min}}}(Q\chi_{\left\{\left|{\bf u}\right|>0\right\}}+\left|\nabla{\bf u}\right|^{2}) \Big)        \\
      \ls  &  C_{N,K,m,\theta}\Big(\frac{M}{Q_{\min}r}+\frac{1}{2\sqrt{Q_{\min}}}\Big) \int_{B_{r_{1}}(x_0) }\Big(Q\chi_{\left\{\left|{\bf u}\right|>0\right\}}+\left|\nabla{\bf u}\right|^{2}\Big) \text{,}
    \end{aligned}
  \end{equation}
  where for the second inequality we have used $|\nabla {\bf u}|=0$ $\mu$-a.e. in $\{|{\bf u}|=0\}$ (\cite{Che99}), and for the third inequality we have used  $\sup_{B_{r_{1}}(x_0) }\left|{\bf u}\right| \ls  M$.  From the combination of (\ref{equ6.11}), (\ref{equ6.14}), (\ref{equ6.15}) and the fact 
$$ \lim_{r_1\to\theta r} I_{\theta r,r_1}(w_i)=0,\qquad \lim_{r_2\to r}I_{r_2,r}(w_i)=0 ,$$
by letting $r_1\to \theta r$ and $r_2\to r$, we conclude 
  \begin{equation}\label{equ6.16}
    \begin{aligned}
& \sum_{i=1}^m I_{\theta r,r}(w_i)= \sum_{i=1}^{m} \Big(I_{\theta r,r_1}(w_i)+ I_{  r_1,r_2}(w_i)+I_{  r_2,r}(w_i) \Big)                                                                                              \\
                                         \ls  &\lim_{r_1\to \theta r}\sum_{i=1}^{m} C_{N,K,\theta}\frac{M}{r}\left.\frac{\mathrm{d}}{\mathrm{d}s}\right|_{s=r_1}\int_{B_{s}(x_{0}) }u_{i}\\
                                          \ls   & C'_{N,K,m,\theta}\frac{M}{r}\Big(\frac{M}{Q_{\min}r}+\frac{1}{2\sqrt{Q_{\min}}}\Big) \int_{B_{\theta r}(x_0) }\Big(Q\chi_{\left\{\left|{\bf u}\right|>0\right\}}+\left|\nabla{\bf u}\right|^{2}\Big) .
                                               \end{aligned}
  \end{equation}
 
(iv)\  At last, the assumption $x_0\in\partial\{\left|{\bf u}\right|>0\}$ and the continuity of ${\bf u}$ implies
  \begin{equation*}
    \int_{B_{\theta r}(x_0) }\Big(\left|\nabla{\bf u}\right|^{2}+Q\chi_{\left\{\left|{\bf u}\right|>0\right\}}\Big) \gs  Q_{\min}\cdot\mu\big( B_{\theta r}(x_0)\cap\{|{\bf u}|>0\}\big)>0.
  \end{equation*}
Thus, by combining this with (\ref{equ6.8}), (\ref{equ6.9}), (\ref{equ6.10}) and  (\ref{equ6.16}), we obtain
  \begin{equation*}
    1\ls  C'_{N,K,m,\theta}\frac{2M}{r  }\Big(\frac{M}{  Q_{\min} \cdot r}+\frac{1}{2\sqrt{Q_{\min}} }\Big).
      \end{equation*}
 This implies 
  \begin{equation*}
    \frac{M}{\sqrt{Q_{\min}}\cdot r}  \gs \min\left\{\frac{1}{2},\frac{1}{2C'_{N,K,m,\theta}}\right\}\text{.}
  \end{equation*}
  This  is the desired estimate (\ref{equ6.1}), since $M=\sup_{B_{r}(x) }\left|{\bf u}\right|$ and $\theta$ is  given in (\ref{equ6.7}).  The proof is finished.
\end{proof}

\begin{remark}\label{rem6.2} If ${\bf u}$ is an absolute minimizer of $J_Q$,  the previous proof (step (ii)--(iv)) still works for general $RCD(K,N)$-spaces.  That is:

\emph{ Let $(X,d,\mu)$ be an $RCD(K,N)$-space with some $K<0$ and $N\in(1,+\infty)$, and let  ${\bf u}$ be an  absolute minimizer of $J_Q$ in  \eqref{equ1.2}. Let  $\Omega'\Subset \Omega.$  If ${\bf u}$ is Lipschitz continuous on $\Omega'$, then for any  ball $B_{r}(x_0)\subset\Omega$ with $x_0\in\partial\{\left|{\bf u}\right|>0 \}\cap \Omega'$, it holds \eqref{equ6.1} for a positive constant $c$ depending only on $m,N$ and $K$.}
\end{remark} 

\subsection{Density estimates near the free boundary}  In the subsection, we will show that both $\{|{\bf u}|>0\}$ and $\{|{\bf u}|=0\}$ have positive density along the free boundary. 
 
    \begin{lemma}\label{lem6.3}Let ${\bf u}=(u_{1},\ldots,u_{m})$ be a local minimizer of $J_Q$ in \eqref{equ1.2}   with $Q$ satisfying  \eqref{equ1.3}, and let $\Omega'\Subset \Omega.$      Then for any ball $B_r(x_0)\subset \Omega'$ with $x_0\in \Omega'\cap\partial\{|{\bf u}|>0\}$ and  $r<2R_1$, where $R_1$ is given in Theorem \ref{thm6.1}, we have  
 \begin{equation} \label{equ6.17}
 \mu\big(B_r(x_0)\cap \{|{\bf u}|>0\}\big)\gs  c_1\cdot \mu(B_r(x_0))
 \end{equation} 
for some constant $c_1$ depending on $m,N,K,\Omega', Q_{\min}, \varepsilon_{\bf u}$ and $L$, the Lipschitz constant of ${\bf u}$ on $\Omega'$.  \end{lemma}
 \begin{proof}
 From Theorem \ref{thm6.1}, there exists some $i_0\in\{1,2,\cdots, m\} $ such that 
 $$\sup_{B_{r/2}(x_0)}u_{i_0}\gs    \frac{c\sqrt{Q_{\min}}}{m}\cdot r/2$$
for all  $r/2<R_1$, where $R_1$ and $c$ are given in Theorem \ref{thm6.1}.   Choose $y_0\in B_{r/2}(x_0)$ such that $u_{i_0}(y_0)\gs  c_2r/2$, where $c_2:= \frac{c\sqrt{Q_{\min}}}{2m}$.  Since $u_{i_0}$ is Lipschitz continuous on $B_{r/2}(y_0)\subset\Omega'$ with a Lipschitz constant $L$, we have
$$\inf_{B_{c_3r}(y_0)}u_{i_0}\gs u_{i_0}(y_0)-c_3L r\gs  c_2r/4, \qquad c_3:=\min\big\{\frac 1 2,\frac{c_2}{8L}\big\}.$$
In particular, this yields $B_{c_3r}(y_0)\subset \{|{\bf u}|>0\}\cap B_r(x_0)$.
 It follows 
$$\mu\big(B_r(x_0)\cap \{|{\bf u}|>0\}\big)\gs  \mu(B_{c_3r }(y_0)).$$
By combining this and the Bishop-Gromov inequality 
$$\mu(B_{c_3r}(y_0))\gs  c_{N,K}{(c_3/2)^N} \mu(B_{2r}(y_0))\gs  c_{N,K}{(c_3/2)^N} \mu(B_{r}(x_0)),$$
we get the desired estimate (\ref{equ6.17}).
\end{proof}

 \begin{lemma}\label{lem6.4}Let ${\bf u}=(u_{1},\ldots,u_{m})$ be a local minimizer of $J_Q$ in \eqref{equ1.2}   with $Q$ satisfying  \eqref{equ1.3},  and let $\Omega'\Subset \Omega.$ 
 Then there exists a number $R_0\in (0,1)$  depending only on $N,K,\Omega',Q_{\max}$ and ${\bf u}$, such that  for any ball $B_r(x_0)\subset \Omega'$ with $x_0\in \Omega'\cap\partial\{|{\bf u}|>0\}$ and  $r<R_0$, we have  
 \begin{equation}\label{equ6.18}
 \mu\big(B_r(x_0)\cap \{|{\bf u}|=0\}\big)\gs  c_4\cdot \mu(B_r(x_0))
 \end{equation} 
for some constant $c_4$ depending on $m,N,K,\Omega', Q_{\min},Q_{\max},\varepsilon_{\bf u}$ and the Lipschitz constant of ${\bf u}$ on $\Omega'$.
 \end{lemma}
 
 \begin{proof}
 
Let ${\bf v}=(v_1,v_2,\cdots,v_m)\in W^{1,2}(B_{r}(x_{0}),\mathbb R^m) $ be the map such that each component  $v_i$ is  the (unique) solution of  the following (relaxed) Dirichlet problem \cite[Theorem 7.12]{Che99}:
  \begin{equation} \label{equ6.19}
      \begin{cases}
      \laplace v_i=0 \quad \text{ on }\ B_{r}(x_{0}), \\
      v_i-u_i\in W_{0}^{1,2}(B_{r}(x_{0})).
    \end{cases}
  \end{equation}
  After extending ${\bf v}$ by ${\bf u}$ on $\Omega\backslash B_{r}(x_{0})$, we have
  ${\bf v}\in\mathscr{A}_{\mathbf{g}}$ because all $v_i\gs0$
  on $B_{R}(x_{0})$ by the maximum principle (see \cite[Theorem 7.17]{Che99}).
In the proof of  Lemma \ref{lem4.3}, it was showed that there is $R_0\in (0,1)$ (depending only on $N,K,\Omega'$ and ${\bf u}$) such that $\dist({\bf u},{\bf v})<\varepsilon_{\bf u}$ for all $r\in(0,R_0)$.

The local minimality of ${\bf u}$ implies  
  \begin{equation}    \label{equ6.20}                
     \int_{B_{r}(x_{0}) }(|\nabla{\bf u}|^2-|\nabla{\bf v}|^2)  \ls  \int_{B_{r}(x_{0}) }Q(\chi_{\left\{\left|{\bf v}\right|>0\right\}}-\chi_{\left\{\left|{\bf u}\right|>0\right\}}) 
       \end{equation}
   for all $r\in (0,R_0)$. From the harmonicity of ${\bf v}$, the same argument in (\ref{equ4.9}) gives
     \begin{equation}    \label{equ6.21}                
     \int_{B_{r}(x_{0})}(|\nabla{\bf u}|^2-|\nabla{\bf v}|^2)  =     \int_{B_{r}(x_{0})}\left|\nabla({\bf u}-{\bf v})\right|^2.
           \end{equation}
Note that $|{\bf v}|$ is not  identical zero on $B_r(x_0)$ by Theorem \ref{thm6.1}. We get  from the strong maximum principle  (see, for example, \cite[Corollary 6.4]{KS01})  that  $|{\bf v}|>0$ on $B_r(x_0)$. Hence,
         \begin{equation}    \label{equ6.22}    
         \begin{split}            
   \int_{B_{r}(x_{0})}Q(\chi_{\left\{\left|{\bf v}\right|>0\right\}}-\chi_{\left\{\left|{\bf u}\right|>0\right\}})  &=\int_{B_r(x_0)}Q\chi_{\left\{\left|{\bf u}\right|=0\right\}}\\
   &\ls  Q_{\max}\cdot\mu\big(B_r(x_0)\cap\{|{\bf u}|=0\}\big).
   \end{split}
          \end{equation}   
   By combining (\ref{equ6.20})--(\ref{equ6.22}) and the Poincar\'e inequality (see Proposition \ref{prop2.1}(5)), we get
        \begin{equation}    \label{equ6.23}                
  \int_{B_{r}(x_{0})}|{\bf u}-{\bf v}|^2 \ls  C_P\cdot r^2\cdot Q_{\max}\cdot\mu\big(B_r(x_0)\cap\{|{\bf u}|=0\}\big), 
           \end{equation}  
           where the Poincar\'e constant $C_P$ depends only on $N,K$ and $\Omega'$.  
           
 From the  nondegeneracy Theorem \ref{thm6.1}, there exists some $i_0\in\{1,2,\cdots, m\}$ such that
 $\sup_{B_{r/2} (x_0)}u_{i_0}  \gs  c\sqrt{Q_{\min}}r/2.$      Recalling ${\bf \Delta}(u_{i_0}-v_{i_0})\gs  0$ on $B_r(x_0)$ in the sense of distributions, the maximum principle implies that $v_{i_0}\gs  u_{i_0}$ on $B_r(x_0)$. Hence, from this and the Harnack inequality, we have
 \begin{equation}\label{equ6.24}
 \inf_{B_{r/2}(x_0)}v_{i_0}\gs  C_1  \sup_{B_{r/2}(x_0)}v_{i_0}  \gs    C_1  \sup_{B_{r/2}(x_0)}u_{i_0}\gs  C_1c\sqrt{Q_{\min}}r/2,
 \end{equation}
 where the  constant $C_1$ depends only on $N,K,\Omega'.$ Since $u_{i_0}(x_0)=0$ and that $u_{i_0}$ is Lipschitz continuous on $B_r(x_0)\subset \Omega'$ with a Lipschitz constant $L$, we have
 $$\sup_{B_{c_1r}(x_0)}u_{i_0}\ls  Lc_1r,\quad {\rm with}\ \ c_1:=\min\Big\{\frac{C_1c\sqrt{Q_{\min}}}{4L},\frac 1 4\Big\}.$$
Combining this with (\ref{equ6.24}), we conclude that
 \begin{equation*}
\inf_{B_{c_1r}(x_0)}  (v_{i_0}-u_{i_0})\gs  \frac{C_1c\sqrt{Q_{\min}} }{4}r: =C_2\sqrt{Q_{\min}} \cdot r,
\end{equation*}
where the constant $C_2$ depends on $m, N, K, \Omega',  \varepsilon_{\bf u}$ and $L$. 
This yields 
 \begin{equation}\label{equ6.25}
 \int_{B_{c_1r}(x_{0})}|{\bf u}-{\bf v}|^2\gs   \int_{B_{c_1r}(x_0)}  (v_{i_0}-u_{i_0})^2\gs  C_2^2 Q_{\min}\cdot r^2\cdot\mu(B_{c_1r}(x_0)).
\end{equation} 
From (\ref{equ6.23}), (\ref{equ6.25}) and that $\mu(B_{c_1r}(x_0))\gs  C_{N,K}c_1^N\cdot \mu(B_r(x_0))$, it follows 
$$C_P Q_{\max}\cdot \mu\big(B_r(x_0)\cap\{|{\bf u}|=0\}\big)\gs  C_2^2 Q_{\min}C_{N,K}c_1^N\cdot  \mu(B_r(x_0)).$$
This is the desired \eqref{equ6.18} with  the constant $c_4:=C_P^{-1}C^2_2C_{N,K}c_1^N\cdot Q_{\min}/Q_{\max}$. The proof is finished.
\end{proof}

\subsection{Local finiteness of perimeter}

In this subsection, we will derive the local finiteness of the perimeter of the free boundary in the sense of \cite{Amb01,Mir03,ABS19}, via the estimates of density, Lemma \ref{lem6.3} and Lemma \ref{lem6.4}, and a similar argument in \cite{MTV20}. 
 We need an estimate on the perimeter, which is similar to the one in \cite[Lemma 2.4]{MTV20}.

\begin{lemma}\label{lem6.5} Let $(X,d,\mu)$ be an $RCD(K,N)$ space with $K\in \mathbb R$ and $N\in(1,+\infty)$, and let $D\subset X$ be an open subset.
  Suppose that  $0\ls \psi\in\Lip_{\mathrm{loc}}(D) \cap W^{1,1}(D) $.
  If there are positive constants   $\bar{\varepsilon}$  and $C$ such that
  \begin{equation}\label{equ6.26}
    \int_{\left\{0<\psi< \varepsilon\right\}\cap D}\left|\nabla \psi\right|  \ls  C\varepsilon
  \end{equation}
  for all $\varepsilon\in(0,\bar{\varepsilon}) $, then
  \begin{equation}\label{equ6.27}
    \mathcal{P}(\left\{\psi>0\right\},D) \ls 2 C\text{.}
  \end{equation}
\end{lemma}
\begin{proof}
  Given $\varepsilon\in(0,\bar{\varepsilon}/2)$,
  then, by the coarea formula \cite[Theorem 3.3]{Amb01} (see also Proposition \ref{prop2.15}(2)) to $\chi_D$, we have
    \begin{equation}\label{equ6.28}
    \begin{aligned}
      \int_{\varepsilon}^{2\varepsilon}\mathcal{P}(\left\{\psi>t\right\},D) \mathrm{d}t & =\int_{\varepsilon\ls \psi<2\varepsilon}\chi_D{\rm d}|D\psi|   = \int_{\left\{\varepsilon\ls\psi<  2\varepsilon\right\}\cap D}\left|\nabla\psi\right|                                                                                                           \\
      & \ls \int_{\left\{0<\psi< 2 \varepsilon\right\}\cap D}\left|\nabla\psi\right|                                                                                                                                                                                                 \ls 2C\varepsilon.
    \end{aligned}
  \end{equation}
 For each  integer $j$ so  large that $2^{-j}\ls  \bar \varepsilon/2$,  by setting $\varepsilon_j=2^{-j}$,
  there exists some $\varepsilon'_{j}\in(\varepsilon_j,2\varepsilon_j) $ such that
  \begin{equation}\label{equ6.29}
    \mathcal{P}(\left\{\psi>\varepsilon'_{j}\right\},D) \ls \varepsilon_j^{-1}\int_{\varepsilon_j}^{2\varepsilon_j}\mathcal{P}(\left\{\psi>t\right\},D) \mathrm{d}t\ls \varepsilon_j^{-1}{2C\varepsilon_j} \ls {2C}\text{.}
  \end{equation}
It is clear that  
 $   \chi_{\left\{\psi>\varepsilon'_{j}\right\}}\xrightarrow{L_{\mathrm{loc}}^{1}(D) }\chi_{\left\{\psi>0\right\}},$ by Lebesgue's dominated converge theorem. 
 From (\ref{equ6.29}) and the lower semicontinuity of perimeter \cite{Amb01}, it follows   that
  \begin{equation}\label{equ6.30}
    \mathcal{P}(\left\{\psi>0\right\},D) \ls  \liminf_{j\to\infty}\mathcal{P}(\left\{\psi>\varepsilon'_{j}\right\},D) \ls  {2C}\text{.}
  \end{equation}
This finishes the proof.  \end{proof}

Now let us check the condition \eqref{equ6.26} of the above lemma for the local minimizers of $J_Q$. 
\begin{lemma}\label{lem6.6}
  Let ${\bf u}=(u_{1},\ldots,u_{m}) $ be a local minimizer of $ J_Q$ in \eqref{equ1.2}   with $Q$ satisfying  \eqref{equ1.3}, and let $\Omega'\Subset \Omega$.   Then there exist $\bar r>0,\bar{\varepsilon}>0$ and $C>0$ such that  
  \begin{equation}\label{equ6.31}
   \int_{\left\{\left|{\bf u}\right|\ls \varepsilon\right\}\cap B_{r/2}(x_0) }  \left|\nabla{\bf u}\right|^{2}+ \mu\big( \{0<|{\bf u}|\ls  \varepsilon\}\cap B_{r/2}(x_0) \big)\ls  C\mu\big(B_r(x_0)\big)\cdot\frac{\varepsilon}{r^2}
     \end{equation}
  for all  $\varepsilon\in (0,\bar\varepsilon)$, $r\in (0,\bar r)$ and $B_r(x_0)\subset \Omega'$.  Here the constants $\bar r, \bar\varepsilon, C$  depend only on $N,K,\Omega', {\rm diam}(\Omega),   Q_{\min},\varepsilon_{\bf u}$ and  $L$, the Lipschitz constant of ${\bf u}$ on $\Omega'$. 
    \end{lemma}

\begin{proof}  Fix any ball $B_r(x_0)\subset\Omega'$ such that $r\in (0,\bar r)$, where $\bar r<1$ will be determined later.  Let  $\phi:M\to\left[0,1\right]$ be a cutoff function with 
  $\phi\equiv 1$ on $B_{r/2}(x_{0}) $, $\supp(\phi) \subset B_{r}(x_{0}) $,
  and $r\left|\nabla\phi\right|\ls  C_{N}$.
For each $\varepsilon>0$, we  set
  \begin{equation}\label{equ6.32}
    {\bf v}:  =\begin{cases}
      (1-\phi) {\bf u}                                            & \text{if }\left|{\bf u}\right|\ls \varepsilon\text{,} \\
      (1-\varepsilon\frac{\phi}{\left|{\bf u}\right|}) {\bf u} & \text{if }\left|{\bf u}\right|>\varepsilon\text{.}
    \end{cases}
  \end{equation}
  It is easy to check ${\bf v}\in\mathscr A_{\bf g}$ and ${\bf v}-{\bf u}\in W^{1,2}_0(B_r(x_0),\mathbb R^m)$. 
  
  (i) We  will first show that $d({\bf u},{\bf v})<\varepsilon_{\bf u}$ provided both $r$ and $\varepsilon$ are sufficiently small. By the Poincar\'e inequality (Proposition \ref{prop2.1}(5)) and ${\bf v}-{\bf u}\in W^{1,2}_0(B_r(x_0),\mathbb R^m)$, we have
    \begin{equation}\label{equ6.33}
 \dist({\bf u},{\bf v})   \ls   C_P\Big(\int_{B_r(x_0)}|\nabla({\bf u}-{\bf v})|^2\Big)^{1/2}+  \mu\big(B_r(x_0)\big),
  \end{equation}
where $C_P$ depends on $N,K$ and ${\rm diam}(\Omega).$
From ${\bf u}-{\bf v}=\phi {\bf u} $ on $\{|{\bf u}|\ls  \varepsilon\}\cap B_r(x_0)$ and \cite[Corollary 2.25]{Che99}, we have 
 \begin{equation*} 
 |\nabla ({\bf u}-{\bf v})|=|\nabla (\phi{\bf u})|\ls   L+\varepsilon\cdot \frac{C_N}{r},\quad \mu{\rm-a.e.\ \ on}\ \ \{|{\bf u}|\ls \varepsilon\}\cap B_r(x_0), 
 \end{equation*}
 where $L$ is a Lipschitz constant of ${\bf u}$ on $\Omega'$,  and  we have used $|\nabla\phi|\ls  \frac{C_N}{r}.$ Thus,
 \begin{equation}\label{equ6.34}
\int_{B_r(x_0)\cap\{|{\bf u}|\ls \varepsilon\}} |\nabla ({\bf u}-{\bf v})|^2\ls  \big(L+C_N\cdot\frac{\varepsilon}{r}\big)^2\cdot\mu(B_r(x_0)).
\end{equation}
On $\{|{\bf u}|>\varepsilon\}\cap B_r(x_0)$, we denote by $h:=\frac{\varepsilon\phi}{|{\bf u}|}.$ Then, from the Chain rule (see \cite{Gig15}) it follows
\begin{equation*}
|\nabla h|\ls  \frac{\varepsilon}{|\bf u|}\cdot \frac{C_N}{r}+\frac{\varepsilon\phi|\nabla |{\bf u}| |}{|{\bf u}|^2}\ls \frac{\varepsilon}{|\bf u|}\cdot \frac{C_N}{r}+\frac{L}{|{\bf u}|} 
\end{equation*}
$\mu{\rm-a.e.}$  on $\{|{\bf u}|>\varepsilon\}\cap B_r(x_0)$, where we have used $|\nabla |{\bf u}| |\ls |\nabla {\bf u}|\ls L$, $\varepsilon\phi\ls  |{\bf u}|$ and $|\nabla\phi|\ls  \frac{C_N}{r}.$ 
From this and ${\bf u}-{\bf v}=h{\bf u} $ on $\{|{\bf u}|> \varepsilon\}$, we have by \cite[Corollary 2.25]{Che99}, that 
 $$|\nabla ({\bf u}-{\bf v})|=|\nabla (h{\bf u})|\ls   L+\varepsilon\cdot \frac{C_N}{r}+h|\nabla {\bf u}|\ls   2L+\varepsilon\cdot \frac{C_N}{r},  $$
$\mu{\rm-a.e.}$  on $\{|{\bf u}|>\varepsilon\}\cap B_r(x_0)$.    Thus, we obtain
 \begin{equation}\label{equ6.35}
\int_{B_r(x_0)\cap\{|{\bf u}|>\varepsilon\}} |\nabla ({\bf u}-{\bf v})|^2\ls  \big(2L+C_N\cdot\frac{\varepsilon}{r}\big)^2\cdot\mu(B_r(x_0)).
\end{equation}
 From (\ref{equ6.33}), (\ref{equ6.34}) and (\ref{equ6.35}), it follows that 
  \begin{equation}\label{equ6.36}
  \begin{split}
  \dist({\bf u},{\bf v})&  \ls  C_1\Big(L+\frac{ \varepsilon}{r}\Big)\cdot \sqrt{\mu(B_r(x_0))}+ \mu(B_r(x_0)),
    \end{split}
  \end{equation}
where $C_1$ depending only on $N,K$ and ${\rm diam}(\Omega)$.   By the non-collapsing property of $X$, we have  $N\gs 2$ (since $N$ is an integer and $N>1$), and then   $\mu(B_r(x_0))\ls  C_{2} r^N\ls C_{2}r^2$ for $r<1$, where $C_2$ depends only on $N,K$ (see \eqref{equ2.12}).  Therefore, we conclude    $  \dist({\bf u},{\bf v})  <\varepsilon_{\bf u} $  provided 
\begin{equation}\label{equ6.37}
\varepsilon\ls  \bar\varepsilon:=\frac{\varepsilon_{\bf u}}{4C_1C_{2}^{1/2}} \qquad  r\ls  \bar{r}:=\min\Big\{1,\frac{\varepsilon_{\bf u}}{4(C_1LC^{1/2}_{2}+ C_{2})}\Big\}.
\end{equation}
  
(ii) Fix any $r\in (0,\bar r)$ and $\varepsilon\in(0,\bar{\varepsilon})$ given in (\ref{equ6.37}).  From the local minimality of ${\bf u}$, we obtain
  
    \begin{equation}\label{equ6.38}
    \begin{aligned}
                \int_{B_{r}(x_{0}) }\big(|\nabla{\bf u}|^{2}-|\nabla {\bf v}|^2\big) &\ls  \int_{B_r(x_0)}Q\big(\chi_{\{|{\bf v}|>0\}}-\chi_{\{|\bf u|>0\}}\big)  \\
    &\ls  \int_{B_{r/2}(x_0)\cap \{|{\bf u}|\ls \varepsilon\}}Q\big(\chi_{\{|{\bf v}|>0\}}-\chi_{\{|\bf u|>0\}}\big)  \\
                         &\ls  - \int_{B_{r/2}(x_0) }Q\chi_{\{0<|\bf u|\ls  \varepsilon\}}\\
                         &\ls -Q_{\min}\cdot \mu\big( B_{r/2}(x_0)\cap\{ 0<|{\bf u}|\ls \varepsilon\}\big), 
                            \end{aligned}
  \end{equation}  
where for the second inequality we have used  $\chi_{\{|{\bf v}|>0\}}- \chi_{\{|\bf u|>0\}}\ls  0$ on $B_r(x_0)$, and for the third inequality we have used ${\bf v}=0 $ on $B_{r/2}(x_0)\cap \{|{\bf u}|\ls  \varepsilon\}$.

  On the one hand,  from ${\bf v}=(1-\phi) {\bf u} $ on $\{|{\bf u}|\ls  \varepsilon\}\cap B_r(x_0)$ and \cite[Corollary 2.25]{Che99}, 
  it follows $|\nabla {\bf v}|=|\nabla((1-\phi){\bf u})|$ $\mu$-a.e. in $\{|{\bf u}|\ls  \varepsilon\}\cap B_r(x_0)$. Thus, we have  
    {\small
  \begin{equation}\label{equ6.39}
    \begin{aligned}
                & \int_{\left\{|{\bf u}|\ls \varepsilon\right\}\cap B_{r}(x_{0})}(\left|\nabla{\bf u}\right|^{2}-\left|\nabla{\bf v}\right|^{2})                                                                                                                       \\
      =         & \int_{\left\{\left|{\bf u}\right|\ls \varepsilon\right\}\cap B_{r}(x_{0})}(\left|\nabla{\bf u}\right|^{2}-\left|\nabla((1-\phi) {\bf u}) \right|^{2})                                                                                       \\
      =         & \int_{\{|{\bf u}|\ls \varepsilon\}\cap B_{r}(x_{0})}((2\phi-\phi^{2})|\nabla{\bf u}|^{2}+2(1-\phi) \sum_{i=1}^{m}u_{i}\ip{\nabla\phi}{\nabla u_{i}}-\left|\nabla\phi\right|^{2}\left|{\bf u}\right|^{2})  \\
      \gs  & \int_{\left\{\left|{\bf u}\right|\ls \varepsilon\right\}\cap B_{r}(x_{0})}\big(\chi_{B_{r/2}(x_{0}) }\left|\nabla{\bf u}\right|^{2}\big)-\int_{\left\{\left|{\bf u}\right|\ls \varepsilon\right\}\cap B_{r}(x_{0})}\big(2\frac{C_NL}{r}|{\bf u}|+\frac{C^2_N}{r^2} |{\bf u}|^2\big)\\
      \gs   & \int_{\left\{\left|{\bf u}\right|\ls \varepsilon\right\}\cap B_{r/2}(x_{0})}\left|\nabla{\bf u}\right|^{2}-C_{N,L}\frac{\varepsilon}{r^2}\mu\big(B_r(x_0)\big),      \end{aligned}
  \end{equation}} 
where for the first inequality we have used  $(2\phi-\phi^2)=1$ on $B_{r/2}(x_0)$,  $|\nabla {\bf u}|\ls  L$ and $|\nabla \phi|\ls  C_N/r$ on $B_r(x_0)$.

   On the other hand, we denote $h=\frac{\varepsilon\phi}{|{\bf u}|}$. Then ${\bf v}=(1-h) {\bf u} $ on $\{|{\bf u}|>\varepsilon\}\cap B_r(x_0)$,  and by \cite[Corollary 2.25]{Che99}, 
 we have that, for  $\mu$-a.e. in $\{|{\bf u}|> \varepsilon\}\cap B_r(x_0)$, 
  \begin{equation*}\begin{split}
 &   |\nabla {\bf u}|^2-  |\nabla {\bf v}|^2=|\nabla {\bf u}|^2 -|\nabla((1-h){\bf u})|^2   \\
     =&(2h - h^2) |\nabla {\bf u}|^2-|\nabla h|^2|{\bf u}|^2+2(1-h)\sum_{i=1}^m u_i\ip{\nabla h}{\nabla u_i}\\
     =& (2h- h^2) |\nabla {\bf u}|^2-|\nabla h|^2|{\bf u}|^2+(1-h) \ip{\nabla h}{\nabla (|{\bf u}|^2)}\\
          =& (2h- h^2) |\nabla {\bf u}|^2-|\nabla h|^2|{\bf u}|^2+2(1-h)|{\bf u}| \ip{\nabla h}{\nabla |{\bf u}|}\\
               =& (2h- h^2) |\nabla {\bf u}|^2-  |\nabla (h|{\bf u}|) |^2+2\ip{\nabla |{\bf u}|}{\nabla (h|{\bf u}|)}+(h^2-2h) |\nabla|{\bf u}||^2  \\ 
      =& (2h-h^2)\big( |\nabla {\bf u}|^2- |\nabla|{\bf u}||^2\big) -   |\nabla (\varepsilon\phi) |^2+2\ip{\nabla |{\bf u}|}{\nabla (\varepsilon\phi)}.
     \end{split}
  \end{equation*}
  Thus, by $|\nabla {\bf u}|\gs  |\nabla|{\bf u}||$ $\mu$-almost everywhere and $h\ls  1$, we have  
  \begin{equation}\label{equ6.40}
    \begin{aligned}
                &  \int_{\left\{\left|{\bf u}\right|>\varepsilon\right\}\cap B_{r}(x_{0}) } (\left|\nabla{\bf u}\right|^{2}-\left|\nabla{\bf v}\right|^{2})                                                                                                                                                                                             \\
      \gs  & \int_{\left\{\left|{\bf u}\right|>\varepsilon\right\}\cap B_{r}(x_{0}) }\big(-\varepsilon^{2}\left|\nabla\phi\right|^{2}+2\varepsilon\ip{\nabla|{\bf u}|}{\nabla\phi}\big)                                                                                                                                                          \\
      \gs  & \big(-\varepsilon^2\frac{C^2_N}{r^2}-2\varepsilon\frac{C_NL}{r}\big)\mu\big(B_r(x_0)\big)\gs  -C_{N,L}\frac{\varepsilon}{r^2} \mu\big(B_r(x_0)\big),                                                                                                                                                                                                                              
    \end{aligned}
  \end{equation}
  where we have used $|\nabla {\bf u}|\ls  L$ and $|\nabla \phi|\ls  C_N/r$. The inequality (\ref{equ6.31}) follows from the combination of (\ref{equ6.38})--(\ref{equ6.40}). The proof is finished.
\end{proof}

Now we are in the position to show the free boundary of ${\bf u}$ is a set of locally finite perimeter in the sense of \cite{Amb01,Mir03,ABS19}.
\begin{proposition}[Local finiteness of perimeter]\label{prop6.7}
  Let ${\bf u}=(u_{1},\ldots,u_{m}) $ be a local minimizer of $J_Q$ in \eqref{equ1.2}  and let $Q$ satisfy\eqref{equ1.3}. Then $\Omega_{\bf u}=\Omega\cap\{|{\bf u}|>0\}$ is of locally finite perimeter. Moreover,  the followings hold: 
 \begin{enumerate} 
 \item   For all   $\Omega'\Subset\Omega$, 
$ \mathscr{H}^{N-1}\big(\partial\{|{\bf u}|>0\}\cap \Omega'\big)  <+\infty$; 
\item There exist nonnegative Borel functions $q_i, i=1,2,\cdots, m,$ such that 
$${\bf\Delta} u_i=q_i\cdot \mathscr H^{N-1}\llcorner (\partial\{|{\bf u}|>0\}\cap \Omega).$$
 \end{enumerate}  
 \end{proposition}
\begin{proof}
Fix any $\Omega'\Subset \Omega$. By using   Lemma \ref{lem6.6}, we get
\begin{equation*}
\begin{split}
   \int_{\left\{0<|{\bf u}|\ls \varepsilon\right\}\cap B_{r/2}(x_0) }  |\nabla{\bf u}|&\ls    \int_{\left\{0< |{\bf u}|\ls \varepsilon\right\}\cap B_{  r/2}(x_0) }  \frac{|\nabla{\bf u}|^{2}+ 1}{2}\\
   &\ls  C_{N,K, L, r,\mu(B_{r}(x_0))}\cdot\varepsilon 
   \end{split}
\end{equation*}
 for all $\varepsilon\in(0,\bar \varepsilon),$ which holds for any  ball $B_r(x_0)\subset\Omega'$ with radius $r<\bar r$.   Then,  by using Lemma \ref{lem6.5}, we conclude that  
 $\{|{\bf u}|>0\}\cap B_{r/2}(x_0)$ has finite perimeter. Hence,  $\{|{\bf u}|>0\}\cap \Omega'$ has finite perimeter.
  
From Proposition \ref{prop2.19}(1), we have $\mathscr H^{N-1}(\partial^*\{|{\bf u}|>0\}\cap  \Omega')<+\infty$. The density estimates in Lemma \ref{lem6.3} and Lemma \ref{lem6.4} imply
\begin{equation} \label{equ6.41}
\partial^*\{|{\bf u}|>0\}\cap \Omega=\partial\{|{\bf u}|>0\}\cap \Omega.
\end{equation}
Now  the assertion (1) follows. 

For the assertion (2), by noticing that Proposition \ref{prop2.19}(2) and that ${\bf \Delta}u_i$ is a  Radon measure supported in $\partial\{|{\bf u}|>0\}\cap \Omega$  (see Lemma \ref{lem4.1} and Lemma \ref{lem4.4}), we need only to show that  ${\bf\Delta} u_i$ is absolutely continuous with respect to $\mathscr H^{N-1}$, for each $i=1,2,\cdots,m$.

Let  $B_r(x)\Subset \Omega$ with $x\in \partial\{|{\bf u}|>0\}\cap \Omega $. Taking a cut-off function $\phi:\Omega\to[0,1]$   with $\phi=1$ on $B_{r/2}(x)$, ${\rm supp}(\phi)\subset B_{r}(x)$ and $|\nabla\phi|\ls C_N/r$, we have
$${\bf \Delta}u_i\big(B_{r/2}(x)\big)\ls {\bf \Delta}u_i(\phi)=-\int_{B_r(x)}\ip{\nabla u_i}{\nabla \phi}\du\ls \frac{C_N\cdot L}{r}\cdot\mu(B_{r}(x)),$$
where we have used ${\bf \Delta}u_i\gs0$, $|\nabla u_i|\ls L$. Thus, by (\ref{equ2.12}), we get
${\bf \Delta}u_i\big(B_{r/2}(x)\big)\ls C_{N,K,L}\cdot r^{N-1}$. This shows, that the Radon measure ${\bf \Delta} u_i$ is absolutely continuous with respect to $\mathscr H^{N-1}$, and then shows the assertion (2).   
\end{proof}

\section{Compactness and the Euler-Lagrange equation\label{sec:compactness}}

In this section, we consider the compactness of local minimizers of $J_Q$ living in a sequence of $pmGH$-converging $ncRCD$-spaces, under some uniformity assumptions.
Let $K\ls 0, N\in(1,+\infty)$ and let $(X_j,d_j,\mu_j)$ be a sequence of $ncRCD(K,N)$ metric measure spaces. Fix $p_j\in X_j$ for each $j\in\mathbb N$. Suppose that 
$$  (X_j,d_j,\mu_j,p_j)\overset{pmGH}{\longrightarrow}(X_\infty,d_\infty,\mu_\infty,p_\infty).$$
According to \cite{DPG16}, the limit $(X_\infty, d_\infty,\mu_\infty)$ is still an $ncRCD(K,N)$ metric measure space.

Fix $R>0$. For each $j\in\mathbb N$, let  $Q_j\in C(\overline{B_R(p_j)})$  and let
$${\bf u}_j:=(u_{j,1},u_{j,2},\cdots, u_{j,m})\in W^{1,2}(B_R(p_j),\mathbb R^m)$$
be a local minimizer of $J_{Q_j}$ on $B_R(p_j)$ with size $\varepsilon_{{\bf u}_j}>0$. That is, for each $j\in\mathbb N$, there exists a data ${\bf g}_j\in W^{1,2}(B_R(p_j),[0,\infty)^m)$ such that 
$J_{Q_j}({\bf u}_j) \ls J_{Q_{j}}({\bf v}_j)$ for all ${\bf v}_j\in \mathscr A_{{\bf g}_j}$   with $\dist({\bf u}_j,{\bf v}_j)<\varepsilon_{{\bf u}_j}$, where the $\mathscr A_{{\bf g}_j}$ and $\dist({\bf u}_j,{\bf v}_j)$ are given in (\ref{equ1.4}) and (\ref{equ1.5}), respectively.

\begin{theorem}[Compactness]\label{thm7.1}
  Let $R, Q_j,{\bf u}_j$  be  as above. Let $Q_\infty\in C(\overline{B_R(p_\infty)})$ such that  
  \begin{equation*} 
\lim_{j\to\infty}Q_j(x_j)= Q_\infty(x_\infty) \quad {\rm whenever}\ \ x_j\overset{GH}{\longrightarrow}x_\infty\in B_R(p_\infty).
\end{equation*}
{\rm(}Recall that $x_j\overset{GH}{\longrightarrow}x_\infty$ means $\Phi_{j}(x_{j}) \to \Phi_\infty(x_\infty)$ in $Z$, where $\Phi_{j},\Phi_\infty$  and $Z$ are given in the Definition \ref{pmGH}(3).{\rm)}  Assume that $\{{\bf u}_j\}$ are uniformly bounded on $\overline{B_R(p_j)}$, 
 \begin{equation}\label{equ7.1}
\lim_{j\to+\infty}\varepsilon_{{\bf u}_j}=+\infty,
\end{equation}
and  there exist  positive constants $Q_{\min},Q_{\max}$ and $L$ such that
 \begin{equation}\label{equ7.2}
  0<Q_{\min}\ls    Q_j\ls  Q_{\max}<\infty \ \ {\rm on}\ B_{R}(p_j),\ \ \forall j\in \mathbb N,
  \end{equation}
 and  
   \begin{equation}\label{equ7.3}
  |\nabla {\bf u}_j|\ls L\quad {\rm on}\ B_{R}(p_j),\ \ \forall j\in\mathbb N.\qquad  
    \end{equation}
Then   there exist a subsequence, denoted by  $ \{{\bf u}_j\}_j$ again,  and a map ${\bf u}_\infty\in Lip(\overline{B_{R}(p_\infty)})$  such that  ${\bf u}_{j} \to  {\bf u}_\infty$ uniformly over $B_{R}(p_{j})$ as $j\to\infty$ and that for any $R'\ls R$,  the limit map ${\bf u}_\infty$ is a  minimizer of $J_{Q_\infty}$ on $B_{R'}(p_\infty)$. Moreover, for any $R'\ls R,$ the followings hold:
 \begin{align}
 \label{equ7.4} \lim_{j\to\infty}\int_{B_{R'}(p_{j}) }\left|\nabla{\bf u}_{j}\right|^{2}{\rm d}\mu_j&= \int_{B_{R'}(p_{\infty}) }\left|\nabla{\bf u}_{\infty}\right|^{2}{\rm d}\mu_\infty,\\
 \label{equ7.5}   \lim_{j\to\infty}\int_{B_{R'}(p_{j}) }Q_{j}\chi_{\left\{\left|{\bf u}_{j}\right|>0\right\}}{\rm d}\mu_j&= \int_{B_{R'}(p_{\infty}) }Q_\infty\chi_{\left\{\left|{\bf u}_{\infty}\right|>0\right\}}{\rm d}\mu_\infty,\\
  \label{equ7.6}    \partial \{|{\bf u}_{j}|>0\}\cap B_{R'}(p_{j})&\overset{GH}{\longrightarrow} \partial\{|{\bf u}_\infty|>0\}\cap B_{R'}(p_\infty), \\
 \label{equ7.7}  \mu_{j}\big(B_{R'}(p_{j})\cap \{|{\bf u}_{j}|>0\}\big)&\longrightarrow\mu_\infty\big(B_{R'}(p_\infty)\cap \{|{\bf u}_\infty|>0\}) .
 \end{align}
 \end{theorem}

\begin{proof}  
From \eqref{equ7.3} and the  Arzela-Ascoli   theorem,    there exist  a subsequence of $\{{\bf u}_{j}\} _{j}$ converging  uniformly to some
  ${\bf u}_{\infty}\in\Lip(\overline{B_{R}(p_\infty)})$ with the same Lipschitz constant  $L$.  
  
(i) Fix any $R'\ls R$. We first show the minimality of ${\bf u}_\infty$ on $B_{R'}(p_\infty)$.  The lower semicontinuity of the Cheeger energy (see, for example, \cite[Lemma 2.12]{ZZ19}) gives
  \begin{equation}\label{equ7.8}
    \liminf_{j\to\infty}\int_{B_{R'}(p_{j}) }\left|\nabla{\bf u}_{j}\right|^{2}{\rm d}\mu_j\gs \int_{B_{R'}(p_{\infty}) }\left|\nabla{\bf u}_{\infty}\right|^{2}{\rm d}\mu_\infty\text{.}
  \end{equation}
Since $\left\{\left|{\bf u}_{\infty}\right|>0\right\}$ is an open set, it is easy to check    that 
  $$x_{j}\overset{GH}{\longrightarrow}  x_\infty\quad   \Longrightarrow\quad  \liminf_{j\to\infty}Q_{j}\chi_{\{{\bf u}_{j} >0\} }(x_{j})\ \gs  Q_\infty\chi_{\{{\bf u}_\infty >0\} }(x_\infty).$$
  By Fatou's lemma (see  \cite[Lemma 2.5]{DPG18} or \cite[Appendix A]{ZZ19}, for the Fatou's lemma for functions defined on varying spaces), we have  
  \begin{equation}\label{equ7.9}
    \liminf_{j\to\infty}\int_{B_{R'}(p_{j}) }Q_{j}\chi_{\left\{\left|{\bf u}_{j}\right|>0\right\}}{\rm d}\mu_j\gs \int_{B_{R'}(p_{\infty}) }Q_\infty\chi_{\left\{\left|{\bf u}_{\infty}\right|>0\right\}}{\rm d}\mu_\infty\text{.}
  \end{equation}
 Let  ${\bf v}_{\infty}:B_{R'}(x_{\infty}) \to\mathbb{R}^{m}$ be a minimizer of $J_{Q_\infty}$ with   ${\bf v}_\infty-{\bf u}_\infty\in W^{1,2}_0(B_{R'}(p_\infty),\mathbb R^m)$. 
We first claim that
\begin{equation}\label{equ7.10}
J_{Q_\infty}({\bf v}_\infty)\gs \limsup_{j\to\infty} J_{Q_{j}}({\bf u}_{j}).
\end{equation}
If this claim holds, by combining (\ref{equ7.8}), (\ref{equ7.9}),  the minimality of ${\bf v}_\infty$ and  (\ref{equ7.10}),   then we have
    \begin{equation*} 
    \liminf_{j \to\infty}J_{Q_{j}}({\bf u}_{j}) \gs  J_{Q_\infty}({\bf u}_{\infty})\gs J_{Q_\infty}({\bf v}_\infty) \gs \limsup_{j\to\infty} J_{Q_{j}}({\bf u}_{j})\text{.}
  \end{equation*} 
  This yields
      $\lim_{j \to\infty}J_{Q_{j}}({\bf u}_{j}) = J_{Q_\infty}({\bf u}_{\infty})= J_{Q_\infty}({\bf v}_\infty).$
      Thus, we conclude that ${\bf u}_\infty$ is also a minimizer of $J_{Q_\infty}$ on $B_{R'}(p_\infty)$, and that the inequalities in   (\ref{equ7.8}) and (\ref{equ7.9}) must be equalities, i.e.,  the both assertions of (\ref{equ7.4}) and (\ref{equ7.5}) hold.

Now let us prove the claim \eqref{equ7.10} by a contradiction argument. 
  Suppose not, there exists some  $\delta_0>0$ and a subsequence of  $\{{\bf u}_{j}\}_j$, denoted by $\{{\bf u}_{j}\}_j$ again, such that 
\begin{equation}\label{equ7.11}
J_{Q_\infty}({\bf v}_\infty)\ls \lim_{j\to\infty} J_{Q_{j}}({\bf u}_{j})-\delta_0.
\end{equation}
Let  $\delta\in(0, \delta_{0})$ be a constant so small  that
\begin{equation}\label{equ7.12}
4m\sqrt{\Lambda_\infty}\cdot  \mu_{\infty}^{1/2}\big(A_{R'-\delta,R'}(p_\infty)\big)    +\big(4m+Q_{\max}\big)\cdot    \mu_{\infty}\big(A_{R'-\delta,R'}(p_\infty)\big)\ls \frac{\delta_0}{2},
\end{equation}
where $  \Lambda_\infty:= \int_{B_{R'}(p_{\infty})}|\nabla {\bf v}_{\infty}|^2\du_\infty,$ and $A_{r_1,r_2}(p):=B_{r_2}(p)\backslash B_{r_1}(p)$.
 
   From Theorem \ref{thm1.4}, we have ${\bf v}_\infty\in Lip_{\rm loc}(B_{R'}(p_\infty))$.    By combining this and the facts that ${\bf u}_{j}$ converges uniformly to ${\bf u}_\infty$ over $B_{R'}(p_{j})$, $|\nabla {\bf u}_{j}|\ls L$ for all $j\in\mathbb N$,  we conclude, by Lemma \ref{lem2.11}, that there exist a sequence of maps $
{\bf v}_{j}:B_{R'}(p_{j})\to\mathbb R^m$ such that ${\bf v}_{j}-{\bf u}_{j}\in W^{1,2}_0\big(B_{R'}(p_{j}),\mathbb R^m\big)$, ${\bf v}_{j}\to {\bf v}_\infty$ uniformly over $B_{R'-\delta}(p_{j}) $ as $j\to\infty$,  and 
\begin{equation}\label{equ7.13}
\lim_{j\to\infty} \int_{B_{R'}(p_{j})}|\nabla {\bf v}_{j}|^2\du_j= \int_{B_{R'}(p_{\infty})}|\nabla {\bf v}_{\infty}|^2\du_\infty\quad  (=\Lambda_\infty).
\end{equation}

For each $j\in \mathbb N$,  we put $\mathbf{w}_{j}=(w_{j,1},\ldots,w_{j,m}) $, where
  \begin{equation}\label{equ7.14}
    w_{j,\alpha}:=\max\left\{0,v_{j,\alpha}-\delta\phi_{j}\right\}\quad \forall\ \alpha=1,2,\cdots, m,
  \end{equation}
and $\phi_{j}$ is a cut-off function on $B_{R'}(p_{j})$ such that ${\rm supp}(\phi_{j})\subset B_{R'}(p_{j})$, $\phi_{j}(x)=1$ if $x\in B_{R'-\delta}(p_{j})$, and $|\nabla \phi_{j}|\ls 2/\delta$.
It is clear that  $\mathbf{w}_{j}$ is an admissible map  for ${\bf u}_j$ in \eqref{equ1.4}. 
 From the definition of ${\rm d}({\bf u}_j,{\bf v}_j)$ in \eqref{equ1.5}, by using the Poincar\'e inequality to ${\bf v}_{j}-{\bf u}_{j}\in W^{1,2}_0\big(B_{R'}(p_{j}),\mathbb R^m\big)$,  and then $|\nabla {\bf u}_{j}|\ls L$, \eqref{equ7.13} and (\ref{equ2.12}), we get  
$${\rm d}({\bf v}_j,{\bf u}_j)\ls C'$$
 for some constant $C'>0$ independent of $j$ (may depend on $N,K,R',L,\Lambda_\infty$ and the Poincar\'e constant $C_P$ in Proposition \ref{prop2.1}(5).) 
Thus, by combining this and \eqref{equ7.14}, we obtain
\begin{equation*}
\begin{split}
{\rm d}({\bf w}_j,{\bf u}_j)&\ls {\rm d}({\bf w}_j,{\bf v}_j)+{\rm d}({\bf v}_j,{\bf u}_j)\\
&\ls \delta\cdot \|\phi_j\|_{W^{1,2}(B_{R'}(p_j))}+\mu(B_{R'}(p_j))+C'\ls C'' 
\end{split}
\end{equation*}
for some constant $C''>0$ independent of $j$. 
Thus, for any $j$ sufficiently large (such that $\varepsilon_{{\bf u}_j}>C''$ since the assumption \eqref{equ7.1}),    the local minimality of ${\bf u}_{j}$ implies
  \begin{equation}\label{equ7.15}
    \begin{aligned}
    J_{Q_{j}}({\bf u}_{j}) \ls &  J_{Q_{j} }(\mathbf{w}_{j})                                                                                                                                                                                                                          = \int_{B_{R'}(p_{j}) }\big(\left|\nabla\mathbf{w}_{j}\right|^{2}+Q_{j}\chi_{\left\{\left|\mathbf{w}_{j}\right|>0\right\}}\big)\du_{j}                                                                                                                       \\
      \ls                                                        & \int_{B_{R'} (p_{j}) }\left|\nabla\big({\bf v}_{j}-\delta\phi_{j}\big)\right|^{2}                                                                                                +\int_{B_{R'}(p_{j}) }Q_{j}\chi_{\left\{\left|\mathbf{w}_{j}\right|>0\right\}},      
     \end{aligned}
  \end{equation}
  where we have used $|\nabla {\bf w}_{j}|\ls |\nabla (v_{j}-\delta\phi_{j})|$ $\mu_{j}$-a.e. in $B_{R'}(p_{j})$ (see, for instance,  Proposition \ref{prop2.1}(3)). Noticing that $\phi_{j}=1$ on $B_{R'-\delta}(p_{j})$ and $|\nabla \phi_{j}|\ls 2/\delta$ on $A_{R'-\delta,R'}(p_{j})$, we have
         \begin{equation}\label{equ7.16}
    \begin{aligned}
  &\int_{B_{R'} (p_{j}) } \left|\nabla\big({\bf v}_{j}-\delta\phi_{j}\big)\right|^{2}                     \\
                    \ls& \int_{B_{R'-\delta}(p_{j})}\left|\nabla{\bf v}_{j}\right|^{2} \\
                   &  \  + \int_{A_{R'-\delta, R'} (p_{j}) }\Big(\left|\nabla {\bf v}_{j} \right|^{2}+2 m\cdot\delta \left|\nabla {\bf v}_{j} \right|\cdot|\nabla \phi_{j}|+m\cdot\delta^2|\nabla \phi_{j}|^2 \Big) \\      
               \ls    & \int_{B_{R'} (p_{j}) } \left|\nabla {\bf v}_{j} \right|^{2}+4 m\int_{A_{R'-\delta, R'} (p_{j}) }\left|\nabla {\bf v}_{j} \right|  + 4m\cdot\mu_{j}\big(A_{R'-\delta,R'}(p_{j})\big) \\
   \ls&  \int_{B_{R'} (p_{j}) } \left|\nabla {\bf v}_{j} \right|^{2}+4 m    \Big(\int_{B_{R'}(p_{j})}|\nabla {\bf v}_{j}|^2\Big)^{1/2}\cdot \mu_{j}\big(A_{R'-\delta,R'}(p_{j})\big)^{1/2}\\
   &+ 4m\cdot\mu_{j}\big(A_{R'-\delta,R'}(p_{j})\big).
         \end{aligned}
  \end{equation}
  By the  definition of ${\bf w}_{j},$ (\ref{equ7.14}), we have
  $$\{|{\bf w}_{j}|>0\}\cap B_{R'-\delta}(p_{j})\subset \{|{\bf v}_{j}|\gs \delta \} \cap B_{R'-\delta}(p_{j}),$$ 
and then 
\begin{equation*}
  \int_{B_{R'}(p_{j}) }Q_{j}\chi_{\left\{\left|\mathbf{w}_{j}\right|>0\right\}}\ls \int_{B_{R'-\delta}(p_{j}) }Q_{j}\chi_{\left\{\left|\mathbf{v}_{j}\right|\gs \delta\right\}} +  Q_{\max}\cdot \mu_{j}\big(A_{R'-\delta,R'}(p_{j})\big).
\end{equation*} 
Substituting this and (\ref{equ7.16}) into (\ref{equ7.15}), we obtain    
   
   \begin{equation}\label{equ7.17}
    \begin{aligned}
    J_{Q_{j}}({\bf u}_{j}) \ls & \int_{B_{R'}(p_{j})}\left|\nabla {\bf v}_{j} \right|^{2}  +\int_{B_{R'-\delta}(p_{j}) }Q_{j}\chi_{\left\{\left|\mathbf{v}_{j}\right|\gs \delta\right\}}  \\
    &+ 4 m    \Big(\int_{B_{R'}(p_{j})}|\nabla {\bf v}_{j}|^2\Big)^{1/2}\cdot \mu_{j}\big(A_{R'-\delta,R'}(p_{j})\big)^{1/2}\\
   &    +(4m+  Q_{\max})\cdot \mu_{j}\big(A_{R'-\delta,R'}(p_{j})\big).  
         \end{aligned}
  \end{equation} 
Recall that ${\bf v}_{j}\to {\bf v}_\infty$ uniformly over $B_{R'-\delta}(p_{j}) $ as $j\to\infty$. Hence, $$\limsup_{j\to\infty}\chi_{\{|{\bf v}_{j}|\gs \delta\}}(x_{j})\ls \chi_{\{|{\bf v}_\infty|>0\}}(x_\infty),\quad  \forall \ x_{j}\overset{GH}{\longrightarrow}x_\infty.$$ By using this and  letting $j\to\infty$ in (\ref{equ7.17}), we obtain 
    \begin{equation*} 
        \begin{aligned}
  \limsup_{j\to\infty}  J_{Q_{j}}({\bf u}_{j}) \ls& \Lambda_\infty+\int_{B_{R'}(p_\infty) }Q_{\infty}\chi_{\left\{\left|\mathbf{v}_{\infty}\right|>0\right\}} + 4m\sqrt{\Lambda_\infty}\cdot  \mu_{\infty}^{1/2}\big(A_{R'-\delta,R'}(p_\infty)\big)\\&\ \ +\big(4m+Q_{\max}\big)\cdot    \mu_{\infty}\big(A_{R'-\delta,R'}(p_\infty)\big)   \\
\ls &  J_{Q_\infty}({\bf v}_\infty)+\frac{\delta_0}{2},
          \end{aligned}
  \end{equation*}   
where we have used (\ref{equ7.13}), (\ref{equ7.12}) and $\mu_{j}\big(A_{R'-\delta,R'}(p_{j})\big)\to \mu_{\infty}\big(A_{R'-\delta,R'}(p_{\infty})\big)$. This contradicts with (\ref{equ7.11}), and then proves the claim   (\ref{equ7.10}). Therefore, we have proved the minimality of ${\bf u}_\infty$ and the equalities (\ref{equ7.4}), (\ref{equ7.5}).

 (ii) Next we will prove (\ref{equ7.6}). On the one hand, let $x_{j}\in \partial \{|{\bf u}_{j}|>0\}\cap B_{R'}(p_{j})$ such that $x_{j}\overset{GH}{\longrightarrow}x_\infty\in B_{R'}(p_\infty)$.  Fixing any $r>0$, we have $B_r(x_{j})\overset{GH}{\longrightarrow}B_r(x_\infty)$. By Theorem \ref{thm6.1}, we get $\sup_{B_r(x_{j})}|{\bf u}_{j}|\gs C r$ for a constant $C>0$ independent of $j$.  Since ${\bf u}_{j}\to {\bf u}_\infty$ over $B_{R'}(p_{j})$ uniformly, we have $|{\bf u}_\infty(x_\infty)|=0$ and  $\sup_{B_r(x_{\infty})}|{\bf u}_{\infty}|\gs C r$. By the arbitrariness of $r>0$, we conclude $x_\infty\in \partial\{|{\bf u}_\infty|>0\}.$    On the other hand, for each $y_\infty\in \partial\{|{\bf u}_\infty|>0\}\cap B_{R'}(p_\infty)$. We can find a sequence $y_{j}\in B_{R'}(p_{j})$ such that $y_{j}\overset{GH}{\longrightarrow}y_\infty$ and $|{\bf u}_{j}(y_{j})|\to 0$. By the nondegeneracy, we get 
 $$d\big(y_{j},\partial \{|{\bf u}_{j}|>0\}\big)\to 0\quad {\rm as}\quad j\to\infty.$$
 Therefore, there exists a sequence $\{z_{j}\}$ such that $z_{j}\in     \partial \{|{\bf u}_{j}|>0\}\cap B_{R'}(p_{j})$ and $d(y_j,z_j)\to 0$ as $j\to \infty$. So we have $z_{j}\overset{GH}{\longrightarrow}  y_\infty$. This proves the assertion (\ref{equ7.6}).
 
 (iii) The assertion (\ref{equ7.7}) follows from Lemma \ref{lem6.6}. Since all sets $B_{R'}(p_{j})\cap \{|{\bf u}_{j}|>0\}$ are open, and $\mu_{j}\rightharpoonup\mu_\infty$, we have
 $$\liminf_{k\to\infty} \mu_{j}\Big(B_{R'}(p_{j})\cap \{|{\bf u}_{j}|>0\} \Big)\gs \mu_\infty\Big(B_{R'}(p_{\infty})\cap \{|{\bf u}_{\infty}|>0\} \Big).$$
 Similarly,  given any $\varepsilon>0$, the fact that all sets $B_{R'}(p_{j})\cap \{|{\bf u}_{j}|\gs \varepsilon\}$ are closed implies
\begin{equation*}
\begin{split}
\limsup_{k\to\infty} \mu_{j}\Big(B_{R'}(p_{j})\cap \{|{\bf u}_{j}|\gs\varepsilon\} \Big)&\ls \mu_\infty\Big(B_{R'}(p_{\infty})\cap \{|{\bf u}_{\infty}|\gs 
 \varepsilon/2\} \Big)\\
 &\ls \mu_\infty\Big(B_{R'}(p_{\infty})\cap \{|{\bf u}_{\infty}|>0\} \Big).
 \end{split}
 \end{equation*}
From Lemma \ref{lem6.6}, there exists a constant $C>0$ (independent of $j$) such that for all $\varepsilon\in(0,\bar\varepsilon)$, we have
 \begin{equation*}
\begin{split}
&\mu_{j}\Big(B_{R'}(p_{j})\cap \{|{\bf u}_{j}|>0\} \Big)-\mu_{j}\Big(B_{R'}(p_{j})\cap \{|{\bf u}_{j}|\gs\varepsilon\} \Big)\\
&\ls \mu_{j}\Big(B_{R'}(p_{j})\cap \{0<|{\bf u}_{j}|\ls \varepsilon\} \Big)\ls C\varepsilon,\qquad \forall\ j\in \mathbb N.
\end{split}
\end{equation*}
 By  combining  these three inequality and the arbitrariness of $\varepsilon\in(0,\bar \varepsilon)$, the assertion (\ref{equ7.7}) follows. Now the proof is finished.
\end{proof}

We now apply Theorem \ref{thm7.1} to the special case of blow-up limits, to get the Euler-Lagrange equation of local minimizers.
\begin{corollary}\label{cor7.2}
Let ${\bf u}=(u_1,u_2,\cdots,u_m)$ be a local minimizer of $J_Q$ on $\Omega$ with $Q\in C(\Omega).$ Recall that ${\bf \Delta}u_i=q_i\cdot\mathscr H^{N-1}\llcorner  (\partial\{|{\bf u}|>0\}\cap\Omega)$  for some nonnegative Borel functions $q_i$, $i=1,2,\cdots,m$ (see Proposition \ref{prop6.7}). Then
\begin{equation} \label{ equ7.18}
\sum_{i=1}^m q^2_i(x_0)=Q(x_0), \quad\mathscr H^{N-1}{\rm-a.e.}\ x_0\in \partial\{|{\bf u}|>0\}\cap \Omega.
\end{equation}
\end{corollary}

\begin{proof}Since $\Omega_{\bf u}=\Omega\cap \{|{\bf u}|>0\}$ is of locally finite perimeter (by Proposition \ref{prop6.7}), we know from Proposition \ref{prop2.19}(2) and \eqref{equ6.41} that the reduced boundary  $\mathcal F\{|{\bf u}|>0\}\cap \Omega$ has full $\mathscr H^{N-1}$-measure in $\partial\{|{\bf u}|>0\}\cap\Omega$. Suppose that $x_0\in\mathcal F\{|{\bf u}|>0\}\cap \Omega$ and that it is a Lebesgue's point of $q_i$, for all $i=1,2, \cdots, m$, with respect to $\mathscr H^{N-1}$. It suffices to show (\ref{ equ7.18}) at such $x_0$. 

 Let $\{r_j\}_{j=1,2,\cdots}$ be a sequence of real numbers  such that $r_j\to0^+$ as $j\to\infty$, and consider the blow-up sequence of spaces 
$$X_j:=\big(X,d_j:=r^{-1}_jd,\mu_j, x_0\big),\qquad  \mu_j:=\mu^{x_0}_{r_j}=c_j\cdot\mu ,$$ where $c_j^{-1}=\int_{B_{r_j}(x_0)}\big(1-r^{-1}_jd(x,x_0)\big)\du(x)$ (is given in \eqref{equ2.11}).  We   denote by   $B^{(j)}_R(x_0)$  the ball in $X_j$   with radius $R$ (with respect to the metric $d_j$). 

 Given any ${\bf v}\in W^{1,2}(B_R(x_0),\mathbb R^m)$, it is clear that the blow-up sequence of maps 
  ${\bf v}_j:=r_j^{-1}{\bf v}\in W^{1,2}(B^{(j)}_R(x_0),\mathbb R^m)$. Moreover, if ${\bf v}\in C(\overline{B_R(x_0)})$, then  for each $j\in \mathbb N$, we have
\begin{align}
\label{ equ7.19}
\int_{B^{(j)}_{R}(x_0) }| {\bf v}_j|^2\du_{j}&=r^{-2}_jc_j\int_{B_{Rr_j}(x_0)}|  {\bf v}|^2\du,\\
\label{  equ7.20}
\int_{B^{(j)}_{R}(x_0) }|\nabla^{(j)} {\bf v}_j|^2\du_{j}&=c_j\int_{B_{Rr_j}(x_0)}|\nabla {\bf v}|^2\du,
\end{align}
  where $|\nabla^{(j)}v|$ is the  minimal weak upper gradient for $v$ with respected to $d_j$,  and 
    \begin{equation}\label{equ7.21}
\int_{B^{(j)}_{R}(x_0)}Q(x)\chi_{\{|{\bf v}_j|>0\}}\du_{j}=c_j \int_{B_{Rr_j}(x_0)}Q(x)\chi_{\{|{\bf v}|>0\}} \du,
\end{equation}
 since $\{|{\bf v}_j|>0\}\cap  B^{(j)}_R(x_0)=\{|{\bf v}|>0\}\cap  B_{Rr_j}(x_0)$. Denoting by $B_R^{(j)}:=B^{(j)}_R(x_0)$ and $B_{Rr_j}:=B_{Rr_j}(x_0)$, the combination of \eqref{ equ7.20}-\eqref{equ7.21} gives, for each $j\in\mathbb N$, that
\begin{equation}\label{equ7.22}
\begin{split}
J_Q({\bf v}_j,B^{(j)}_R):& =\int_{B^{(j)}_{R} }|\nabla^{(j)} {\bf v}_j|^2\du_{j}+\int_{B^{(j)}_{R}}Q(x)\chi_{\{|{\bf v}_j|>0\}}\du_{j}\\
&=c_j\cdot J_Q({\bf v},B_{Rr_j})
\end{split}
 \end{equation}
and that,  by the definition \eqref{equ1.5},   
\begin{equation}\label{equ7.23}
\begin{split}
{\rm d}_{B_R^{(j)}}({\bf v}_j,{\bf w}_j):& =\left\|{\bf v}_j-{\bf w}_j\right\|_{W^{1,2}(B_R^{(j)},\mathbb{R}^{m}) }+\left\|\chi_{\left\{ |{\bf v}_j|>0\right\}}-\chi_{\left\{ |{\bf w}_j |>0\right\}}\right\|_{L^{1}(B_R^{(j)})}\\
&=c_j \Big({\rm d}_{B_{Rr_j}}({\bf v},{\bf w}) - \left\|{\bf v}-{\bf w} \right\|_{L^{2}(B_{Rr_j},\mathbb{R}^{m}) }\Big)\\
& \ \quad + c_j r^{-1}_j \left\|{\bf v}-{\bf w} \right\|_{L^{2}(B_{Rr_j},\mathbb{R}^{m}) }\\
&\gs c_j {\rm d}_{B_{Rr_j}}({\bf v},{\bf w})  \qquad ({\rm by}\ \ r_j\ls 1).
\end{split}
\end{equation}

Noticing that $x_0$ is a regular point, we have $X_j\overset{pmGH}{\longrightarrow}(\mathbb R^N,d_E,c_N\mathscr H^N,0)$ as $j\to+\infty$.
Since ${\bf u}$ is a local minimizer of $J_Q$ on $B_R(x_0)\subset \Omega$ with size $\varepsilon_{{\bf u}}>0$, we conclude, for each $j\in \mathbb N$, that the blow-up map ${\bf u}_j:=r_j^{-1}{\bf u}$ is a local minimizer of $J_Q$ on $B_R^{(j)}(x_0)$ with size $c_j\cdot \varepsilon_{{\bf u}}$ (from \eqref{equ7.23}).  
Since $c_j\to+\infty$ as $j\to+\infty$,  by using Theorem 7.1 and a diagonal argument, there exist a subsequence of $r_j$ such that  ${\bf u}_j$ converges to a limit map 
$${\bf u}_0=(u_{0,1},u_{0,2},\cdots, u_{0,m})$$
on the tangent cone $(\mathbb R^N,d_E,c_N\mathscr H^N,0)$, and that for each $R>0$, ${\bf u}_0$ is  a minimizer of $J_{Q_0}$ on each Euclidean ball $B^e_R(0)\subset \mathbb R^N$, where $Q_0=Q(x_0)$. Moreover, by applying (\ref{equ7.4}), (\ref{equ7.5}), (\ref{  equ7.20}) and (\ref{equ7.21}), we obtain
  \begin{equation}\label{equ7.24}
\int_{B^e_{R}(0)}|\nabla {\bf u}_0|^2{\rm d}x=R^N\cdot \omega_N\cdot \lim_{j\to \infty}\fint_{B_{Rr_j}(x_0)}|\nabla {\bf u}|^2\du,
\end{equation}
 and 
\begin{equation}\label{equ7.25}
\int_{B^e_{R}(0)}Q_0\chi_{\{|{\bf u}_0|>0\}} {\rm d}x= R^N\cdot \omega_N \cdot \lim_{j\to \infty}\fint_{B_{Rr_j}(x_0)}Q(x)\chi_{\{|{\bf u}|>0\}} \du,
\end{equation}
where we have used   $ c_j\mu(B_{r_j}(x_0))=\mu_j(B^{(j)}_1(x_0))$, $\lim_{j\to\infty}\mu_j(B^{(j)}_1(x_0))=c_N\omega_N$ and $\lim_{j\to \infty}\frac{ \mu\big(B_{Rr_j}(x_0)\big)}{\mu\big(B_{r_j}(x_0)\big)}=R^N$ (see \cite[Corollary 1.7]{DPG18}).

Remark that ${\bf u}_0$ is also a blow-up limit of itself on $\mathbb R^N$ (Indeed, by taking a subsequence of $\{r_j\}$, says $\{r'_j\}$, such that $r'_j/r_j:=\epsilon_j\to 0$ as $j\to+\infty$, we get that  ${\bf u}_0(\cdot)$ is the blow-up limit of $\epsilon_j^{-1}{\bf u}_0(\epsilon_j\cdot)$ on $(\mathbb R^N, d_E,c_N\mathscr H^N,0)$.) According to the classification of blow-up limits   on Euclidean space $\mathbb R^N$ (see \cite[Proposition 4.2]{MTV17} or \cite[Lemma 23]{CSY18}), we conclude that   there is a 1-homogeneous nonnegative global minimizer $u: \mathbb R^N\to [0,\infty)$  of the one-phase Alt–Caffarelli functional
$$J(u): = \int \big(|\nabla u|^2+Q_0\cdot \chi_{\{u>0\}}\big){\rm d}x$$
such that ${\bf u}_0(x) =\xi\cdot u(x)$, where $\xi=(\xi_1,\cdots,\xi_m) \in \mathbb R^N$ with $|\xi|  = 1$. On the other hand, since $x_0\in \mathcal F\{|{\bf u}|>0\}$, we have $\lim_{j\to \infty}\fint_{B_{Rr_j}(x_0)\subset X}\chi_{\{|{\bf u}|>0\}} \du=1/2$, and hence, by (\ref{equ7.25}) and  that $Q$ is continuous at $x_0$, we get
  $$\int_{B^e_R(0)}\chi_{\{u>0\}}{\rm d}x=\int_{B^e_{R}(0)}\chi_{\{|{\bf u}_0|>0\}} {\rm d}x=1/2.$$
This yields, by Theorem 5.5 in \cite{AC81},  that $\partial\{u>0\}$ is a $(N-1)$-dimensional  hyperplane in $\mathbb R^N$ and   ${\Delta} u=\sqrt{Q_0}\cdot\mathscr H^{N-1}\llcorner \partial\{u>0\}$ in the sense of measures. Thus, we obtain
\begin{equation}\label{equ7.26}
{ \Delta} u(B^e_R(0))=\sqrt{Q_0}\cdot\mathscr H^{N-1}\big(B^e_R(0)\cap \partial\{u>0\}\big)=\sqrt{Q_0}\cdot\omega_{N-1}R^{N-1}.
\end{equation}

For each $i_0\in \{1,2,\cdots, m\}$, recalling that ${\bf \Delta}u_{i_0}$ is a Radon measure supported on $\partial\{|{\bf u}|>0\}\cap \Omega $, we will calculate  the density of ${\bf \Delta}u_{i_0}$ at $x_0$. Fix any $\delta\in(0,1/8)$.   For each $j\in\mathbb N\cup\{+\infty\}$, we take  the Lipschitz cut-off $\phi_j:X_j\to [0,1]$ as 
$$\phi_j(x):=
\begin{cases}\min\big\{1,\frac{1+\delta}{\delta}-\frac{d_j(x,x_0)}{\delta }\big\},&\quad x\in B_{1+\delta}^{(j)}(x_0)\subset X_j\\
0&\quad x\not\in B_{1+\delta}^{(j)}(x_0)\subset X_j,\end{cases}
$$
where $X_\infty=(\mathbb R^N,d_E,\mu_\infty=c_N\mathscr H^N,0)$, $x_\infty=0$ and $B^{(\infty)}_R(0)=B^e_R(0).$ 
Then
\begin{equation}\label{equ7.27}
\begin{split}
{\bf \Delta}u_{i_0}(B_{r_j}(x_0))&\ls {\bf \Delta}u_{i_0}(\phi_j) =-\int_{B_{(1+\delta)r_j}(x_0)}\ip{\nabla u_{i_0}}{\nabla \phi_j}\du\\
&=-r_j^{-1}\int_{B^{(j)}_{1+\delta}(x_0)}\ip{\nabla^{(j)}(r_j^{-1}u_{i_0})}{\nabla^{(j)}\phi_j}c^{-1}_j\du_j,
\end{split}
\end{equation}
where we have used $|\nabla\phi_j|= |\nabla^{(j)}(r_j^{-1} \phi_j)|$ and ${\bf \Delta }u_{i_0}\gs0$. 
  Letting $j\to\infty$, by using (\ref{equ7.24}) and   
  $$\lim_{j\to+\infty}\int_{B^{(j)}_{1+\delta}(x_0) }|\nabla\phi_j|^2{\rm d}\mu_j =\int_{B^{e}_{1+\delta}(0) }|\nabla\phi_\infty|^2{\rm d}\mu_\infty,$$  
   [because $  \int_{B^{(j)}_{1+\delta}(x_0)} |\nabla\phi_j|^2{\rm d}\mu_j =\delta^{-2}\cdot\mu_j(B^{(j)}_{1+\delta}(x_0)\backslash B^{(j)}_1(x_0) ),$]
 we get
\begin{equation*}
\begin{split}
\limsup_{j} r_jc_j{\bf \Delta}u_{i_0}(B_{r_j}(x_0))&  \ls -\int_{B^e_{1+\delta}(0)} \ip{\nabla u_{0,i_0}}{\nabla \phi_\infty}{\rm d}\mu_\infty\\
&= c_N\cdot { \Delta}u_{0,i_0}(\phi_\infty)\ls c_N\cdot{\Delta}u_{0,i_0}(B^e_{1+\delta}(0)).
\end{split}
\end{equation*} 
By combining with the fact $c_j\cdot\mu(B_{r_j}(x_0))\to c_N\omega_N$, we obtain
\begin{equation}\label{equ7.28}
 \limsup_{j\to+\infty} \frac{r_j{\bf \Delta}u_{i_0}(B_{r_j}(x_0))}{\mu(B_{r_j}(x_0))}  \ls       \frac{1}{\omega_N}\cdot{\Delta}u_{0,i_0}(B^e_{1+\delta}(0)).
\end{equation} 
 
 By replacing  the Lipschitz cut-off $\phi_j$ by  
 another    $\psi_j:X_j\to [0,1]$,     defined by
$$\psi_j(x):=
\begin{cases}\min\big\{1,\frac{1}{\delta}-\frac{d_j(x,x_0)}{\delta }\big\},&\quad x\in B_{1}^{(j)}(x_0)\subset X_j\\
0&\quad x\not\in B_{1}^{(j)}(x_0)\subset X_j,\end{cases}
$$
the same argument implies that
 ${\bf \Delta}u_{i_0}(B_{r_j}(x_0))\gs {\bf \Delta}u_{i_0}(\psi_j)$ and then
 \begin{equation}\label{ equ7.29}
\liminf_{j\to+\infty}\frac{r_j{\bf \Delta}u_{i_0}(B_{r_j}(x_0))}{\mu(B_{r_j}(x_0))} \gs \frac{1}{\omega_N}{\Delta}u_{0,i_0}(B^e_{1-\delta}(0)).
\end{equation} 
By combining the two inequalities,   ${\bf u}_{0,i_0}=\xi_{i_0}\cdot u$, (\ref{equ7.26}) and $\lim_{r\to0}\frac{\mu(B_r(x_0))}{\omega_Nr^N}=1$ (see Corollary 1.9 of \cite{DPG18}), we get
 \begin{equation*}
 \begin{split}
\xi_{i_0}\sqrt{Q_0}(1+\delta)^{N-1}&\gs\limsup_{j\to\infty} \frac{{\bf \Delta}u_{0,i_0}(B^e_{r_j}(0))}{\omega_{N-1}r_j^{N-1}}\\
&\gs  \liminf_{j\to\infty} \frac{{\bf \Delta}u_{0,i_0}(B^e_{r_j}(0))}{\omega_{N-1}r_j^{N-1}}\gs \xi_{i_0}\sqrt{Q_0}(1-\delta)^{N-1}.
\end{split}
\end{equation*} 
Since $x_0$ is a Lebesgue's point of $q_{i_0}$, letting $\delta\to0$, we get $q_{i_0}(x_0)
=\xi_{i_0}\sqrt{Q_0}.$ This completes the proof.
\end{proof}

\begin{proof}[Proof of Theorem \ref{thm1.6}.]
It follows from the combination of Proposition \ref{prop6.7} and Corollary \ref{cor7.2}.
\end{proof}

\section{Regularity of the  free boundary}

Suppose that $(X,d,\mu:=\mathscr H^N)$ is a non-collapsed $RCD(K,N)$ metric measure space with some $K\ls 0$, $N\in(1,+\infty)$, and that  ${\bf u}$ is a minimizer of $J_Q$ on a bounded domain $\Omega\subset X$ and that $Q\in C(\Omega)$ and satisfies \eqref{equ1.3}.   
In this section, we consider the regularity of the free boundary $\partial\{|{\bf u}|>0\}\cap \Omega$.

Let $x_0\in \partial\{|{\bf u}|>0\}\cap \Omega $ and $R>0$ with $B_R(x_0)\subset \Omega$. We have known that for almost all $s\in (0,R)$,  the ball $B_s(x_0)$ has  finite perimeter.  
We define the {\it Weiss' density} by
\begin{equation}\label{equ8.1}
\begin{split}
 W_{{\bf u}}(x_0,s,Q):=&\frac{1}{s^{N}}\int_{B_{s}(x_0)}\!\!\left(\left|\nabla\mathbf{u}\right|^{2}+Q\chi_{\left\{\left|\mathbf{u}\right|>0\right\}}\right)\du\\&
 -\frac{1}{s^{N+1}}\int_{X} |\mathbf{u}|^{2}{\rm d}|D\chi_{B_s(x_0)}|,
 \end{split}
\end{equation}
for almost all $s\in(0,R)$. 
\begin{lemma}\label{lem8.1}
For every $x_0\in\partial\{|{\bf u}|>0\}\cap \Omega $, the function $r\mapsto W_{\bf u}(x_0,r,Q)$ is in $L^\infty(0,R)$ provided $R\ls1$.
\end{lemma}
\begin{proof}
From the Lipschitz continuity of ${\bf u}$, we know that  $|\nabla {\bf u}|(x)\ls L$ and $|{\bf u}|(x)\ls Lr$ for all $x\in B_r(x_0)$, since ${\bf u}(x_0)=0$. By combining \eqref{equ2.12} and \eqref{equ2.13}, we have  $$-C_{N,K}L^2\ls  W_{{\bf u}}(x_0,r,Q)\ls C_{N,K}\big(L^2+Q_{\max}\big)$$
provided $R\ls1.$ This finishes the proof.
\end{proof}
Theorem \ref{thm7.1} implies the following continuity of $W$ under pointed-measured Gromov-Hausdorff topology.

\begin{lemma}\label{lem8.2}
Let $K\ls 0, N\in (1,+\infty)$ and let $(X_j,d_j,\mu_j)$ be a sequence of $ncRCD(K,N)$ metric measure spaces such that 
$(X_j,d_j,\mu_j,p_j) \overset{pmGH}{\longrightarrow}(X_\infty,d_\infty,\mu_\infty,p_\infty)$. 
Let $R>0$. Suppose that $Q_j,{\bf u}_j$ are given in Theorem \ref{thm7.1} satisfying \eqref{equ7.2} and \eqref{equ7.3} with uniform constants $Q_{\min},Q_{\max}$ and $ L $.  Assume that ${\bf u}_j$, for all $j\in\mathbb N\cap\{+\infty\}$,  satisfy the conclusions in Theorem  \ref{thm7.1}. Then for almost every $s\in(0,R)$, we have
 \begin{equation}\label{equ8.2}
  \lim_{j\to\infty} W_{{\bf u}_j}(p_j,s,Q_j)=   W_{{\bf u}_\infty}(p_\infty,s,Q_\infty).
  \end{equation}
 \end{lemma}
\begin{proof} This follows  from the combination of   (\ref{equ7.4}), (\ref{equ7.5}), Lemma \ref{lem2.16} and the fact that ${{\bf u}_j}$ converge  uniformly to ${\bf u}_\infty$ on $B_R(p_j)$.
\end{proof}

Let us recall some properties of Weiss' density in the special case where $X=\mathbb R^N$ (with Euclidean metric and Lebesgue measure $\mathscr H^N$). It is well-known \cite{Wei99,MTV17,CSY18} that for 
$x_0\in \partial\{|{\bf u}|>0\}\cap\Omega$, $W_{{\bf u}}(x_0,r,Q)$ is absolutely continuous and almost monotonicity  in $r$:
\begin{equation}\label{equ8.3}
W_{{\bf u}}(x_0,r,Q)\gs W_{{\bf u}}(x_0,s,Q)-C_{N}Q_{\max}\cdot \int_s^r\frac{{\rm osc}_{B_t(x_0)}Q}{t}{\rm d}t 
\end{equation}
for all $r>s>0$. In particular, when $Q$ is a constant then $r\mapsto W_{{\bf u}}(x_0,r,Q)$ is non-decreasing in $r$, and strictly increasing unless ${\bf u}$ is homogeneous of degree one.  Moreover, we have \cite{Wei99,MTV17,CSY18} that
\begin{equation*} 
\lim_{r\to0}W_{{\bf u}}(x_0,r,Q)=\lim_{r\to0}\frac{\ |\{|{\bf u}|>0\}\cap B_r(x_0)|\ }{|B_r(x_0)|}\gs\frac 1 2 Q(x_0)\omega_N,
\end{equation*}
 with equality if and only if $x_0$ is regular, where $|E|$ denotes the Lebesgue measure of a Borel set $E\subset \mathbb R^N$. Furthermore, there exists a constant $\epsilon_N>0$, depending only on $N$, such that: 
  \begin{equation}\label{equ8.4}
  x_0{\rm \ is\  regular }\ \ \Longleftrightarrow\ \  \lim_{r\to0}W_{{\bf u}}(x_0,r,Q)<\frac{\omega_N}{2}Q(x_0)(1+\epsilon_N).
  \end{equation}
 
 The properties of Weiss' density in the Euclidean case $X=\mathbb R^N$ suggest that,  for the general case $(X,d,\mu)$, we can define the almost regularity as follows.

\begin{definition}\label{def8.3}
 Let $\varepsilon>0$ and a point $x_0\in \partial\{|{\bf u}|>0\}\cap \Omega$. The set $\partial \{|{\bf u}|>0\}$   is called {\it $\varepsilon$-regular at $x_0$}, if   $B_1(x_0)\subset \Omega$ and  if the followings hold:
   \begin{enumerate}
    \item $x_0\in\mathcal R_{\varepsilon,1},$ that is,  $d_{pmGH}\left(B_1(x_0),B_1(0^N)\right)<\varepsilon$, where $B_1(x_0)\subset X_1:=(X,d,\mu_1^{x_0},x_0)$ given in \eqref{equ2.11} and   $B_1(0^N)$ is the unit ball in $\mathbb R^N$ centered at $0$, with measure $c_N\mathscr H^N$ given in \eqref{equ-cn}, 
    \item  we have
  \begin{equation}\label{equ8.5}
\int_{0}^{1} {\overline{W}_{{\bf u}}}(x_0,s,Q){\rm d}s<  \frac{c_N\omega_N}{2}Q(x_0)\left(1+\varepsilon\right),\footnote{Here the reason for the factor $c_N$ is  that the measure on $\mathbb R^N$ is chosen by $c_N\mathscr H^N$.}\end{equation} 
where ${\overline{W}_{{\bf u}}}(x_0,s,Q)$ is the Weiss' density of ${\bf u}$ with respect to the rescaled metric measure space $(X,d,\mu^{x_0}_1,x_0)$.\end{enumerate}
The notion that  $\partial \{|{\bf u}|>0\}$ is  \emph{$\varepsilon$-regular at $x_0$ in  the scalar $r$}  can be introduced  by scaling. That is, the set $\partial \{|{\bf u}|>0\}$   is called   $\varepsilon$-regular at $x_0$ in the scalar $r$, if on the rescaling   space $(X,r^{-1}d,\mu_r^{x_0},x_0)$ and putting ${\bf u}_r:=r^{-1}{\bf u}$,   the set $\partial \{|{\bf u}_r|>0\}$   is  {\it $\varepsilon$-regular} at $x_0$ in the ball $B_1^{(r)}(x_0)$, where $B^{(r)}_1(x_0)$ is the unit ball centered at $x_0$ in $(X,r^{-1}d,\mu_r^{x_0},x_0)$.
  \end{definition}

 We introduce some notations for the quantitative estimates for singular sets of the free boundary of ${\bf u}$. Given $\varepsilon>0$ and $r>0$, we put 
\begin{equation*}
\begin{split}
\mathcal R^{\Omega_{\bf u}}_{\varepsilon,r}&:=  {\rm all \ points\ where} \  \partial\{|{\bf u}|>0\}  \ {\rm is}  \ \varepsilon{\rm-regular\ in\ the\ scalar}\ r,\\
\mathcal R^{\Omega_{\bf u}}_{\varepsilon}&:=\cup_{r>0}\mathcal R^{\Omega_{\bf u}}_{\varepsilon,r}=\{ x\in\partial\{|{\bf u}|>0\}\cap \Omega|\ \exists r>0 \ {\rm such\ that}\ x\in \mathcal R^{\Omega_{\bf u}}_{\varepsilon,r}\},\\
\mathcal S^{\Omega_{\bf u}}_{\varepsilon}&:= (\partial\{|{\bf u}|>0\}\cap \Omega) \setminus \mathcal R^{\Omega_{\bf u}}_{\varepsilon},
\end{split}
\end{equation*}
and finally, 
$$ \mathcal R^{\Omega_{\bf u}}:= \cap_{\varepsilon>0} \mathcal R^{\Omega_{\bf u}}_{\varepsilon} \qquad {\rm and}\qquad \mathcal S^{\Omega_{\bf u}}:= (\partial\{|{\bf u}|>0\}\cap \Omega) \setminus \mathcal R^{\Omega_{\bf u}}. \qquad\qquad$$
Clearly, by Lemma \ref{lem8.2} and $Q\in C(\Omega)$,     $\mathcal R^{\Omega_{\bf u}}_{\varepsilon,r}  $ is relatively open in $\partial\{|{\bf u}|>0\}\cap \Omega $ for all $\varepsilon>0$ and $r>0$,  and  then $\mathcal R^{\Omega_{\bf u}}_{\varepsilon}$ is also relatively open.    It is easy to check $\mathcal R^{\Omega_{\bf u}}_{\varepsilon_1}\subset \mathcal R^{\Omega_{\bf u}}_{\varepsilon_2}$ for any $0<\varepsilon_1<\varepsilon_2.$

We need the following two simple facts for rescaling metric measure spaces. 

\begin{lemma}\label{lem8.4}
  Let $a,b>0$  and let  ${\bf u}_a:=a^{-1}{\bf u}$. Suppose that ${\overline{W}_{{\bf u}_a}}(x_0,s,Q)$ is the Weiss' density of ${\bf u}_a$ with respect to the rescaled  space $X_{a,b}:=(X,a^{-1}d,b\cdot \mu,x_0)$. Then
    \begin{equation}\label{equ8.6}
    {\overline{W}_{{\bf u}_a}}(x_0,s,Q)=b\cdot a^N\cdot {\overline{W}_{\bf u}}(x_0,as,Q)
    \end{equation}
    for almost all $s\in (0,R/a)$.  In particular, ${\overline{W}_{{\bf u}_r}}(x_0,s,Q)=(c^{x_0}_r/c^{x_0}_1)\cdot r^N\cdot {\overline{W}_{\bf u}}(x_0,rs,Q)$ for almost all $s\in R/r$, where $c^{x_0}_r$ is given in \eqref{equ2.11}.
     \end{lemma}
 \begin{proof}We denote $|\nabla^{(a)}v|$ be the minimal weak upper gradient of $v$ with respect to $X_{a,b}$. Then $|\nabla^{(a)}{\bf u}_a|=|\nabla{\bf u}| $. Therefore, from this and the definition of perimeter measure (see Definition \ref{def2.14}), we have $|D\chi_{B^{(a)}_s(x_0)}|=ba|D\chi_{B_{as}(x_0)}|$ for almost all $s\in (0,R/a)$. According to \eqref{equ8.1}, it is easy to check that 
  ${\overline{W}_{{\bf u}_a}}(x_0,s,Q)=b\cdot a^N\cdot {\overline{W}_{\bf u}}(x_0,sa,Q)$ for almost all $s\in(0,R/a).$
 
Notice that $X_r:=(X,r^{-1}d,\mu^{x_0}_r, x_0)$ in \eqref{equ2.11} and $\mu^{x_0}_r=c^{x_0}_r\cdot \mu=\frac{c^{x_0}_r}{c^{x_0}_1} \mu^{x_0}_1.$ 
 The second assertion follows from \eqref{equ8.6}, by taking $a=r$ and  $b=\frac{c^{x_0}_r}{c^{x_0}_1} $.  
 \end{proof}

\begin{lemma}\label{lemma-sect8+1}
Let  $(X_j,d_j,\mu_j,x_j)$ be a sequence of non-collapsed $RCD(K_j,N)$ space, $K_j\to 0$  as $j\to\infty$, and 
$$(X_j,d_j,\mu_j,x_j)\overset{pmGH}{\to} (X_\infty,d_\infty,\mu_\infty,x_\infty),\quad {\rm as} \quad j\to\infty.$$
 If
$B_1(x_\infty)\subset X_\infty$ is isometric to $B_1(0^N)$, then for any sequence $\{r_j\}$ with $r_j\to0$ as $j\to\infty$, it holds
\begin{equation}\label{equ+add+1}
d_{pmGH}\big(B^{(j)}_{1/2}(x_j),B_{1/2}(0^N)\big)\to0, \quad {as}\ \ j\to\infty,
\end{equation}
where 
 $B_{r}^{(j)}(x_j):=r^{-1}_jB_{r_j r}(x_j)$ is the ball in the rescaling spaces $(X_j,r^{-1}_jd_j,\mu_{r_j}^{x_j})$ given in \eqref{equ2.11}.
 \end{lemma}
 \begin{proof}
 This is, in fact, a consequence of the volume rigidity of non-collapsed $RCD$-spaces \cite[Theorem 1.3 and Theorem 1.6]{DPG18}. For completeness, we include a proof here. 
 
Given any $\epsilon>0$,  according to \cite[Theorem 1.6]{DPG18}, there is $\delta=\delta(\epsilon,N)>0$ such that the following hold.  Let $(Z,d_Z,\mathscr H^N)$ be a $ncRCD(-\delta,N)$ space and $\bar x\in Z$ with $\mathscr H^N(B^Z_1(\bar x))\gs \mathscr H^N(B_1(0^N))(1-\delta)$, then 
\begin{equation}\label{equation+add+2}
d_{GH}\big(B^Z_{1/2}(\bar x), B_{1/2}(0^N)\big)\ls \epsilon.
\end{equation}
From $(X_j,d_j,\mu_j,x_j)\overset{pmGH}{\to} (X_\infty,d_\infty,\mu_\infty,x_\infty) $  and the contnuity of $\mathscr H^N$ (Theorem 1.3 in \cite{DPG18}), we have that 
$$\mathscr H^N(B_1(x_j))> \mathscr H^N(B_1(0^N))(1-\delta/2)$$
for all $j\gs j_0(\delta)\in\mathbb N$.  Since $(X_j,d_j,\mu)$ is $RCD(K_j,N)$, without loss of generality, we assume that $K_j\ls0$,  the Bishop-Gromov inequality implies 
$$\mathscr H^N(B_{r_j}(x_j))\gs \frac{v_{K_j,N}(r_j)}{v_{K_j,N}(1)}\mathscr H^N(B_1(x_j))\gs \frac{v_{K_j,N}(r_j)}{v_{K_j,N}(1)}\mathscr H^N(B_1(0^N))(1-\delta/2),$$
where $$v_{K,N}(r):=\omega_N\int_0^r \left(\frac{\sinh (t\sqrt{-K/(N-1)} )}{\sqrt{-K/(N-1)}} \right)^{N-1}dt. $$
Therefore, there is $j_1:=j_1(\delta)\in \mathbb N$ such that
$$\mathscr H^N(B^{(j)}_1(x_j))=\frac{\mathscr H^N(B_{r_j}(x_j))}{r^N_j}\gs \mathscr H^N(B_1(0^N))(1-\delta)$$
for any $j\gs j_1.$  By applying (\ref{equation+add+2}) to the space $(X_j, r^{-1}_jd_j,\mu_{r_j}^{x_j})$ and $x_j$, we conclude that 
$$d_{GH}\big(B^{(j)}_{1/2}(x_j), B_{1/2}(0^N)\big)\ls \epsilon$$ 
for any sufficiently large $j$.  The arbitrariness of $\epsilon$ yields the desired (\ref{equ+add+1}).
 \end{proof}

As a consequence, we have the following lemma.

\begin{lemma}\label{lem8-5+add1}
For any $\varepsilon>0$, there exist constants $\delta:=\delta(\varepsilon| N)>0$   such that it holds: Let $(X,d,\mu)$ be a non-collapsed $RCD(-\delta,N)$-space. If $d_{pmGH}\big(B_1(x),B_1(0^N)\big)<\delta$, then $d_{pmGH}\big(B_{r/2}(x),B_{r/2}(0^N)\big)<\varepsilon\cdot r$ for all $r\in(0,1)$.
\end{lemma}
\begin{proof}
Suppose that the statement is not true, then there is a sequence of $ncRCD(-1/j,N)$-spaces $(X_j,d_j,\mu_j)$, $x_j\in X_j$ and a sequence $r_j$ with $r_j\in(0,1)$, as $j\to\infty$, such that 
\begin{equation}\label{equ8-7}
d_{pmGH}\big(B_1(x_j),B_1(0^N)\big)<j^{-1},
\end{equation}
but the rescaling balls 
\begin{equation}\label{equ8-8}
d_{pmGH}\big(B^{(j)}_{1/2}(x_j),B_{1/2}(0^N)\big)\gs \epsilon_0,\quad \forall j\in\mathbb N,
\end{equation}
for some  $\epsilon_0>0$, where 
 $B_{1/2}^{(j)}(x_j):=r^{-1}_jB_{r_j/2}(x_j)$ are the ball in the rescaling spaces $(X_j,r^{-1}_jd_j,\mu_{r_j}^{x_j})$ given in \eqref{equ2.11}.  It is obvious that $r_j\to0$, indeed, $\epsilon_0\cdot r_j\ls d_{pmGH}\big(B_{r_j/2}(x_j),B_{r_j/2}(0^N)\big) \ls d_{pmGH}\big(B_{1/2}(x_j),B_{1/2}(0^N)\big)\ls  j^{-1}$. 

Let $(X_\infty,d_\infty,\mu_\infty,x_\infty)$ be one of the limit space of the sequence $(X_j,d_j,\mu_j,x_j)$ under the $pmGH$-converging. From \eqref{equ8-7}, we know that 
$B_1(x_\infty)\subset X_\infty$ is isometric to $B_1(0^N)$.  
 
From Lemma \ref{lemma-sect8+1}, we have
\begin{equation*} 
d_{pmGH}\big(B^{(j)}_{1/2}(x_j),B_{1/2}(0^N)\big)\to0, \quad {as}\ \ j\to\infty,
\end{equation*}
which contradicts (\ref{equ8-8}). The proof is finished. 
\end{proof}

By a scaling argument, we have the following.
\begin{lemma}\label{lem8-5}
For any $\varepsilon>0$, there exist constants $\delta:=\delta(\varepsilon|K, N)>0$ and $r_0:=r_0(\varepsilon| K,N)\in(0,1)$   such that it holds: Let $(X,d,\mu)$ be a non-collapsed $RCD(K,N)$-space. If $d_{pmGH}\big(B_1(x),B_1(0^N)\big)<\delta$, then $d_{pmGH}\big(B_r(x),B_r(0^N)\big)<\varepsilon\cdot r$ for all $r\in(0,r_0)$.
\end{lemma}
\begin{proof}
Choose $\delta=\delta(\varepsilon/2|N)$ given by Lemma \ref{lem8-5+add1}, and let $\bar r$ such that $\bar r^2|K|\ls \delta$ and $\bar r\ls1$. By applying Lemma \ref{lem8-5+add1} to $(X,\bar r^{-1}d,\mu^{\bar r}_x,x)$, we conclude
$$d_{pmGH}\big(B_{\bar r r/2}(x),B_{\bar r r/2}(0^N)\big)= \bar r \cdot d_{pmGH}\big(B^{(\bar r)}_{r/2}(x),B_{r/2}(0^N)\big)<\frac{\varepsilon}{2}\cdot r \bar r$$ for all $r\in(0,1).$ By setting $r_0:=\bar r/2$, it follows the conclusion.
\end{proof}
  
\begin{theorem}[$\varepsilon$-regularity]\label{thm8-6}
For any $\varepsilon>0$, there exists a positive constant   $\delta:=\delta(\varepsilon| K,N,R, Q_{\min},Q_{\max},L)>0$  such that the following holds: 

Let $(X,d,\mu)$ be an $ncRCD(K,N)$-space and  let ${\bf u}$ be a minimizer of $J_Q$ on $B_R(x_0)\subset X$ with $R\gs 2$ and $Q,{\bf u}$ satisfying  \eqref{equ7.2}, \eqref{equ7.3}.    If    ${\rm osc}_{B_2(x_0)}Q<\delta$ and if  $x\in\mathcal R^{\Omega_{\bf u}}_{\delta,1}$  for all $x\in B_1(x_0)\cap \partial\{|{\bf u}|>0\} $, then   
$y\in \mathcal R^{\Omega_{\bf u}}_{\varepsilon,r}$  for  all $\ y\in B_{r_0/4}(x_0)\cap \partial \{|{\bf u}|>0\} $ and all $r\in(0,r_0/4),$ where $r_0$ is given in Lemma \ref{lem8-5}.
\end{theorem}

\begin{proof}
We argue by contradiction. Suppose not, then  there exists $\epsilon_0>0$ such that for each $j\in\mathbb N$, there is an  $ncRCD(K,N)$-spaces $(X_j,d_j,\mu_j)$ and a minimizer ${\bf u}_j$  of $J_{Q_j}$ on $B_R(x_j)\subset X_j$ with $Q_j$ satisfying the uniform estimates \eqref{equ7.2}, \eqref{equ7.3}  such that followings hold:
    \begin{enumerate}
\item[(i)]  ${\rm osc}_{B_2(x_j)}Q_j\ls j^{-1}$,
\item[(ii)]  $x'_j\in \mathcal R^{\Omega_{{\bf u}_j}}_{j^{-1},1}$ for all $x'_j\in B_1(x_j)\cap \partial\{|{\bf u}_j|>0\}$,
\item[(iii)]  there exist $y_j\in B_{r_0/4}(x_j)\cap\partial\{|{\bf u}_j|>0\}$ and  $r_j\in(0,r_0/4)$ such that 
 $y_j\not\in \mathcal R^{\Omega_{{\bf u}_j}}_{\epsilon_0,r_j}.$
\end{enumerate}

 By applying the standard compactness of $RCD$-spaces and Theorem 7.1 to $X_j,Q_j$ and ${\bf u}_j$, there exists subsequences of $\{y_j\}, {X}_j$ and $\{{\bf  u}_j\}$ such that:
\begin{equation}
  ({X}_j,d_j,\mu^{y_j}_1,y_j)\overset{pmGH}{\to}  {X}_\infty:=(X_\infty,d_\infty,\mu_\infty,y_\infty),
  \end{equation}
 and $ {\bf u}_j\to  {\bf u}_\infty$ uniformly in any ${B}_s(y_j)$ for all $s<1,$ where $\mu^{y_j}_1$ is given in \eqref{equ2.11}. The limit map ${\bf u}_\infty$ is a minimizer of $J_{Q_\infty}$ on $B_1(y_\infty)$. 
By  $y_j\in \mathcal R^{\Omega_{{\bf u}_j}}_{j^{-1},1}$ and Definition \ref{def8.3}(1),  we conclude that   the limit ball $B_1(y_\infty)\subset {X}_\infty$ is isometric to $B_1(0^N)\subset \mathbb R^N$ with the Euclidean metric $d_e$ and the measure $c_N\mathscr H^N$. The limit $Q_\infty:=\lim_{j\to+\infty} Q_j$ is a constant (because $ {\rm osc}_{B_1(y_j)} Q_j\ls {\rm osc}_{B_2(x_j)}Q_j\ls 1/j$). 
By Lemma \ref{lem8.2}, we get ${\overline{W}_{ {\bf u}_j}}(y_j,s,Q_j)\to {\overline{W}_{ {\bf u}_\infty}}(y_\infty,s,Q_\infty)$ for almost all $s\in (0,1)$. By integrating on $(0,1)$ and using Lemma \ref{lem8.1}, we obtain 
\begin{equation*} 
\begin{split}
\int^1_{0}{\overline{W}_{ {\bf u}_\infty}}(y_\infty,s,Q_\infty){\rm d}s &=\lim_{j\to \infty}\int^1_{0}{\overline{W}_{ {\bf u}_j}}(y_j,s,Q_j){\rm d}s\\
&\ls \frac{1}{2}\lim_{j\to\infty}Q_j(y_j)c_N\cdot\omega_N\big(1+j^{-1}\big)= \frac{c_N\omega_N}{2} Q_\infty,
\end{split}
\end{equation*}
where we have used $y_j\in \mathcal R^{\Omega_{{\bf u}_j}}_{j^{-1},1}$ and Definition \ref{def8.3}(2).  From the fact that the limit $B_1(y_\infty)$ is the Euclidean ball and  the monotonicity  \eqref{equ8.3}, we get that   $y_\infty$ is a regular point in the free boundary of ${\bf u}_\infty$ and that
\begin{equation}\label{equ8-10}
 {\overline{W}}_{{\bf u}_\infty}(y_\infty,s,Q_\infty)= \frac{c_N}{2}Q_\infty\omega_N,\quad \forall\ s\in(0,1).
 \end{equation}

Now let us consider the rescaled spaces 
$ \overline{X}_j:=\big(X_j,\overline{d_j}:=r^{-1}_jd_j,\mu^{y_j}_{r_j},y_j\big) $
and maps $\overline{\bf u}_j:=r^{-1}_j{\bf u}_j$, where $\mu^{y_j}_{r_j}$ is given in \eqref{equ2.11}. 
Remark that the Lipschitz constant of $\overline{{\bf u}}_j$ is the same as the one of ${\bf u}_j$ for each $j\in\mathbb N$.       By applying Theorem 7.1 to $\overline{X}_j,Q_j:=Q$ and $\overline{\bf u}_j$, there exists subsequences of $\{r_j\}, \{y_j\}, \overline{X}_j$ and $\{\overline{\bf  u}_j\}$ such that:
  $$ r_j\to \bar r\in[0,r_0/4],\qquad  \overline{X}_j \overset{pmGH}{\to} \overline{X}_\infty:= (\overline{X}_\infty,\overline{d_\infty},\overline{\mu_\infty},\overline{y_\infty}),$$
 and $ \overline{\bf u}_j\to \overline{\bf u}_\infty$ uniformly in any ${B}^{(j)}_s(y_j)$ for all $s<1,$ where $B^{(j)}_s(y_j)$ denotes the ball in $\overline{X}_j$. 

 From $y_j\in \mathcal R_{j^{-1},1}$ (see Definition \ref{def8.3}(1)) and Lemma \ref{lem8-5}, we know that 
$$d_{pmGH}\big(B_{r_j}(y_j),B_{r_j}(0^N)\big)<\frac{\epsilon_0}{2}\cdot r_j$$
for all sufficiently large $j$. That is, $y_j\in \mathcal R_{\epsilon_0/2,r_j}$ for all $j$ large enough. By combining with the condition   $y_j\not\in \mathcal R^{\Omega_{{\bf u}_j}}_{\epsilon_0,r_j},$ and by using Lemma \ref{lem8.2},  we obtain
 \begin{equation}\label{equ8-11}
 \int_0^1\overline{W}_{\overline{\bf u}_\infty}(y_\infty,s,Q_\infty){\rm d}s\gs \frac{c_N}{2}Q_\infty\omega_N(1+\epsilon_0).
  \end{equation}
  
From Lemma \ref{lemma-sect8+1},  the ball $B_{1/2}(\overline{y_\infty})\subset \overline{X}_\infty$ is isometric to $B_{1/2}(0^N)$. Up to a subsequence, $\overline{{\bf u}}_\infty  $ is the limit of $r^{-1}_j{\bf u}_\infty$. Thus, from Lemma \ref{lem8.4}, $\mu^{y_\infty}_{r_j}=\frac{c^{y_\infty}_{r_j}}{c^{y_\infty}_1} \mu_1^{y_\infty}$ and $ \mu_1^{y_\infty}=\mu_\infty=c_N\mathscr H^N$, we get, for all $s\in (0,1)$ that  
$$\overline{W}_{r^{-1}_j {\bf u}_\infty}(y_\infty,s,Q_\infty) =  \frac{c^{y_\infty}_{r_j}}{c^{y_\infty}_1}\cdot r^N_{j}\cdot \overline{W}_{{\bf u}_\infty}(y_\infty,r_js,Q_\infty).$$
By using \eqref{equ8-10},  $r_j\ls 1 $ and
\begin{equation*}
\begin{split}
\frac{c^{y_\infty}_{r_j}}{c^{y_\infty}_{1}}&=\frac{1}{c^{y_\infty}_{1}\int_{B_{r_j}(y_\infty=0^N)}\big(1-r_j^{-1}d_e(z,y_\infty)\big){\rm d}\mathscr H^N(z)}\\ &=\frac{1}{c^{y_\infty}_{1}\int_{B_{r_j}(0^N)}\big(1-r_j^{-1}|z|\big){\rm d}\mathscr H^N(z)}=\frac{1}{r^N_j},
\end{split}
\end{equation*}
we conclude that  $\overline{W}_{r^{-1}_j {\bf u}_\infty}(y_\infty,s,Q_\infty) =\frac{c_N}{2}Q_\infty\omega_N$ for all $s\in(0,1)$, and hence $\overline{W}_{\overline{\bf u}_\infty}(y_\infty,s,Q_\infty) =\frac{c_N}{2}Q_\infty\omega_N$ for all $s\in(0,1)$. By integrating on $(0,1)$, it contradicts with (\ref{equ8-11}), and hence, the proof is completed.
 \end{proof}

This $\varepsilon$-regularity is the reason for us to define the almost regular part of the free boundary $\mathcal R^{\Omega_{\bf u}}_\varepsilon=\cup_{r>0}\mathcal R^{\Omega_{\bf u}}_{\varepsilon,r}.$
  A simple but important corollary of this definition is that singular points do not disappear under $pmGH$-converging as follows.
\begin{lemma}\label{lem8-7}
Let  $(X_j,d_j,\mu_j)$ be a sequence of $ncRCD(K,N)$-spaces and $(X_j,d_j,\mu_j,p_j)\overset{pmGH}{\longrightarrow}(X_\infty,d_\infty,\mu_\infty,p_\infty)$. Let  $Q_j\in C(B_R(x_j))$ and   ${\bf u}_j$  be a minimizer  of $J_{Q_j}$ on $B_R(x_j)\subset X_j$. Suppose that   $Q_j$ and ${\bf u}_j$ satisfy  the uniformly estimates \eqref{equ7.2}, \eqref{equ7.3} and that ${\bf u}_j$ converges  uniformly to ${\bf u}_\infty$ on $B_R(p_j)$.  Then, for any $\varepsilon>0$, if $x_j\in \mathcal S^{\Omega_{{\bf u}_j}}_{\varepsilon}\cap B_{R/2}(p_j)$ and $x_j\overset{GH}{\longrightarrow} x_\infty$, we have $x_\infty\in \mathcal S^{\Omega_{{\bf u}_\infty}}_{\varepsilon}.$
\end{lemma}
\begin{proof}From (\ref{equ7.6}) in Theorem \ref{thm7.1}, we know that $x_\infty\in \partial\{|{\bf u}_\infty|>0\}\cap B_{R/2}(p_\infty)$. 

Fix $ r> 0$ arbitrarily. We know that $x_j\not\in    \mathcal R^{\Omega_{{\bf u}_j}}_{\varepsilon,r},$ from $ \mathcal R^{\Omega_{{\bf u}_j}}_{\varepsilon}=\cup_{r>0}\mathcal R^{\Omega_{{\bf u}_j}}_{\varepsilon,r}$. By the definition of   $\mathcal R^{\Omega_{{\bf u}_j}}_{\varepsilon,r}$, we know either
 $$d_{pmGH}\left(B_1^{r^{-1}d_j}(x_j) ,B_1(0^{N-1})\right)\gs \varepsilon,$$
 where $B_1^{r^{-1}d_j}(x_j)$ is the unit ball centered at $x_j$ on the rescaling space $(X_j,r^{-1}d_j,\mu_{j,r}^{p_j},p_j)$, or 
     \begin{equation*}
\int_{0}^{1}\overline{W}_{\overline{\bf u}_j}(x_j,s,Q){\rm d}s\gs  \frac{C_N\cdot \omega_N}{2}Q(x_j)\left(1+\varepsilon\right), 
\end{equation*} 
where $\bar {\bf u}_j:=r^{-1}{\bf u}_j$. From (\ref{equ7.5}), Lemma \ref{lem8.2} and $Q(x_j)\to Q(x_0)$, we get  $x_\infty\not\in \mathcal R^{\Omega_{r^{-1}{\bf u}_\infty}}_{\varepsilon,1}$, and then $x_\infty\not\in\mathcal R^{\Omega_{{\bf u}_\infty}}_{\varepsilon}.$    The proof is finished.
\end{proof}

 A similar argument in the proof of the $\varepsilon$-regularity also gives the following Reifenberg's property.
\begin{lemma}\label{lem8-8}
 For any $\varepsilon>0$, there exists a constant $\delta=\delta(\varepsilon | N,K,R,L,Q_{\max},Q_{\min})>0$  such that the following holds: Let $(X,d,\mu)$ be an $ncRCD(K,N)$-space and let ${\bf u}$ be a minimizer of $J_Q$ on $B_R(x) \subset X$ with $R\gs1$, $0<Q_{\min}\ls Q\ls Q_{\max}<+\infty$  and $|\nabla {\bf u}|\ls L.$ 
 If $x\in \mathcal R^{\Omega_{\bf u}}_{\delta,1}$ and if ${\rm osc}_{B_1(x)}Q\ls\delta$, then 
  $$d_{GH}\left(B_1(x)\cap \partial\{|{\bf u}|>0\},B_1(0^{N-1})\right)<\varepsilon,$$  where $B_1(0^{N-1})$ is the unit ball in $\mathbb R^{N-1}$ centered at $0$.
  \end{lemma}
\begin{proof}
Suppose that this assertion is not true. There exists some $\epsilon_0>0$ such that   for each $j\in\mathbb N$, there is an  $ncRCD(K,N)$-spaces $(X_j,d_j,\mu_j)$, $Q_j\in C(B_R(x_j))$ and a minimizer ${\bf u}_j$  of $J_{Q_j}$ on $B_R(x_j)\subset X_j$ with the uniform estimates \eqref{equ7.2}, \eqref{equ7.3}  such that:\\
$\quad (i)$  
$x_j\in \mathcal R^{\Omega_{{\bf u}_j}}_{j^{-1},1}$, ${\rm osc}_{B_1(x_j)}Q_j\ls j^{-1}$ and \\
$(ii)$  $d_{GH}\left(B_1(x_j)\cap \partial\{|{\bf u}_j|>0\},B_1(0^{N-1})\right)\gs\epsilon_0.$ 
 
 By applying the standard compactness of $RCD$-spaces and Theorem 7.1 to $X_j,Q_j$ and ${\bf u}_j$, there exist  subsequences of $\{x_j\}, {X}_j$ and $\{{\bf  u}_j\}$ such that:
\begin{equation*}
  ({X}_j,d_j,\mu^{x_j}_1,x_j)\overset{pmGH}{\to}  {X}_\infty:=(X_\infty,d_\infty,\mu_\infty,x_\infty),
  \end{equation*}
 and $ {\bf u}_j\to  {\bf u}_\infty$ uniformly in any ${B}_r(x_j)$ for all $r<R,$ where $\mu^{x_j}_1$ is given in \eqref{equ2.11}. The limit map ${\bf u}_\infty$ is a minimizer of $J_{Q_\infty}$ on $B_1(x_\infty)$, and from (\ref{equ7.6}) that
 \begin{equation}\label{equ8-12}
 d_{GH}\left(B_1(x_\infty)\cap \partial\{|{\bf u}_\infty|>0\},B_1(0^{N-1})\right)\gs \epsilon_0.
 \end{equation}
On the other hand, since  $x_j\in \mathcal R^{\Omega_{{\bf u}_j}}_{j^{-1},1}$, similar to the proof of (\ref{equ8-10}), by Definition \ref{def8.3} and Lemma \ref{lem8.2},  we conclude that   the limit ball $B_1(y_\infty)\subset {X}_\infty$ is isometric to $B_1(0^N)\subset \mathbb R^N$ with the Euclidean metric $d_e$ and the measure $c_N\mathscr H^N$, and that \begin{equation*} 
\begin{split}
\int^1_{0}{\overline{W}_{ {\bf u}_\infty}}(x_\infty,s,Q_\infty){\rm d}s =\frac{c_N\omega_N}{2} Q_\infty,
\end{split}
\end{equation*}
where  the limit $Q_\infty:=\lim_{j\to+\infty} Q_j$ is a constant (because $\lim_{j\to+\infty}{\rm osc}_{B_1(x_j)}Q_j=0$). From this and  the monotonicity  \eqref{equ8.5}, we get that   
$ {\overline{W}}_{{\bf u}_\infty}(x_\infty,s,Q_\infty)= \frac{c_N}{2}Q_\infty\omega_N,$ for all $ s\in(0,1).$
 Thus,  we obtain that  $x_\infty$ is a regular point and that  ${\bf u}_\infty$ is homogeneous of degree one. This implies 
$$d_{GH}\left(B_1(x_\infty)\cap \partial\{|{\bf u}_\infty|>0\},B_1(0^{N-1})\right)=0.$$ 
This contradicts with \eqref{equ8-12}. The proof is finished.
\end{proof}

Consequently, we have the following topological regularity for the almost regular part of the free boundary $\partial\{|{\bf u}|>0\}\cap\Omega $.  
\begin{corollary}\label{cor8-9}
Suppose that $(X,d,\mu:=\mathscr H^N)$ is a non-collapsed $RCD(K,N)$ metric measure space, and that  ${\bf u}$ is a minimizer of $J_Q$ on a bounded domain $\Omega\subset X$ and that $Q\in C(\Omega)$ satisfies \eqref{equ1.3}. 
Then for any $\varepsilon>0$, there exists $\delta>0$ such that the set $\mathcal R_{\delta}^{\Omega{\bf u}}$ satisfies the following property: for any $x\in \mathcal R^{\Omega_{\bf u}}_\delta$, there exists  $r_x>0$ such that it holds for all $y\in B_{r_x}(x)\cap \partial\{|{\bf u}|>0\}$ and all $r\in(0,r_x)$ that
    \begin{enumerate}
\item[(a)] $y\in \mathcal R^{\Omega_{\bf u}}_{\varepsilon,r} $, and
\item[(b)] $d_{GH}\left(B_r(y)\cap \partial\{|{\bf u}|>0\},B_r(0^{N-1})\right)< \varepsilon  r.$
\end{enumerate}

Consequently,  for any $\alpha\in(0,1)$, there exists   $\delta_\alpha>0$ such that the almost regular set $\mathcal R_{\delta_\alpha}^{\Omega_{\bf u}}$ is a $C^\alpha$-biH\"older homeomorphic to an $(N-1)$-dimensional topological manifold. 
\end{corollary}
\begin{proof} 
For any $\delta>0$ sufficiently small,  for each $x\in \mathcal R^{\Omega_{\bf u}}_{\delta}$, there exists $r'_x>0$ such that $x\in \mathcal R^{\Omega_{\bf u}}_{\delta,r'_x}$. Notice that     $\mathcal R^{\Omega_{\bf u}}_{\delta,r}$ is relatively open in $\partial\{|{\bf u}|>0\}\cap \Omega$ for any $r>0$. Hence, there exists  a neighborhood $B_{r''_x}(x)$ such that $x'\in \mathcal R^{\Omega_{\bf u}}_{\delta,r'_x}$   for all $x'\in \partial\{|{\bf u}|>0\}\cap B_{r''_x}(x)$ and ${\rm osc}_{B_{2r''_x}(x)}Q\ls \delta$. We can assume that $r''_x\ls r'_x$. Thus, by $\varepsilon$-regularity Theorem \ref{thm8-6}, we conclude that 
$$y\in\mathcal R^{\Omega_{\bf u}}_{\delta',r} \quad {\rm for\ all\ }\ y\in B_{r''_x/4}(x)\cap \partial\{|{\bf u}|>0\} \ {\rm and\ all }\   r\in(0,r''_x/4),$$
where $\delta'=\delta'(\delta)>0$ with $\lim_{\delta\to0}\delta'(\delta)=0$. According to Lemma \ref{lem8-8}, we have
  $$d_{GH}\left(B_r(y)\cap \partial\{|{\bf u}|>0\},B_r(0^{N-1})\right)<\delta''\cdot r,$$  
  for all $r\in (0,r''_x/4)$ and all $y\in B_{r''_x/4}(x),$ where $\delta''=\delta''(\delta')>0$ with $\lim_{\delta'\to0}\delta''(\delta')=0.$
We put $r_x:=r_x''/4$ and take $\delta$  sufficiently small that $\max\{\delta'(\delta),\delta''(\delta')\}\ls \varepsilon.$  Now the first assertion follows.

The second assertion from the first one and the Reifenberg's disk theorem for metric spaces, see \cite[Apendix A]{CC97}. In fact, for any $\alpha\in(0,1)$, we know that  
  $B_{r_x}(x)\cap\partial\{|{\bf u}|>0\}\cap\Omega$ is $C^\alpha$-homeomorphic to the ball $B_{r_x}(0^{N-1})$ provided that $\varepsilon<\varepsilon(\alpha)$, where  $\varepsilon(\alpha)>0$ such that $\varepsilon(\alpha)\to0$ as $\alpha\to1^-$.  The proof is finished.
    \end{proof}

In the rest of this section, we want to estimate the size of the singular part of  $\partial\{|{\bf u}|>0\}\cap\Omega$. Firstly, we need to deal with the minimizers of $J_Q$ on metric measure cones. Let $N\gs 2$ and let $(\Sigma,d_\Sigma,\mu_\Sigma)$ be a   metric measure space with $diam(\Sigma)\ls \pi$. The  metric measure cone over $\Sigma$ is the metric measure space $(C(\Sigma),d_C,\mu_C)$ given  by $$C_N(\Sigma)= [0,\infty)\times\Sigma/(\{0\}\times \Sigma)$$
with the   distance   
$$d_C\big((r_1,\xi_1),(r_2,\xi_2)\big)=\sqrt{r_1^2+r^2_2-2r_1r_2\cos d_{\Sigma}(\xi_1,\xi_2)},$$
and the    measure 
$$d\mu_{C}(r, \xi)=r^{N-1}dr\times d\mu_{\Sigma}(x).$$ 
In the following, we always assume that  $\big(C(\Sigma),d_C,\mu_C\big)$ is an $ncRCD(0,N)$-space.  Remark that,  for any point $x_0$ of an $ncRCD(K,N)$-space, any tangent cone at $x_0$  must be a metric measure cone satisfying $ncRCD(0,N)$ (see \cite{DPG18}).  
 
\begin{lemma}\label{lem8-10}
Let ${\bf u}$ be  a global minimizer of $J_{Q_0}$ on a cone $(C(\Sigma),d_C,\mu_C)$ with the vertex $p$, where $Q_0>0$ is a constant. Then the  Weiss' density $W_{\bf u}(p,r,Q_0)$ is non-decreasing  in $r$; moreover, if $W_{\bf u}(p,r,Q_0)$ is a constant then ${\bf u}$ is homogeneous of degree one, i.e. ${\bf u}(\xi,t)=t\cdot{\bf u}(\xi,1)$ for any $t>0$ and $\xi\in\Sigma$. 

In particular,  if ${\bf u}_0$ is one of the blow-up limits of a minimizer ${\bf u}$ at a point $x_0$ in an $ncRCD(K,N)$-space $(X,d,\mu),$ then any doube-blowup ${\bf u}_{00}=\lim_{j\to\infty}{\epsilon_j}^{-1}u_0$ must be homogeneous of degree one.
\end{lemma}
\begin{proof}
The first assertion is similar to the case of Euclidean space. For the completeness, we give the details in the Appendix \ref{app} (see Lemma \ref{lem-a3}).

 Since ${\bf u}_0$ is one of   blow up limits of ${\bf u}$,  there exists a  sequence $r_j\to 0$ such that  $(Y,d_Y,\mu_Y,o_Y)$ is the $pmGH$-limit of $(X,r^{-1}_jd,\mu^{x_0}_{r_j},x_0)$  and that ${\bf u}_0$ is the limit of $r_j^{-1}{\bf u}$.  For the second assertion, we only need to check that the Weiss' density $W_{{\bf u}_0}(o_Y,r,Q(x_0))$ is a constant.

By taking any subsequence    $\epsilon_j \to0$ and letting  ${\bf u}_{0,\epsilon_j}(\xi,s):=\epsilon_j^{-1}{\bf u}_0(\xi, \epsilon_j\cdot s)$, the double-blowup blow up   ${\bf u}_{00}$ is  the limit of  ${\bf u}_{0,\epsilon_j}(\xi,s)$ under the rescaling space $(Y, \epsilon_j^{-1}d_Y,\mu^{o_Y}_{\epsilon_j},o_Y)$ as $\epsilon_j\to0^+$. 
By Lemma \ref{lem8.4}, for such  sequence  $\epsilon_j$, we obtain for any $s>0$ that
\begin{equation*}
\begin{split}
\overline{W}_{{\bf u}_0}(o_Y,s,Q(x_0))&=\lim_{\epsilon_j\to0}c^{o_Y}_{\epsilon_j}/C^{o_Y}_1\cdot\epsilon_j^N\cdot \overline{W}_{{\bf u}_0}(o_Y,\epsilon_j\cdot s,Q(x_0))\\
&=\lim_{s\to0}  \overline{W}_{{\bf u}_0}(o_Y,  s,Q(x_0)),
\end{split}
\end{equation*}
where we have used $c^{o_Y}_r/c^{o_Y}_1\cdot r^N=1$ for all $r>0$ and the existence of the limit $\lim_{s\to0}  \overline{W}_{{\bf u}_0}(o_Y,  s,Q(x_0))$, by the monotonicity in the first assertion.
\end{proof}

\begin{lemma}\label{lem8-11}
Let $Q_0$ be a positive constant and ${\bf u}=(u_1, u_2,\cdots, u_m)$ be a minimizer of $J_{Q_0}$ on a two-dimensional cone $C(S_a)$, where $S_a$ is a circle with length $a\in(0,2\pi)$. Assume that ${\bf u}$ is homogeneous of degree one. 
Then the vertex $o\not\in \partial\{|{\bf u}|>0\}.$

Consequently, the singular set $\mathcal S^{\Omega_{\bf u}}$ is empty for  any minimizer ${\bf u}$ of $J_{Q_0}$  on a bounded domain $\Omega$ of an $ncRCD(K,2)$-space without boundary. 
\end{lemma}
\begin{proof}
When $m=1$, this assertion is the main result in \cite{ACL15}.

  For $m\gs 1$, we shall reduce it to the case $m=1,$  by an argument in \cite{CL08}.
Since ${\bf u}$ is homogeneous of degree one, the set $\{ |{\bf u}|>0\}$ is a cone over an interval $(b_1,b_2)\subset(0,a)$. By Lemma \ref{lem4.4}, we know that $u_i$ is harmonic on $\{|{\bf u}|>0\}$ for each $i=1,2,\cdots,m.$ From  the fact that ${\bf u}(\xi,r)$ is homogeneous on $r$ of degree one and that  any locally Lipschitz continuous function $v$ on $C(S_a)$ satisfies  
$$\Delta_{C(s_a)}v=r^{-2}\Delta_{S_a}v+r^{-1}\frac{\partial v}{\partial r}+\frac{\partial^2v}{\partial r^2}$$
in the sense of distributions, we know that $u_i(\xi,1)$ is a Dirichlet eigen-function of $\Delta_{S_a}$ on $(b_1,b_2)$ with respect to the eigenvalue $\lambda=1$, for all $i=1,2,\cdots, m.$ Notice that $u_i(\xi,1)\gs 0$ and the fact that the first Dirichlet eigenvalue of $\Delta_{S_a}$ is single. Thus, $u_i(\xi,1)/u_1(\xi,1)=c_i$ for some  constant $c_i>0$, for all $i=2,3,\cdots,m.$ Combining with the fact that  ${\bf u}$ is homogeneous, we get
$${\bf u}=(u_1,c_2u_1,\cdots, c_mu_1).$$
Thus, from the minimality of ${\bf u}$, it is clear that $u_1$ is a minimizer of $J_{Q_{\bf c}}$ with the constant
$$Q_{\bf c}:=\frac{Q_0}{1+\sum_{i=2}^mc^2_i}.$$
According to \cite{ACL15}, we get the vertex  $o\not\in\partial\{u_1>0\}=\partial\{|{\bf u}|>0\}.$ 

For the second assertion. Suppose not, if there exists a minimizer ${\bf u}$ on a 2-dimensional $ncRCD(K,2)$-space without boundary, such that it has a singular point $x_0$. Then we blow up the spaces $(X,2^jd,\mu^{x_0}_{2^{-j}},x_0)$ and the maps ${\bf u}_j:=2^{j}{\bf u}$. By Theorem \ref{thm7.1} and Lemma \ref{lem8-7}, up to a subsequence, we can obtain a blow-up limit map ${\bf u}_0$ on one of the tangent cones at $x_0$ such that ${\bf u}_0$ is a minimizer and has a singular point at the vertex $o$. This contradicts the first assertion. The proof is finished.
\end{proof}

We shall estimate the size of the singular part of $\partial\{|{\bf u}|>0\}\cap\Omega $ by a variant of the classical dimension reduction argument. See \cite{F69} and \cite{Giu84} for the case of perimeter minimizers in the Euclidean setting, [Wei99] for the case of free boundary problems in the Euclidean setting,  \cite{DPG18} for the dimension bounds for the singular strata on non-collapsed RCD spaces,  and \cite{MS21} for the dimension bounds of the singular part of perimeter minimizers in the non-collapsed RCD spaces.

\begin{theorem}\label{thm8-12}
Suppose that $(X,d,\mu:=\mathscr H^N)$ is a non-collapsed $RCD(K,N)$ metric measure space with $N\gs 3$, and that  ${\bf u}$ is a minimizer of $J_Q$ on a bounded domain $\Omega\subset X$ and that $Q\in C(\Omega)$ satisfies \eqref{equ1.3}.   Assume that $  \Omega  \cap \partial X=\emptyset$. Then for any $\varepsilon>0$,  
\begin{equation}\label{equ8-13}
\dim_{\mathscr H}\left( \mathcal S^{\Omega_{\bf u}}_{\varepsilon}\right)\ls N- 3.
 \end{equation}
Moreover,  if $N=3 $, then $\mathcal S^{\Omega_{\bf u}}_{\varepsilon}$ contains   at most isolated points.\footnote{Recently, Wang \cite{W26} showed that $\dim(\mathcal S^k({\bf u}))\ls k$ for any $k=0,\cdots, n-2$, where $\mathcal S^k({\bf u})$ is the $k$-singular stratum of the singular set of the free boundary of ${\bf u}$, see \cite[Definition 2.24]{W26}. The quantitative volume estimate on the singular set $\mathcal S({\bf u})$ is given in \cite{BW26}.  }
   \end{theorem}

\begin{proof}
Fix any $\varepsilon>0$.   Assume that $\mathscr H_{N- 3 +\eta} \left(\mathcal S^{\Omega_{\bf u}}_{\varepsilon}\right)>0$ for some $\eta>0$. Then there exists a point $x_0\in \mathcal S^{\Omega_{\bf u}}_{\varepsilon}$ such that (see for example Lemma 3.6 in \cite{DPG18}):
$$\limsup_{r\to0} \frac{\mathscr H^{N- 3 +\eta}_\infty\left(  \mathcal S^{\Omega_{\bf u}}_{\varepsilon}\cap B_r(x_0)\right)}{r^{N- 3 +\eta}}\gs C_0:= 2^{-N+ 3 -\eta}\omega_{N- 3 +\eta},$$
where $\mathscr H^{N- 3 }_\infty$ is the $\infty$-Haudorff premeasure. Let $r_j$ be a sequence such that $r_j\to0$ and 
\begin{equation}\label{equ8-14}
\frac{\mathscr H^{N- 3 +\eta}_\infty\left( \mathcal S^{\Omega_{\bf u}}_{\varepsilon}\cap B_{r_j}(x_0)\right)}{r_j^{N- 3 +\eta}}\gs C_0/2>0.
\end{equation}
Now we consider the blow-up sequence of pointed metric measure spaces $X_j:=(X,r_j^{-1}d,\mu^{x_0}_{r_j},x_0)$ and let ${\bf u}_j:=r^{-1}_j{\bf u}$. From (\ref{equ8-14}), we have  
\begin{equation}\label{equ8-15}
\mathscr H^{N- 3 +\eta}_\infty\left(\mathcal S^{\Omega_{{\bf u}_j}}_{\varepsilon}\cap B_{1}^{(j)}(x_0)\right)  \gs C_0/2,
\end{equation}
where $B^{(j)}_1(x_0)$ is the unit ball in $X_j$. By Theorem \ref{thm7.1}, up to a subsequence of $\{r_j\}$, the $X_j$ converges to a tangent cone at $x_0$ in the $pmGH$-topology, denoted by $\big( C(Y),d_C,\mu_C, o_Y)$, and ${\bf u}_j$ converges to a blow up limit ${\bf u}_0$ defined on $C(Y)$, which is a global minimizer of $J_{Q_0:=Q(x_0)}$  on $C(Y)$. 
   By using the upper semicontinuity of the $\infty$-Hausdorff premeasure $\mathscr H^{N- 3 +\eta}_\infty$  under GH-convergence (see \cite{DPG18}) and Lemma \ref{lem8-7}, we get
\begin{equation*}
   \mathscr H^{N- 3 +\eta}_\infty\left(\mathcal S^{\Omega_{{\bf u}_0}}_{\varepsilon}\cap B_{1}(o_Y)\right)  \gs C_0/2.
\end{equation*}
 This implies 
 \begin{equation}\label{equ8-16}
   \mathscr H^{N- 3 +\eta} \left(\mathcal S^{\Omega_{{\bf u}_0}}_{\varepsilon}\cap B_{1}(o_Y)\right) >0.
   \end{equation}
   
Since $N-  3 +\eta>0$, it follows that there exists a point $x_1\not =o_Y$ such that  $x_1\in \mathcal S^{\Omega_{{\bf u}_0}}_{\varepsilon}\cap B_{1}(o_Y)$ 
and
$\limsup_{r\to0} \frac{\mathscr H^{N- 3 +\eta}_\infty\left( \mathcal S^{\Omega_{{\bf u}_0}}_{\varepsilon}\cap B_{r}(x_1)\subset C(Y)\right)}{r^{N- 3 +\eta}}\gs C_0.$ By the same argument, 
we shall blow up again ${\bf u}_0$ at $x_1$ along a  sequence ${s_j}\to 0^+$ such that  
$$ \frac{\mathscr H^{N- 3 +\eta}_\infty\left( \mathcal S^{\Omega_{{\bf u}_0}}_{\varepsilon}\cap B_{s_j}(x_1)\subset C(Y)\right)}{{s_j}^{N- 3 +\eta}}\gs C_0/2.$$
 We consider the blow up sequence of metric measure spaces $C_{s_j}:=\big(C(Y),d_j:=s_j^{-1}d_C,\mu_j:=\mu_{C,s_j}^{x_1},x_1\big),$ and the blow up sequence of maps $u_{0,j}:=s^{-1}_j{\bf u}_0.$ Letting $j\to+\infty$, up to a subsequence, the metric measure spaces $C_{s_j}$ converge to a  limit space $C_\infty$ in the $pmGH$-topology, which is isometric to a product space $C_\infty=Z\times \mathbb R$ with the natural product metric and product measure (by the splitting theorem in  \cite{Gig13}), and the maps ${\bf u}_{0,j}$ converge to a limit map ${\bf u}_{00}$, which a global minimizer on $Z\times \mathbb R$ of  homogeneous of degree one (by Theorem \ref{thm7.1} and    Lemma \ref{lem8-10}), and 
  \begin{equation} \label{equ8-17}
   \mathscr H^{N- 3 +\eta} \left(\mathcal S^{\Omega_{{\bf u}_{00}}}_{\varepsilon}\cap B_{1}((z_0,0))\right) >0,
   \end{equation}
where $(z_0,0)$ is the limit of the points $x_1\in C_{s_j}$ as $j\to+\infty$. To continue the proof, we need the following lemma.
\begin{lemma}\label{lem8.13}
The map ${\bf u}_{00}|_{Z\times\{0\}}$ is a global minimizer of $J_{Q_0}$ on $Z$. 
\end{lemma}

\begin{proof}
We first claim that ${\bf u}_{00}(z,t)={\bf u}_{00}(z,0)$ for any $z\in Z$ and $t\in\mathbb R$. It can be intuitively observed  from homogeneity of ${\bf u}_0$ on each  $C_{s_j}$ and the converging of $C_{s_j}\overset{pmGH}{\longrightarrow} Z\times\mathbb R$. Here, for clarity, we include the details for the realization of this observation as follows.

Let  $\gamma:[0,+\infty)\to C(Y)$ is the ray  with $\gamma(0)=o_Y$ and $\gamma(L)=x_1$, where $L=d_C(x_1,o_Y)>0$. On each $C_{s_j}$, the curve $\gamma_j(t):=\gamma(s_j\cdot t+L )\in C_{s_j}$ is one of the shortest on every sub-interval $[a,b]\subset [-L/s_j,+\infty)$. 
We first consider the   functions $f_j$ on $C_{s_j}$ given by
\begin{equation*}
\begin{split}
f_j(x):&=  d_j(x_1,o_Y)-d_j(x,o_Y)\\
&=s^{-1}_j\big(d_C(x_1,o_Y)-d_C(x,o_Y)\big).
\end{split}
\end{equation*}
 Since ${\bf u}_0=(u_0^1, u^2_0, \cdots, u^m_0)$ is homogeneous of degree one on $C(Y)$, we have     for all $s_j$, that
 \begin{equation}\label{equ8-18}
\ip{\nabla { u}^\alpha_{0,j}}{\nabla f_j}(x)=\ip{\nabla (s^{-1}_ju^\alpha_0)}{\nabla d_j(x,o_Y)} =0\quad {\rm a.e.\ in}\  C_{s_j}, \end{equation} 
for all $\alpha=1,2,\cdots,m.$

Letting $s_j\to 0^+$,  up to a subsequence, the curves $\gamma_j:[-L/s_j,+\infty)\subset C_{s_j}$ converge to a line $\gamma_\infty $ on the limit space $C_\infty$. According to the splitting theorem in \cite{Gig13}, we know that $C_\infty$ splits isometrically to a product space $Z\times \mathbb R$. Moreover, letting $b$ be  the Busemann function with respect to the ray $\gamma_{\infty}|_{(-\infty,0]}$ on $C_\infty$, then $Z=b^{-1}(0)$ and that the gradient flow of $b$ exists, denoted by $\Phi_t$, and furthermore   $(z,t)=\Phi_t(z,0)$ for any $z\in Z$ and $t\in\mathbb R$.   On the other hand,
from the definition of $f_j(x) =  d_j(x_1,o_Y)-d_j(x,o_Y) =d_j\big(\gamma_j(0),\gamma_j(-L/s_j)\big)- d_j\big(x,\gamma_j(-L/s_j)\big)$ and the fact that  $f_j$ is 1-Lipschitz  on $C_{s_j}$, it is clear that, up to a subsequence, the functions $f_j$  converge uniformly on each compact set to the Busemann function $b(x)$. Notice that $|\nabla f_j|(x)=1$ a.e. $x\in C_{s_j}$ and $|\nabla b|=1$ a.e. on $C_\infty$. In particular, we get that $|\nabla f_{j}|\to |\nabla b|$ in the $L^2_{\rm loc}$ as $j\to\infty$. By combining this and \eqref{equ7.4}, we conclude that $\ip{\nabla { u}^\alpha_{0,j}}{\nabla f_j}\to \ip{\nabla { u}^\alpha_{00}}{\nabla b}$ in $L^1_{\rm loc}$, for all $\alpha=1, 2, \cdots, m$.  From  \eqref{equ8-18}, we get that 
 \begin{equation*}
\ip{\nabla { u}^\alpha_{00}}{\nabla b}=0\quad {\rm a.e.\ in } \ C_\infty= Z\times\mathbb R,\qquad \forall \alpha=1, 2, \cdots, m.
\end{equation*} 
This implies for almost all $z\in Z$, ${\bf u}_{00}\circ\Phi_t$ is a constant map, where $\Phi_t$ is the gradient flow of $b$. Recalling $\Phi_t(z,0)=(z,t)$ for all $z\in Z$, this gives 
${\bf u}_{00}(z,t)={\bf u}_{00}(z,0)$ for almost all $z\in Z$. Finally, from the fact that ${\bf u}_{00}$ is Lipschitz continuous, we know that ${\bf u}_{00}(z,t)={\bf u}_{00}(z,0)$ for all $z\in Z$.
This claim is proved.
 
 With the  help of the fact that ${\bf u}_{00}(z,t)={\bf u}_{00}(z,0)$ for any $z\in Z$ and $t\in\mathbb R$, we will use the argument in \cite{Wei99} to prove this Lemma  \ref{lem8.13}.  Suppose that ${\bf u}_{00}$ is not a minimizer of $J_{Q_0}$ on a ball $B\subset Z$. Then there exists a map ${\bf v}\in W^{1,2}(B,\mathbb [0,+\infty)^m)$ such that ${\bf v}-{\bf u}_{00}|_{Z\times\{0\}}\in W^{1,2}_0(B,\mathbb R^m)$ and 
$$\int_B(|\nabla {\bf v}|^2+Q_0\chi_{|{\bf v}|>0}){\rm d}\mu_Z\ls \int_B(|\nabla {\bf u}_{00}|^2+Q_0\chi_{|{\bf u}_{00}|>0}){\rm d}\mu_Z-\epsilon_0$$
for some $\epsilon_0>0.$ We define a map ${\bf v}_T$ on $B\times (-T,T)$, for any $T>1$,  by 
\begin{equation*}
{\bf v}_T(z,t):=
\begin{cases}
{\bf v}(z), & |t|\ls T-1\\
(T-|t|){\bf v}(z) +(|t|-T+1){\bf u}_{00}(z,0), & T-1\ls |t|\ls T\\
{\bf u}_{00}(z,0), & |t|\gs T.
\end{cases}
\end{equation*}
It is clear that ${\bf v}_T-{\bf u}_{00}\in W^{1,2}_0(B\times(-T,T),[0,+\infty)^m)$, by using the fact ${\bf u}_{00}(z,t)={\bf u}_{00}(z,0)$ for all $z\in Z$.  Note that 
\begin{equation*}
\begin{split}
&\int_{B\times(-T,T)}(|\nabla {\bf v}_T|^2+Q_0\chi_{|{\bf v}_T|>0}){\rm d}\mu_Z{\rm d}t-  \int_{B\times(-T,T)}(|\nabla {\bf u}_{00}|^2+Q_0\chi_{|{\bf u}_{00}|>0}){\rm d}\mu_Z{\rm d}t\\
\ls & 2(T-1)\ \Big(\int_B(|\nabla {\bf v}|^2+Q_0\chi_{|{\bf v}|>0}){\rm d}\mu_Z-\int_B(|\nabla {\bf u}_{00}|^2+Q_0\chi_{|{\bf u}_{00}|>0}){\rm d}\mu_Z\Big)\\
&\ \ + \int_{B\times \big((-T,-T+1)\cap (T-1,T)\big)} \Big(|\nabla {\bf v}_T|^2+Q_0\chi_{|{\bf v}_T|>0}) \Big){\rm d}\mu_Z{\rm d}t\\
\ls &  -2(T-1)\epsilon_0+ 4\int_B(|\nabla {\bf v}|^2+|{\bf v}|^2+|\nabla {\bf u}_{00}|^2+|{\bf u}_{00}|^2){\rm d}\mu_Z+2Q_0\mu_Z(B),
\end{split}
\end{equation*}
which contradicts the fact that ${\bf u}_{00}$ is a minimizer on $B\times(-T,T)$ when $T$ is large enough.  Therefore, the Lemma \ref{lem8.13}  is proved. 
\end{proof}

We now come back to the proof of Theorem \ref{thm8-12}. From the assumption $\partial X\cap\Omega=\emptyset$, we know that both $C(Y)$ and  $Z\times\mathbb R$ have no boundary. Thus, $Z$ has no boundary.  

If $N-1\gs  3 $, by the combination of  the above Lemma \ref{lem8.13},  (\ref{equ8-17}) and the fact  that ${\bf u}_{00}(z,t)={\bf u}_{00}(z,0)$ for all $z\in Z$ and $t\in\mathbb R$, we obtain
  that there exists an $(N-1)$-dimensional $ncRCD(0,N-1)$-space without boundary, $Z$, and a minimizer of $J_{Q_0}$ on $Z$, $\widetilde{\bf u}:={\bf u}_{00}|_{Z\times\{0\}}$,  such that 
 \begin{equation*} 
   \mathscr H^{N-4 +\eta} \left(\mathcal S^{\Omega_{\widetilde{\bf u}}}_{\varepsilon}\cap B_{1}(z_0)\subset Z\right) >0.
   \end{equation*} 

Iterating this procedure we conclude that  there exists a  $ 3 $-dimensional $ncRCD(0, 3)$-space without boundary, denoted by $\bar{X}$, and a minimizer of $J_{Q_0}$, denoted by $\bar{\bf u}$, on $\bar{X}$, such that 
 \begin{equation*} 
   \mathscr H^{\eta} \left(\mathcal S^{\Omega_{\bar{\bf u}}}_{\varepsilon} \cap B_1(\bar{x})\right) >0.
   \end{equation*} 
  We claim that the singular set of $\bar{\bf u}$ must contain only isolated points. Suppose that a sequence $\bar x_j\in S^{\Omega_{\bar u}}_\epsilon$ and $\bar x_j\to \bar x_0$. 
  Let $\bar s_j=d(\bar x_0,\bar x_j)\to 0^+$. We consider the blow up sequence of spaces $(\bar X, \bar s_j^{-1}d,\mu^{\bar x_0}_{\bar s_j},\bar x_0)$ and maps $\bar {\bf u}_{j}=\bar s_j^{-1}\bar{\bf u}$. Letting $j\to+\infty$, we get a blow-up limit map $\bar{\bf u}_0$ on a tangent cone $C(Y)$ at $\bar x_0$. From Lemma \ref{lem8-7}, we know that $\bar {\bf u}_0$ has at least two singular points $o_Y$ and $\bar y_\infty$, the limit of $\bar x_j$. By $\bar s^{-1}_jd(\bar x_0,\bar x_j)=1$, we have $\bar y_\infty\not=o_Y$. Now, we blow up again at $\bar y_\infty$ as above, from Lemma \ref{lem8.13}, we get a minimizer of $J_{Q_0}$, $\bar {\bf u}_{00}|_{\bar Z}$, on some two-dimensional $ncRCD(0,2)$-space $\bar Z$. Moreover, it has at least one singular point at $\bar z_0$, where the point $(\bar z_0,0)$ is the limit of points $\bar y_\infty$ under this blow-up procedure.  This contradicts with Lemma \ref{lem8-11}. The claim is proved, and hence the proof of   Theorem \ref{thm8-12} is finished. 
       \end{proof}

\begin{proof}[Proof of Theorem \ref{thm1.9}]
 It follows from the combination of Corollary \ref{cor8-9} (by putting $O_{\varepsilon}:=\mathcal R^{\Omega_{\bf u}}_\delta$ in Corollary \ref{cor8-9}), Theorem \ref{thm8-12} and  Lemma \ref{lem8-11}.
\end{proof}
 
 \begin{proof}[Proof of Corollary \ref{cor1-11}]
 It follows from Theorem \ref{thm8-12}  and   $\mathcal S^{\Omega_{\bf u}}=\cup_{\varepsilon>0} \mathcal S^{\Omega_{\bf u}}_{\varepsilon}$. 
    \end{proof}

\appendix
\section{Weiss-type monotonicity on cones \label{app}}

A Weiss-type monotonicity for minimizers ${\bf u} $ of $J_Q$ defined on $\mathbb R^N$ has been obtained in \cite{CSY18} and \cite{MTV17}. The same argument can be extended to the case where ${\bf u}$ is defined on a metric measure cone. We will provide the details as follows.

Let $N\gs 2$ and let   $(C(\Sigma),d_C,\mu_C)$ be a metric measure cone, with the vertex $p$, over $(\Sigma,d_\Sigma,\mu_\Sigma)$, and assume that  $(C(\Sigma),d_C,\mu_C)$   satisfies $RCD(0,N)$.

    \begin{lemma}
For $\mathscr L^1$-a.e. $r\in \mathbb R^+$, it holds
\begin{equation}\label{equa2}
\int_{C(\Sigma)}g{\rm d}|D\chi_{B_r(p)}|=r^{N-1} \int_{\Sigma}g_r{\rm d}\mu_\Sigma,
\end{equation} 
for any Borel function $g(t,\xi)$ on $C(\Sigma)$, where $g_r(\xi):=g(r,\xi)$.
 \end{lemma}
\begin{proof}
For $\mathscr L^1$-a.e. $r\in \mathbb R^+$, the set $B_r(p)$ has finite perimeter and that, by coarea formula, 
$$\int_{s}^t\int_{C(\Sigma)}g{\rm d}|D\chi_{B_r(p)}|{\rm d}r=\int_{B_t(p)\backslash B_s(p)}g{\rm d}\mu_C=\int_{[s,t)\times\Sigma}gr^{N-1}{\rm d}r{\rm d}\mu_\Sigma$$
for all $0\ls s<t<\infty$ and all Borel function $g$, where we have used $d\mu_{C}(t, \xi)=t^{N-1}dt\times d\mu_{\Sigma}(x)$. It follows that  the function $t\mapsto \int_0^t\int_{C(\Sigma)}g{\rm d}|D\chi_{B_r(p)}|$ is absolutely continuous, and then the desired assertion (\ref{equa2}) holds.   \end{proof}

Let $Q_0>0$ be a constant and let $\mathbf{u}=\left(u_{1},\ldots,u_{m}\right)$ be a global minimizer of $J_{Q_0}$  on $(C(\Sigma),d_C,\mu_C)$. I.e., for each $R>0$, ${\bf u}$ is minimizer of $J_{Q_0}$ on $B_R(p).$ For any $r\in(0,+\infty)$, we denote by ${\bf u}_r(\xi)={\bf u}(r,\xi), \forall \xi\in \Sigma.$

 \begin{lemma}
For each $r\in(0,\infty)$, we have
\begin{equation}\label{equa3}
\begin{split}
\frac{N}{r^{N-2}}\int_{B_r(p)}\left(|\nabla{\bf u}|^2+Q_0\chi_{\{|{\bf u}|>0\}}\right)\ls &  \int_{\Sigma}\Big(|\nabla_\Sigma{\bf u}_r|^2 +|{\bf u}_r|^2\Big){\rm d}\mu_\Sigma\\
&\ +r^{2}\int_{\Sigma}\Big(Q_0\chi_{\{| {\bf u}_r|>0\}}\Big){\rm d}\mu_\Sigma,
\end{split}
\end{equation}
where $\nabla_\Sigma v$ is the weak upper gradient of $v\in W^{1,2}(\Sigma).$
 \end{lemma}
\begin{proof}
Fix any $r\in(0,\infty)$. We set the function ${\bf v}:=(v_1,\cdots, v_m): B_r(p)\to\mathbb R^m$ by
$$v_i(t,\xi):=\frac{t}{r}u_{r,i}(\xi)=\frac{t}{r}u_i(r,\xi),\quad\ \forall\ t\in(0,r),\ \xi\in\Sigma,\quad\forall\ i\in\{1,2,\cdots, m\}.$$
Then we first have ${\bf v}\in Lip(B_r(p))$ and ${\bf v}={\bf u}$ on $\partial B_r(p)$. By the minimizer of ${\bf u}$, we have
\begin{equation}\label{equa4}
\int_{B_r(p)}\left(|\nabla{\bf u}|^2+Q_0\chi_{\{|{\bf u}|>0\}}\right)\ls\int_{B_r(p)}\left(|\nabla{\bf v}|^2+Q_0\chi_{\{|{\bf  v}|>0\}}\right).
\end{equation}
 Following Proposition 3.4 of \cite{Ket15},
we know  that for any $v\in W^{1,2}(B_r(p))$, it holds, for almost all $(t,x)\in C(\Sigma)$, 
$$|\nabla v|^2(t,x)=|\nabla_{\mathbb R^+} v_\xi|^2(t)+t^{-2}|\nabla_\Sigma v_t|^2(\xi),$$
where $v_t(\cdot):=v(t,\cdot)$ and $v_\xi(\cdot):=v(\cdot,\xi)$. By applying this to each component of ${\bf v}$, we get
\begin{equation}\label{equa5}
\begin{split}
\int_{B_r(p)}|\nabla {\bf v}|^2&=\int_{B_r(p)}\Big(\frac{1}{t^2}|\nabla_\Sigma{\bf v}_t|^2(\xi)+|\nabla_t{\bf v}_\xi|^2(t)\Big){\rm d}\mu_C\\
&=\int_{0}^r\int_{\Sigma}\Big(\frac{1}{r^2}|\nabla_\Sigma{\bf u}_r|^2(\xi)+\frac{1}{r^2}|{\bf u}_r|^2(\xi)\Big)t^{N-1}{\rm d}t{\rm d}\mu_\Sigma\\
&=\frac{r^{N-2}}{N}\int_{\Sigma}\Big(|\nabla_\Sigma{\bf u}_r|^2(\xi)+|{\bf u}_r|^2(\xi)\Big){\rm d}\mu_\Sigma.
\end{split}
\end{equation}
Noticing that $|{\bf v}|(t,\xi)>0\Longleftrightarrow |{\bf u}_r|(\xi)>0,$  we have
\begin{equation}\label{equa6}
\begin{split}
\int_{B_r(p)}Q_0\chi_{\{| {\bf v}|>0\}}&=\int_{0}^r\int_{\Sigma}\Big(Q_0\chi_{\{| {\bf v}|>0\}}\Big)t^{N-1}{\rm d}t{\rm d}\mu_\Sigma\\
&=\frac{r^{N}}{N}\int_{\Sigma}\Big(Q_0\chi_{\{| {\bf u}_r|>0\}}\Big){\rm d}\mu_\Sigma.
\end{split}
\end{equation}
Now the desired estimate (\ref{equa3}) follows from the combination of (\ref{equa4})-(\ref{equa6}), and the proof is finished.
\end{proof} 
 
Now we give the   monotonicity of $W_{{\bf u}}(p,r,Q_0)$.

 \begin{lemma}\label{lem-a3}
 Suppose the cone $C(\Sigma)$ is non-collapsed. Then the function $r\mapsto W_{{\bf u}}(p,r,Q_0) $ is non-decreasing. Moreover, if  $W_{{\bf u}}(p,r,Q_0) $ is a constant then ${\bf u}$ is homogeneous of degree one.  
 \end{lemma}
\begin{proof}
Since ${\bf u}$ is locally Lipschitz continuous in $C(\Sigma)$, it is clear that $W_{{\bf u}}(p,r,Q_0) $ is locally Lipschitz continuous in $(0,\infty)$, and then it is differentiable at $\mathscr L^1$-a.e. $r\in(0,\infty)$. At a such $r$, we have
\begin{equation}\label{equa7}
  \begin{aligned}
    \frac{\mathrm{d}}{\mathrm{d}r} W_{{\bf u}}(p,r,Q_0)    =             & -\frac{N}{r^{N+1}}\int_{B_{r}(p)}\left(|\nabla \mathbf{u}|^{2}+Q_0\chi_{ \{ |\mathbf{u}|>0 \}}\right) \du                                                                                                                 \\
                                                      & +\frac{1}{r^{N}}\int_{C(\Sigma)}\left(|\nabla \mathbf{u}|^{2}+Q_0\chi_{ \{ \mathbf{u} |>0\}}\right) {\rm d}|D\chi_{B_r(p)}|                                                                                                                                 \\
                                                      & +\frac{2}{r^{3}}\int_{\Sigma}|\mathbf{u}_r|^{2}{\rm d}\mu_\Sigma-\frac{1}{r^{2}}    \frac{\mathrm{d}}{\mathrm{d}r} \int_{\Sigma}  |{\bf u}_r|^2{\rm d}\mu_\Sigma                                                                                                            \\
\gs& \frac 1 r\int_{\Sigma}|\nabla_{\mathbb R^+}{\bf u}_\xi|^2(r)\du_{\Sigma}+\frac{1}{r^3}\int_{\Sigma}|{\bf u}_r|^2\du_\Sigma   -  \frac{1}{r^{2}}   \frac{\mathrm{d}}{\mathrm{d}r} \int_{\Sigma}  |{\bf u}_r|^2{\rm d}\mu_\Sigma,                                                                                                         \\                                                                                                                                                                         \end{aligned}
\end{equation}
where we have used (\ref{equa3}), (\ref{equa2}) and 
$$|\nabla {\bf u}|^2(r,x)=|\nabla_{\mathbb R^+} {\bf u}_\xi|^2(r)+r^{-2}|\nabla_\Sigma {\bf u}_r|^2(\xi) $$
(by Proposition 3.4 of \cite{Ket15}). Since $C(\Sigma)$ is assumed to be non-collapsed, we know that ${\bf u}$ is locally Lipschitz on $C(\Sigma)$. Thus, we get
$$\frac{\mathrm{d}}{\mathrm{d}r} \int_{\Sigma}  |{\bf u}_r|^2{\rm d}\mu_\Sigma=\sum_{i=1}^m\frac{{\rm d}}{{\rm d}r}\int_{\Sigma}u^2_i(r,\xi)\du_{\Sigma}=\sum_{i=1}^m\int_{\Sigma}2u_i(r,\xi)\frac{\partial u_i}{\partial r}(r,\xi)\du_{\Sigma}.$$
Putting this into (\ref{equa7}), we get 
\begin{equation}\label{equa8}
  \begin{aligned}
  r\cdot  \frac{\mathrm{d}}{\mathrm{d}r} W_{{\bf u}}(p,r,Q_0)  &\gs\sum_{i=1}^m\int_{\Sigma}\bigg(|\nabla_{\mathbb R^+} u_{i,\xi}|^2+\frac{u^2_{i,r}}{r^2}-2\frac{u_{i,r}}{r}|\nabla_{\mathbb R^+}u_{i,\xi}|\bigg)\du_\Sigma\\
  &=\sum_{i=1}^m\int_{\Sigma}\bigg(|\nabla_{\mathbb R^+} u_{i,\xi}|- \frac{u_{i,r}}{r}\bigg)^2\du_\Sigma\gs 0,
    \end{aligned}
    \end{equation}
 where $u_{i,\xi}(\cdot):=  u_i(\cdot, \xi)$ and $u_{i,r}(\cdot):=u_i(r,\cdot).$ It follows that $W_{{\bf u}}(p,r,Q_0)$ is non-decreasing. Moreover, if $W_{{\bf u}}(p,r,Q_0)$ is a constant, then one have
 $$\frac{\partial u_i}{\partial r}(r,\xi)=\frac{u_{i,r}(\xi)}{r}=\frac{u_i(r,\xi)}{r} $$
 for almost all $(r,\xi)$ in $C(\Sigma)$. This implies for almost $\xi\in \Sigma$ that $u_{i,\xi}(r)=ru_{i,\xi}(1).$ Therefore, in this case ${\bf u}$ is homogeneous of degree one. The proof is finished. 
 \end{proof}

%$\ $\\[3pt]

% \noindent{\bf Data Availability Statement}

% Data sharing not applicable to this article as no datasets were generated or analyzed during the current study. \\[3pt]
 
% \noindent{\bf  Conflict of interest}
% The authors have no competing interests to declare that are relevant to the content of this article.\\[3pt]

  \end{document}